\documentclass[]{amsart}
\usepackage{amssymb}

 \newtheorem{Theorem}{Theorem}[section]
 \newtheorem{Corollary}[Theorem]{Corollary}
 \newtheorem{Lemma}[Theorem]{Lemma}
 \newtheorem{Proposition}[Theorem]{Proposition}

 \newtheorem{Definition}[Theorem]{Definition}

 \newtheorem{Remark}[Theorem]{Remark}

 \numberwithin{equation}{section}

\begin{document}

\title
 {minimal $L^2$ integrals for the Hardy spaces and the Bergman spaces}

\author{Qi'an Guan}
\address{Qi'an Guan: School of Mathematical Sciences,
Peking University, Beijing, 100871, China.}
\email{guanqian@math.pku.edu.cn}

\author{Zheng Yuan}

\address{Zheng Yuan: School of Mathematical Sciences,
Peking University, Beijing, 100871, China.}
\email{zyuan@pku.edu.cn}

\thanks{}



\date{\today}

\dedicatory{}

\commby{}


\begin{abstract}
In this article, we consider the minimal $L^2$ integrals for the Hardy spaces and the Bergman spaces, and we present some relations between them, which  can be regarded as the solutions  of the finite points versions of Saitoh's conjecture for conjugate Hardy kernels. As applications, we give  optimal $L^2$ extension theorems for the Hardy spaces, and characterizations for the holding of the equality in the optimal $L^2$ extension theorems.
\end{abstract}

\maketitle

\section{Introduction}

Let $D$ be a planar regular region with finite boundary components, which are analytic Jordan curves (see \cite{saitoh,yamada}).

\begin{Definition}[see \cite{saitoh,GY-weightedsaitoh}]
	We call a holomorphic function $f$ on $D$ belongs to Hardy space $H^2(D)$, if $|f(z)|^2$ have harmonic majorants $U(z)$, i.e., $|f(z)|^2\le U(z)\,\, \text{on}\,\,D.$
\end{Definition}

Each function $f(z)\in H^{2}(D)$ has Fatou's nontangential boundary value a.e. on $\partial D$ belonging to $L^2(\partial D)$ (see \cite{duren}), and we also denote the nontangential boundary value by $f$ for simplicity. The conjugate Hardy $H^2$ kernel $\hat K_t(z,\overline w)$ is defined as follow: 
$$f(w)=\frac{1}{2\pi}\int_{\partial D}f(z)\overline{\hat K_t(z,\overline w)}\left(\frac{\partial G_D(z,t)}{\partial v_z}\right)^{-1}|dz|$$
holds for any $f\in H^{2}(D)$, where  $G_D(z,t)$ is the Green function on $D$, and $\partial/\partial v_z$ denotes the derivative along the outer normal unit vector $v_z$. Fixed $t\in D$, $\frac{\partial G_D(z,t)}{\partial v_z}$ is  positive and continuous on $\partial D$ because of the analyticity of the boundary (see \cite{saitoh}, \cite{guan-19saitoh}).  When $t=w=z$, $\hat K(z)$ denotes $\hat K_t(z,\overline w)$ for simplicity.

 In \cite{guan-19saitoh}, Guan proved the following theorem, which was conjectured by Saitoh (see \cite{saitoh}):

\begin{Theorem}
	[\cite{guan-19saitoh}]\label{thm:saitoh}
	If $D$ is not simple connected, then $\hat K(z)>\pi B(z)$, where $B(z)$ is the Bergman kernel on $D$.
\end{Theorem}

By discussing the weighted kernel functions, we \cite{GY-weightedsaitoh} gave a weighted version of Saitoh's conjecture and a weighted version of Saitoh's conjecture for higher derivatives. 

In \cite{GY-saitohprodct}, we considered two classes of weighted Hardy spaces on products of planar domains.  Let us recall their definitions.

  Let $D_j$ be a planar region bounded by finite analytic Jordan curves   for any $1\le j\le n$. Let $M=\prod_{1\le j\le n}D_j$ be a bounded domain in $\mathbb{C}^n$. Let  $M_j=\prod_{1\le l\le n,l\not=j}D_l$, then $M=D_j\times M_j$ and $\partial M=\cup_{1\le j\le n}\partial D_j\times\overline{M_j}$. 

Let $\rho$ be a Lebesgue measurable function on $\partial M$ such that $\inf_{\partial M}\rho>0$. Now, we recall the Hardy space $H_{\rho}^2(M,\partial M)$ (see \cite{GY-saitohprodct}).
 Note that $\partial M=\cup_{j=1}^n\partial D_j\times \overline{M_j}$.  
Let $d\mu_j$ be the Lebesgue measure on $M_j$ for any $1\le j\le n$ and $d\mu$ is a measure on $\partial M$ defined by
$$\int_{\partial M}hd\mu=\sum_{1\le j\le n}\frac{1}{2\pi}\int_{M_j}\int_{\partial D_j}h(w_j,\hat w_j)|dw_j|d\mu_j(\hat w_j)$$
 for any $h\in L^1(\partial M)$, where $\hat w_j:=(w_1,\ldots,w_{j-1},w_{j+1},\ldots,w_n)\in M_j$.  For any $f\in H^2(D_j)$, $\gamma_j(f)$ denotes the nontangential boundary value of $f$ a.e. on $\partial D_j$.

 \begin{Definition}[\cite{GY-saitohprodct}]
 	 Let $f\in L^2(\partial M,\rho d\mu)$. We call $f\in H^2_{\rho}(M,\partial M)$ if there exists $f^*\in\mathcal{O}(M)$ such that for any $1\le j\le n$, $f^*(\cdot,\hat w_j)\in H^2(D_j)$  for any  $\hat w_j\in M_j$ and $f=\gamma_j(f^*)$ a.e. on $\partial D_j\times M_j$. 
 \end{Definition}
  $H^2_{\rho}(M,\partial M)$ is a Hilbert space (see \cite{GY-saitohprodct}) equipped with the norm $\ll\cdot,\cdot\gg_{\partial M,\rho}$, which is defined by
$$\ll f,g\gg_{\partial M,\rho}:=\int_{\partial M}f\overline{g}\rho d\mu.$$
 Denote that $P_{\partial M}(f)=f^*$ for any $f\in H^2_{\rho}(M,\partial M)$.
  $P_{\partial M}$ is a linear injective map from $H^2(M,\partial D_j\times M_j)$ to $\mathcal{O}(M)$ (see \cite{GY-saitohprodct}). When $n=1$, $P_{\partial M}=\gamma_1^{-1}$, thus $H^2_{\rho}(M,\partial M)$ can be seen as a weighted generalization on product spaces of $H^2(D)$.

Denote that $S:=\prod_{1\le j\le n}\partial D_j$.  Let $\lambda$ be a Lebesgue measurable function on $S$ such that $\inf_{S}\lambda>0$. Let us recall another class of Hardy space $H^2_{\lambda}(M,S)$.
\begin{Definition}[\cite{GY-saitohprodct}]\label{def2}
	Let $f\in L^2(S,\lambda d\sigma)$, where $d\sigma:=\frac{1}{(2\pi)^n}|dw_1|\ldots|dw_n|$. We call $f\in H^2_{\lambda}(M,S)$ if there exists $\{f_m\}_{m\in\mathbb{Z}_{\ge0}}\subset\mathcal{O}(M)\cap C(\overline M)\cap L^2(S,\lambda d\sigma)$ such that $\lim_{m\rightarrow+\infty}\|f_m-f\|_{S,\lambda}^2=0$, where $\|g\|_{S,\lambda}:=\left(\int_{S}|g|^2\lambda d\sigma\right)^{\frac{1}{2}}$ for any $g\in L^2(S,\lambda d\sigma)$.
\end{Definition}
 Denote that 
$$\ll f,g\gg_{S,\lambda}=\frac{1}{(2\pi)^n}\int_S f\overline g\lambda |dw_1|\ldots|dw_n|$$
for any $f,g\in L^2(S,\lambda d\sigma),$
then $H_{\lambda}^2(M,S)$ is a Hilbert space equipped with the inner product $\ll \cdot,\cdot\gg_{S,\lambda}$ (see \cite{GY-saitohprodct}). There exists a  linear injective map $P_S:H^2_{\lambda}(M,S)\rightarrow\mathcal{O}(M)$ satisfying that $P_S(f)=f$ for any $f\in\mathcal{O}(M)\cap C(\overline M)\cap L^2(S,\lambda d\sigma)$ (see \cite{GY-saitohprodct}). When $n=1$, $P_{S}=\gamma_1^{-1}$, thus $H^2_{\lambda}(M,S)$ can also be seen as a weighted generalization on product spaces of $H^2(D)$.

In \cite{GY-saitohprodct}, we discussed some properties and kernel functions for the spaces $H^2_{\rho}(M,\partial M)$ and $H^2_{\lambda}(M,S)$, and we discussed the relations between them and the weighted Bergman kernels on $M$, which can be regarded as the solutions of the product versions of Saitoh's conjecture.

Note that the above mentioned kernel functions for the Hardy spaces and the Bergman spaces can be seen as the reciprocal of some minimal $L^2$ integrals related to one point, such as:
$$\hat K(z)=\frac{1}{\inf\left\{\frac{1}{2\pi}\int_{\partial D}|f(z)|^2\left(\frac{\partial G_D(z,t)}{\partial v_z}\right)^{-1}|dz|:f\in H^2(D)\,\&\,f(z)=1\right\}},$$
and
$$B(z)=\frac{1}{\inf\left\{\int_D|f|^2:f\in\mathcal{O}(D)\,\&\,f(z)=1 \right\}}.$$
In this article, we consider more general minimal $L^2$ integrals for the Hardy spaces and the Bergman spaces, and we give some relations between them. As applications, we give  optimal $L^2$ extension theorems for the Hardy spaces, and characterizations for the holding of equality in the optimal $L^2$ extension theorems.

\subsection{Minimal $L^2$ integrals on a planar region}\label{sec:main}

Let $D$ be a planar region bounded by finite analytic Jordan curves, and let $Z_0:=\{z_1,\ldots,z_m\}\subset D$, where $m$ is a positive integer. 

Let $\psi$ be a Lebesgue measurable function on  $\overline D$, which satisfies that $\psi$ is  subharmonic on $D$, $\psi\equiv 0$ on $\partial D$
 and the Lelong number $v(dd^c\psi,z_j)>0$ for any $z_j\in Z_0$, where $d^c=\frac{\partial-\bar\partial}{2\pi\sqrt{-1}}$. Assume that $\psi\in C^1(U\cap\overline{D})$ and $\frac{\partial\psi}{\partial v_z}$ is positive on $\partial D$, where  $U$ is an open neighborhood of $\partial D$ and  $\partial/\partial v_z$ denotes the derivative along the outer normal unit vector $v_z$.
 
Let $k_j$ be a nonnegative integer for $1\le j \le m$. Let $\varphi$ be a Lebesgue measurable function on $\overline D$ satisfying that $\varphi+2\psi$ is subharmonic on $D$, the Lelong number 
 $$v(dd^c(\varphi+2\psi),z_j)\ge2(k_j+1)$$
 for any $1\le j\le n$, and $\varphi$ is continuous at $z$ for any $z\in\partial D$. Besides, we assume that  one of the following two statements holds:
 
  $(1)$ $(\psi-p_{j}G_{D}(\cdot,z_j))(z_j)>-\infty$, where $p_j=v(dd^c(\psi),z_j)>0$ for any $1\le j\le m$;
  
  $(2)$ for any $1\le j\le m$, there exists $a_j\in[0,1)$ such that $\varphi+2a_j\psi$ is subharmonic near $z_j$.

 Let $c$ be a positive Lebesgue measurable function on $[0,+\infty)$ satisfying that $c(t)e^{-t}$ is decreasing on $[0,+\infty)$, $\lim_{t\rightarrow0+0}c(t)=c(0)=1$ and $\int_0^{+\infty}c(t)e^{-t}dt<+\infty$. Denote that  
$$\tilde\rho:=e^{-\varphi}c(-2\psi),$$ 
and assume that $\tilde\rho$ has a positive lower bound on any compact subset of $D\backslash Z$, where $Z\subset\{\psi=-\infty\}$ is a discrete subset of $D$. Denote that $$\rho:=e^{-\varphi}\left(\frac{\partial\psi}{\partial v_z}\right)^{-1}$$
on $\partial D$.

Let us consider the following two minimal integrals. Let $\mathfrak{a}=(a_{j,l})$ $(1\le j\le m,0\le l\leq k_j )$, where $a_{j,l}\in\mathbb{C}$ such that $\sum_{1\le j\le m}\sum_{0\le l\le k_j}|a_{j,l}|\not=0$.
Denote that 
\begin{displaymath}
	\begin{split}
		M(Z_0,\mathfrak{a},\tilde\rho):=\inf\bigg\{&\int_{ D}|f|^2\tilde\rho:f\in \mathcal{O}(D) 
		\\ &\text{ s.t. } f^{(l)}(z_j)=l!a_{j,l}\text{ for any $0\le l\le k_j$ and any $1\le j\le m$}\bigg\}.	\end{split}
\end{displaymath}
and 
\begin{displaymath}
	\begin{split}
		M_H(Z_0,\mathfrak{a},\rho):=\inf\bigg\{\frac{1}{2\pi}&\int_{\partial D}|f|^2\rho|dz|:f\in H^2(D) 
		\\ &\text{ s.t. } f^{(l)}(z_j)=l!a_{j,l}\text{ for any $0\le l\le k_j$ and any $1\le j\le m$}\bigg\}.	\end{split}
\end{displaymath}

We recall some notations (see \cite{OF81}, see also \cite{guan-zhou13ap,GY-concavity,GMY-concavity2}).
  Let $p:\Delta\rightarrow D$ be the universal covering from unit disc $\Delta$ to $D$.
 we call the holomorphic function $f$  on $\Delta$ a multiplicative function,
 if there is a character $\chi$, which is the representation of the fundamental group of $D$, such that $g^{\star}f=\chi(g)f$,
 where $|\chi|=1$ and $g$ is an element of the fundamental group of $D$. It is known that for any harmonic function $u$ on $D$,
there exists a $\chi_{u}$ and a multiplicative function $f_u\in\mathcal{O}^{\chi_u}(D)$,
such that $|f_u|=p^{\star}\left(e^{u}\right)$.
Recall that for the Green function $G_{D}(z,z_j)$,
there exist a $\chi_{z_j}$ and a multiplicative function $f_{z_j}\in\mathcal{O}^{\chi_{z_j}}(D)$, such that $|f_{z_j}(z)|=p^{\star}\left(e^{G_{D}(z,z_j)}\right)$ (see \cite{yamada,suita}).

We present a relation between $M_H(Z_0,\mathfrak{a},\rho)$ and $M(Z_0,\mathfrak{a},\tilde\rho)$ as follows:

\begin{Theorem}
\label{main theorem1}Assume that $M(Z_0,\mathfrak{a},\tilde\rho)<+\infty$. Then
\begin{equation}
	\label{eq:221201b}M_H(Z_0,\mathfrak{a},\rho)\leq\frac{M(Z_0,\mathfrak{a},\tilde\rho)}{\pi\int_0^{+\infty}c(t)e^{-t}dt} 
	\end{equation}
 holds, and the equality holds if and only if the following statements hold:

	$(1)$ $\varphi+2\psi=2\sum_{1\le j\le m}(k_j+1)G_{D}(\cdot,z_j)+2u$, where $u$ is a harmonic function on $D$;
	 
	$(2)$ $\psi=\sum_{1\le j\le m}p_jG_{D}(\cdot,z_j)$, where $p_j=v(dd^c(\psi),z_j)>0$;
	
	$(3)$ $\chi_{-u}=\prod_{1\le j\le m}\chi_{z_j}^{k+1}$, where $\chi_{-u}$ and $\chi_{z_j}$ are the  characters associated to the functions $-u$ and $G_{D}(\cdot,z_j)$ respectively;
	
	$(4)$  For any $1\le j\le m$, $\lim_{z\rightarrow z_j}\frac{p_*\left(f_u\left(\prod_{1\le j\le m}f_{z_j}^{k_j+1}\right)\left(\sum_{1\le j\le m}p_j\frac{df_{z_j}}{f_{z_j}}\right)\right)}{(z-z_j)^{k_j}dz}=c_0a_{j,k_j}$ and $a_{j,l}=0$ for any $l<k_j$, where $c_0\not=0$ is a constant independent of $j$. 
	\end{Theorem} 

When $m=1$, Theorem \ref{main theorem1} is a solution of the weighted version of Saitoh's conjecture for higher derivatives, which can be referred to \cite{GY-weightedsaitoh}.
	
\begin{Remark}
	\label{rem1.1} Assume that the four statements in Theorem \ref{main theorem1} hold, then we know $\frac{p_*\left(f_{ u}\left(\prod_{1\le j\le m}f^{k_j+1}_{z_j}\right)\left(\sum_{1\le j\le m}p_j\frac{df_{z_j}}{f_{z_j}}\right)\right)}{c_0dz}$ is a (single-valued) holomorphic function on $D$, and we denote it by $F_0$. Then $F_0^{(l)}(z_j)=l!a_{j,l}$ for any $0\le l\le k_j$ and any $1\le j\le m$, and there exists $f_0\in H^2(D)$ such that $f_0^*=F_0$,
	$$M(Z_0,\mathfrak{a},\tilde\rho)=\int_{ D}|F_0|^2\tilde\rho\,\,\,\,and\,\,\,\,M_H(Z_0,\mathfrak{a},\rho)=\frac{1}{2\pi}\int_{\partial D}|f_0|^2\rho|dz|.$$
	We prove the remark in Section \ref{sec:proof1}.
\end{Remark}

Let $Z_0:=\{z_j:1\le k\le m\}$ be a subset of $D$. Let $\lambda$ be a positive continuous function on $\partial D$. By solving Dirichlet problem, there exists a positive continuous function on $\overline D$ denoted also by $\lambda$, such that $\log\lambda$ is harmonic on $D$. Let $c_{\beta}(z)$ be the logarithmic capacity (see \cite{S-O69}) on $D$, which is locally defined by
$$c_{\beta}(z):=\exp\lim_{\tilde z\rightarrow z}(G_{D}(\tilde z,z)-\log|\tilde z-z|).$$

Using Theorem \ref{main theorem1}, we present the following optimal $L^2$ extension theorem for the Hardy space, and give a characterization for the holding of the equality in this extension theorem.

\begin{Corollary}
	\label{corollary1} Let $k_j$ be a nonnegative integer, and let $a_j\in\mathbb{C}$ for any $j$. Assume that $\sum_{1\le j\le m}\frac{2|a_j|^2t_j}{(k_j+1)c_{\beta}(z_j)^{2(k_j+1)}}\lambda(z_j)\in(0,+\infty)$. Then there exists $f\in H^2(D)$ such that $f^{(l)}(z_j)=0$ for $0\le l<k_j$ and $f^{(k_j)}(z_j)=k_j!a_j$ for any $1\le j\le m$, and
	$$\frac{1}{2\pi}\int_{\partial D}|f|^2\lambda\left(\frac{\partial\psi}{\partial v_z}\right)^{-1}|dz|\le\sum_{1\le j\le m}\frac{2|a_j|^2t_j}{(k_j+1)c_{\beta}(z_j)^{2(k_j+1)}}\lambda(z_j),$$
	where $\psi:=\sum_{1\le j\le m}(k_j+1)G_D(\cdot,z_j)$ and $t_j:=e^{-2\sum_{1\le j_1\le m,j_1\not=j}(k_j+1)G_D(z_j,z_{j_1})}$.
	
	Moreover, denote that $M_H:=\inf\{\frac{1}{2\pi}\int_{\partial D}|f|^2\lambda\left(\frac{\partial\psi}{\partial v_z}\right)^{-1}|dz|:f\in H^2(D)$ such that $f^{(l)}(z_j)=0$ for $0\le l<k_j$ and $f^{(k_j)}(z_j)=k_j!a_j$ for any $1\le j\le m\}$, then equality 
	$$M_H=\sum_{1\le j\le m}\frac{2|a_j|^2t_j}{(k_j+1)c_{\beta}(z_j)^{2(k_j+1)}}\rho(z_j)$$
	 holds if and only if the following statements hold:
	
	$(1)$ $\chi_{\frac{1}{2}\log\lambda}=\prod_{1\le j\le m}\chi_{z_j}^{k_j+1}$;
	
	$(2)$ For any $1\le j\le m$, $$\lim_{z\rightarrow z_j}\frac{p_*\left(f_{-\frac{1}{2}\log\lambda}\left(\prod_{1\le j\le m}f_{z_j}^{k_j+1}\right)\left(\sum_{1\le j\le m}(k_j+1)\frac{df_{z_j}}{f_{z_j}}\right)\right)}{(z-z_j)^{k_j}dz}=c_0a_{j},$$ where $c_0\not=0$ is a constant independent of $j$.
\end{Corollary}

Corollary \ref{corollary1} implies the following result. 

\begin{Corollary}
	\label{c:1.2}
Let $k$ be a nonnegative integer. Then there is a constant $C$ (depending on $k$), such that for any $a_{j,l}\in\mathbb{C}$, where  $1\le j\le m$ and $0\le l\le k$, there exists $f\in H^2(D)$ such that $f^{(l)}(z_j)=a_{j,l}$ for any $1\le j\le m$ and $0\le l\le k$, and
	$$\frac{1}{2\pi}\int_{\partial D}|f|^2|dz|\le C\sum_{1\le j\le m}\sum_{0\le l\le k}|a_{j,l}|^2.$$
\end{Corollary}

\subsection{Minimal $L^2$ integrals for the Hardy space $H_{\rho}^2(M,\partial M)$}

Let $D_j$ be a planar region bounded by finite analytic Jordan curves   for any $1\le j\le n$. Let $M=\prod_{1\le j\le n}D_j$ be a bounded domain in $\mathbb{C}^n$.

Let $Z_j=\{z_{j,1},z_{j,2},...,z_{j,m_j}\}\subset D_j$ for any  $j\in\{1,2,...,n\}$, where $m_j$ is a positive integer. Denote that 
$$Z_0:=\prod_{1\le j\le n}Z_j\subset M.$$
Let $\psi=\max_{1\le j\le n}\{\sum_{1\le k\le m_j}p_{j,k}G_{D_j}(\cdot,z_{j,k})\}$ on $M$.
Let $V_{z_{j,k}}\Subset D_{j}$ be a  neighborhood of $z_{j,k}$ satisfying $V_{z_{j,k}}\cap V_{z_{j,k'}}=\emptyset$ for any $j$ and $k\not=k'$. Denote that $I_1:=\{(\beta_1,\beta_2,...,\beta_n):1\le \beta_j\le m_j$ for any $j\in\{1,2,...,n\}\}$, $V_{\beta}:=\prod_{1\le j\le n}V_{z_{j,\beta_j}}$ and $z_{\beta}:=(z_{1,\beta_1},z_{2,\beta_2},\ldots,z_{n,\beta_n})\in M$ for any $\beta=(\beta_1,\beta_2,...,\beta_n)\in I_1$.

Let $\varphi_j$ be a subharmonic function on $D_j$, which satisfies that  $\varphi_j$ is continuous at $z$ for any $z\in \partial D_j$. Denote that 
$$\varphi(w_1,\ldots,w_n):=\sum_{1\le j\le n}\varphi_j(w_j)$$
on $M$. Let $f_0$ be a holomorphic function $\cup_{\beta\in I_1}V_{\beta}$. For any $\beta\in I_1$, let $J_{\beta}$ be an ideal of $\mathcal{O}_{z_{\beta}}$ satisfying $\mathcal{I}(\varphi+2\psi)_{z_{\beta}}\subset J_{\beta}$. Note that for any $\tilde z\in D_j$, $\frac{\partial G_{D_j}(z,\tilde z)}{\partial v_z}$ is a positive continuous function on $\partial D_j$ by the analyticity of the boundary (see \cite{saitoh},\cite{guan-19saitoh}), where $\partial/\partial v_z$ denotes the derivative along the outer normal unit vector $v_z$.
 Let $\rho$ be a Lebesgue measurable function on $\partial M$ such that 
$$\rho(w_1,\ldots,w_n)=\left(\sum_{1\le k\le m_j}p_{j,k}\frac{\partial G_{D_j}(w_j,z_{j,k})}{\partial v_{w_j}}\right)^{-1}\times\prod_{1\le l\le n}e^{-\varphi_l(w_l)}$$
on $\partial D_j\times {M_j}$.
Let $c$ be a positive function on $[0,+\infty)$, which satisfies that $c(t)e^{-t}$ is decreasing on $[0,+\infty)$, $\lim_{t\rightarrow0+0}c(t)=c(0)=1$ and $\int_{0}^{+\infty}c(t)e^{-t}dt<+\infty$.
Denote that 
$$\tilde \rho=c(-2\psi)\prod_{1\le j\le n}e^{-\varphi_j}$$
on $M$.

Let us consider the following two minimal integrals. 
Denote that 
\begin{displaymath}
	\begin{split}
		M(Z_0,J,\tilde\rho):=\inf\bigg\{&\int_{M}|f|^2\tilde\rho:f\in \mathcal{O}(D) 
	\text{ s.t. } (f-f_0,z_{\beta})\in J_{\beta}\text{ for any $\beta\in I_1$}\bigg\}	\end{split}
\end{displaymath}
and 
\begin{displaymath}
	\begin{split}
		M_H(Z_0,J,\rho):=\inf\bigg\{&\|f\|^2_{\partial M,\rho}:f\in H_{\rho}^2(M,\partial M) 
		\\ &\text{ s.t. } (f^*-f_0,z_{\beta})\in J_{\beta}\text{ for any $\beta\in I_1$}\bigg\}.\end{split}
\end{displaymath}

Denote that \begin{equation*}
\begin{split}
G(t):=\inf\bigg\{\int_{\{2\psi<-t\}}|f|^2\tilde\rho:&f\in \mathcal{O}(\{2\psi<-t\}) 
	\\
	&\text{ s.t. } (f-f_0,z_{\beta})\in J_{\beta}\text{ for any $\beta\in I_1$}\bigg\}
\end{split}
\end{equation*}
for any $t\ge0$.  Note that $\tilde\rho=c(-2\psi)\prod_{1\le j\le n}e^{-\varphi_j}$ and
$G(0)=M(Z_0,J_{\beta},\tilde\rho)$. As $\mathcal{I}(\varphi+2\psi)_{z_{\beta}}\subset J_{\beta}$ for any $\beta\in I_1$, it follows from 
Theorem \ref{thm:general_concave} that $G(h^{-1}(r))$ is concave, where $h(t)=\int_t^{+\infty}c(s)e^{-s}ds$.

We present a relation between $M_H(Z_0,J,\rho)$ and $M(Z_0,J,\tilde\rho)$.
 
\begin{Theorem}
	\label{thm:2.1}
	Assume that $M(Z_0,J,\tilde\rho)<+\infty$. Then
\begin{equation}
	\label{eq:0212a}M_H(Z_0,J,\rho)\leq\frac{M(Z_0,J,\tilde\rho)}{\pi\int_0^{+\infty}c(t)e^{-t}dt} 
	\end{equation}
 holds, and equality holds if and only if $G(h^{-1}(r))$ is linear on $[0,\int_0^{+\infty}c(t)e^{-t}dt]$ and there exists $f\in H^2_{\rho}(M,\partial M)$, such that $(f^*-f_0,z_\beta)\in J_{\beta}$ for any $\beta\in I_1,$ $M_H(Z_0,J,\rho)=\|f\|^2_{\partial M,\rho}$ and $M(Z_0,J,\tilde\rho)=\int_M|f^*|^2\tilde\rho$.
 \end{Theorem}
 
 \begin{Remark}
 	Let $\hat\rho$ be any Lebesgue measurable function on $\overline M$, which satisfies that $\inf_{\overline{M}}\hat\rho>0$, $-\log\hat\rho$ is plurisubharmonic on $M$ and $\hat\rho(w_j,\hat w_j)\leq \liminf_{w\rightarrow w_j}\hat\rho(w,\hat w_j)$ for any $(w_j,\hat w_j)\in \partial D_j\times M_j\subset \partial M$ and any $1\le j\le n$, where $M_j=\prod_{l\not=j}D_l$. Let $\rho(w_1,\ldots,w_n)=\left(\sum_{1\le k\le m_j}p_{j,k}\frac{\partial G_{D_j}(w_j,z_{j,k})}{\partial v_{w_j}}\right)^{-1}\times\hat\rho$
on $\partial D_j\times {M_j}$, and let $\tilde \rho=c(-2\psi)\hat\rho$
on $M$. Inequality \eqref{eq:0212a} in Theorem \ref{thm:2.1} also holds for this case (We prove the remark in the Step 1 of the proof of  Theorem \ref{thm:2.1}).
 \end{Remark}

 Using Theorem \ref{main theorem1} and Theorem \ref{thm:prod-finite-point} (a characterization for the concavity of $G(h^{-1}(r))$ degenerating to linearity), we obtain the following theorem.

 \begin{Theorem}
 	\label{thm:2.2}Assume that $J_{\beta}=\mathcal{I}(2\psi)_{z_\beta}$ for any $\beta\in I_1$, and $f_0=\prod_{1\le j\le n}(w_j-z_{j,1})^{\tilde\alpha_j}$ on $V_{\beta^*}$, where $\beta^*=(1,1,...,1)\in I_1$. Then equality 
 	$$M_H(Z_0,J,\rho)=\frac{M(Z_0,J,\tilde\rho)}{\pi\int_0^{+\infty}c(t)e^{-t}dt}$$ holds if and only if the following statements hold:

 $(1)$ $\varphi_j=2\log|g_j|+2u_j$ for any $j\in\{1,2,...,n\}$, where $u_j$ is a harmonic function on $D_j$ and $g_j$ is a holomorphic function on $\mathbb{C}$ satisfying $g_j(z_{j,k})\not=0$ for any $k\in\{1,2,...,m_j\}$;
	
	$(2)$ There exists a nonnegative integer $\gamma_{j,k}$ for any $j\in\{1,2,...,n\}$ and $k\in\{1,2,...,m_j\}$, which satisfies that $\Pi_{1\le k\leq m_j}\chi_{j,z_{j,k}}^{\gamma_{j,k}+1}=\chi_{j,-u_j}$ and $\sum_{1\le j\le n}\frac{\gamma_{j,\beta_j}+1}{p_{j,\beta_j}}=1$ for any $\beta\in I_1$, where $\chi_{-u_j}$ and $\chi_{z_{j,k}}$ are the  characters associated to the functions $-u_j$ and $G_{D_j}(\cdot,z_{j,k})$ respectively;
	
	$(3)$ $f_0=c_{\beta}\Pi_{1\le j\le n}(w_j-z_{j,\beta_j})^{\gamma_{j,\beta_j}}+g_{\beta}$ on $V_{\beta}$ for any $\beta\in I_1$, where $c_{\beta}$ is a constant and $g_{\beta}$ is a holomorphic function on $V_{\beta}$ such that $(g_{\beta},z_{\beta})\in\mathcal{I}(\psi)_{z_{\beta}}$;
	
	$(4)$ $\lim_{z\rightarrow z_{\beta}}\frac{c_{\beta}\Pi_{1\le j\le n}(w_j-z_{j,\beta_j})^{\gamma_{j,\beta_j}}dw_{1}\wedge dw_{2}\wedge...\wedge dw_{n}}{\wedge_{1\le j\le n} g_j(P_{j})_*\left(f_{u_j}\left(\Pi_{1\le k\le m_j}f_{z_{j,k}}^{\gamma_{j,k}+1}\right)\left(\sum_{1\le k\le m_j}p_{j,k}\frac{df_{z_{j,k}}}{f_{z_{j,k}}}\right)\right)}=c_0$ for any $\beta\in I_1$, where $c_0\in\mathbb{C}\backslash\{0\}$ is a constant independent of $\beta$, $P_j:\Delta\rightarrow D_j$ is the universal covering, $f_{u_j}$ is a holomorphic function $\Delta$ such that $|f_{u_j}|=P_j^*(e^{u_j})$ and $f_{z_{j,k}}$ is a holomorphic function on $\Delta$ such that $|f_{z_{j,k}}|=P_j^*\left(e^{G_{D_j}(\cdot,z_{j,k})}\right)$ for any $j\in\{1,2,...,n\}$ and $k\in\{1,2,...,m_j\}$.
 \end{Theorem}

When $m_j=0$ for any $1\le j\le n$, the above theorem is a solution of the product version of Saitoh's conjecture, which can be referred to \cite{GY-saitohprodct}.
 
 \begin{Remark}
	\label{rem2.1} Assume that the four statements in Theorem \ref{thm:2.2} hold, then we know $\frac{c_0\wedge_{1\le j\le n} g_j(P_{j})_*\left(f_{u_j}\left(\Pi_{1\le k\le m_j}f_{z_{j,k}}^{\gamma_{j,k}+1}\right)\left(\sum_{1\le k\le m_j}p_{j,k}\frac{df_{z_{j,k}}}{f_{z_{j,k}}}\right)\right)}{dw_{1}\wedge dw_{2}\wedge...\wedge dw_{n}}$ is a (single-valued) holomorphic function on $D_j$, and we denote it by $F_0$. Then $(F_0-f_0,z_{\beta})\in\mathcal{I}(2\psi)_{z_\beta}$ for any $\beta\in I_1$, and there exists $\tilde F_0\in H^2_{\rho}(M,\partial M)$ such that $\tilde F_0^*=F_0$,
	$$M(Z_0,J,\tilde\rho)=\int_{M}|F_0|^2\tilde\rho \,\,\,\,and\,\,\,\, M_H(Z_0,J,\rho)=\|\tilde F_0\|^2_{\partial M,\rho}.$$
	We prove the remark in Section \ref{sec:proof2}.
\end{Remark}

Denote that $E_{\beta}:=\left\{(\alpha_1,\alpha_2,...,\alpha_n):\sum_{1\le j\le n}\frac{\alpha_j+1}{p_{j,\beta_j}}=1\,\&\,\alpha_j\in\mathbb{Z}_{\ge0}\right\}$ for any $\beta\in I_1$, and assume that  $f_0=\sum_{\alpha\in E_{\beta}}d_{\beta,\alpha}\prod_{1\le j\le n}(w_j-z_{j,\beta_j})^{\alpha_j}$ on $V_{\beta}$.  Denote that
\begin{equation*}
c_{j,k}:=\exp\lim_{z\rightarrow z_{j,k}}\left(\frac{\sum_{1\le k_1\le m_j}p_{j,k_1}G_{D_j}(z,z_{j,k_1})}{p_{j,k}}-\log|w_{j,k}(z)|\right)
\end{equation*}
 for any $j\in\{1,2,...,n\}$ and $k\in\{1,2,...,m_j\}$.

Using Theorem \ref{thm:2.2}, we obtain the following optimal $L^2$ extension theorem for the Hardy space on product spaces, and give a characterization for the holding of the equality in this extension theorem.

\begin{Corollary}
	\label{c:2.1}
	Assume that $\sum_{\beta\in I_1}\sum_{\alpha\in E_{\beta}}\frac{|d_{\beta,\alpha}|^22^n\pi^{n-1}e^{-\varphi(z_{\beta})}}{\Pi_{1\le j\le n}(\alpha_j+1)c_{j,\beta_j}^{2\alpha_{j}+2}}\in(0,+\infty)$.
Then there exists $f\in H^2_{\rho}(M,\partial M)$, satisfying that $(f^*-f_0,z_{\beta})\in\mathcal{I}(2\psi)_{z_\beta}$ for any $\beta\in I_1$ and
	$$\|f\|_{\partial M,\rho}^2\le\sum_{\beta\in I_1}\sum_{\alpha\in E_{\beta}}\frac{|d_{\beta,\alpha}|^22^n\pi^{n-1}e^{-\varphi(z_{\beta})}}{\Pi_{1\le j\le n}(\alpha_j+1)c_{j,\beta_j}^{2\alpha_{j}+2}}.$$
	
	Moreover, assume that $f_0=\prod_{1\le j\le n}(w_j-z_{1,\beta_1})^{\tilde\alpha_{j}}$ on $V_{\beta^*}$, where $\beta^*=(1,1,...,1)\in I_1$, then 
 equality 
 $$M_H(Z_0,\mathcal{I}(2\psi),\rho)=\sum_{\beta\in I_1}\sum_{\alpha\in E_{\beta}}\frac{|d_{\beta,\alpha}|^22^n\pi^{n-1}e^{-\varphi(z_{\beta})}}{\Pi_{1\le j\le n}(\alpha_j+1)c_{j,\beta_j}^{2\alpha_{j}+2}}$$ holds if and only if the following statements hold:

	$(1)$  $\varphi_j=2\log|g_j|+2u_j$ for any $j\in\{1,2,...,n\}$, where $u_j$ is a harmonic function on $\mathcal{C}$ and $g_j$ is a holomorphic function on $D_j$ satisfying $g_j(z_{j,k})\not=0$ for any $k\in\{1,2,...,m_j\}$;
	
	$(2)$  there exists a nonnegative integer $\gamma_{j,k}$ for any $j\in\{1,2,...,n\}$ and $k\in\{1,2,...,m_j\}$, which satisfies that $\Pi_{1\le k\leq m_j}\chi_{j,z_{j,k}}^{\gamma_{j,k}+1}=\chi_{j,-u_j}$ and $\sum_{1\le j\le n}\frac{\gamma_{j,\beta_j}+1}{p_{j,\beta_j}}=1$ for any $\beta\in I_1$;
	
	$(3)$ $f_0=c_{\beta}\Pi_{1\le j\le n}(w_j-z_{j,\beta_j})^{\gamma_{j,\beta_j}}+g_{\beta}$ on $V_{\beta}$ for any $\beta\in I_1$, where $c_{\beta}$ is a constant and $g_{\beta}$ is a holomorphic function on $V_{\beta}$ such that $(g_{\beta},z_{\beta})\in\mathcal{I}(\psi)_{z_{\beta}}$;
	
	$(4)$ $\lim_{z\rightarrow z_{\beta}}\frac{c_{\beta}\Pi_{1\le j\le n}(w_j-z_{j,\beta_j})^{\gamma_{j,\beta_j}}dw_{1}\wedge dw_{2}\wedge...\wedge dw_{n}}{\wedge_{1\le j\le n}g_j(P_{j})_*\left(f_{u_j}\left(\Pi_{1\le k\le m_j}f_{z_{j,k}}^{\gamma_{j,k}+1}\right)\left(\sum_{1\le k\le m_j}p_{j,k}\frac{df_{z_{j,k}}}{f_{z_{j,k}}}\right)\right)}=c_0$ for any $\beta\in I_1$, where $c_0\in\mathbb{C}\backslash\{0\}$ is a constant independent of $\beta$.	\end{Corollary}

Denote that $L_{k}:=\{\sum_{1\le j\le n}\alpha_{j}\le k:\alpha=(\alpha_1,\ldots,\alpha_n)\in\mathbb{Z}_{\ge0}^n\}$.

\begin{Corollary}
	\label{c:2.2}
	Let $k$ be a nonnegative integer. Then there is a constant $C$ (depending on $k$ and $Z_0$), such that for any $a_{\beta,\alpha}\in\mathbb{C}$, where  $\beta\in I_1$ and $\alpha\in L_k$, there exists $f\in H^2_{\rho}(M,\partial M)$ such that $\partial^{\alpha}f^*(z_{\beta})=a_{\beta,\alpha}$ for any $\beta\in I_1$ and $\alpha\in L_k$, and
	$$\|f\|_{\partial M,\rho}^2\le C\sum_{\beta\in I_1,\alpha\in L_k}|a_{\beta,\alpha}|^2,$$
	where $\partial^{\alpha}=\left(\frac{\partial}{\partial w_1}\right)^{\alpha_1}\dots\left(\frac{\partial}{\partial w_n}\right)^{\alpha_n}$.
\end{Corollary}

\subsection{Minimal $L^2$ integrals for the Hardy space  $H^2_{\lambda}(M,S)$}

Let $D_j$ be a planar regular region with finite boundary components which are analytic Jordan curves   for any $1\le j\le n$. Let $M=\prod_{1\le j\le n}D_j$ be a bounded domain in $\mathbb{C}^n$.  
Denote that $S:=\prod_{1\le j\le n}\partial D_j$.

Let $Z_j=\{z_{j,1},z_{j,2},...,z_{j,m_j}\}\subset D_j$ for any  $j\in\{1,2,...,n\}$, where $m_j$ is a positive integer. Denote that 
$$Z_0:=\prod_{1\le j\le n}Z_j\subset M.$$
Let $\psi=\max_{1\le j\le n}\{\sum_{1\le k\le m_j}2G_{D_j}(\cdot,z_{j,k})\}$.
Let $V_{z_{j,k}}\Subset D_{j}$ be a  neighborhood of $z_{j,k}$ satisfying $V_{z_{j,k}}\cap V_{z_{j,k'}}=\emptyset$ for any $j$ and $k\not=k'$. Denote that $I_1:=\{(\beta_1,\beta_2,...,\beta_n):1\le \beta_j\le m_j$ for any $j\in\{1,2,...,n\}\}$, $V_{\beta}:=\prod_{1\le j\le n}V_{z_{j,\beta_j}}$  and $z_{\beta}:=(z_{1,\beta_1},z_{2,\beta_2},\ldots,z_{n,\beta_n})\in M$ for any $\beta=(\beta_1,\beta_2,...,\beta_n)\in I_1$.

Let $\varphi_j$ be a subharmonic function on $D_j$, which satisfies that  $\varphi_j$ is continuous at $z$ for any $z\in \partial D_j$. Denote that 
$$\varphi(w_1,\ldots,w_n):=\sum_{1\le j\le n}\varphi_j(w_j)$$
on $M$. Let $f_0$ be a holomorphic function $\cup_{\beta\in I_1}V_{\beta}$. 
 Let $\rho$ be a Lebesgue measurable function on $\partial M$ such that 
$$\rho(w_1,\ldots,w_n)=\left(\sum_{1\le k\le m_j}2\frac{\partial G_{D_j}(w_j,z_{j,k})}{\partial v_{w_j}}\right)^{-1}\times\prod_{1\le l\le n}e^{-\varphi_l(w_l)}$$
on $\partial D_j\times {M_j}$.
Let $c$ be a positive function on $[0,+\infty)$, which satisfies that $c(t)e^{-t}$ is decreasing on $[0,+\infty)$, $\lim_{t\rightarrow0+0}c(t)=c(0)=1$ and $\int_{0}^{+\infty}c(t)e^{-t}dt<+\infty$. Let 
$$\lambda(w_1,\ldots,w_n)=\prod_{1\le j\le n}\left(\sum_{1\le k\le m_j}2\frac{\partial G_{D_j}(w_j,z_{j,k})}{\partial v_{w_j}}\right)^{-1}e^{-\varphi_j(w_j)}$$
on $S=\prod_{1\le j\le n}\partial D_j$. Note that $\lambda$ is continuous on $S$.

Let $h_j$ be a holomorphic function on a neighborhood of $Z_j$ for any $1\le j\le n$ satisfying that there exists $k\in\{1,\ldots,m_j\}$ such that $h_j(z_{j,k})\not=0$. Denote that $h_0=\prod_{1\le j\le n}h_j$. Let us consider the following  minimal integral. Let $J_{\beta}$ be the maximal ideal of $\mathcal{O}_{z_\beta}$ for any $\beta\in I_1$.
Denote that 
\begin{displaymath}
	\begin{split}
		M_S(Z_0,J,\lambda):=\inf\bigg\{&\|f\|^2_{S,\lambda}:f\in H_{\lambda}^2(M,S) 
		\\ &\text{ s.t. } f^*(z_{\beta})=h_0(z_{\beta})\text{ for any $\beta\in I_1$}\bigg\}\end{split}
\end{displaymath}
and
\begin{displaymath}
	\begin{split}
		M_H(Z_0,J,\rho):=\inf\bigg\{&\|f\|^2_{\partial M,\rho}:f\in H_{\rho}^2(M,\partial M) 
		\\ &\text{ s.t. } f^*(z_{\beta})=h_0(z_{\beta})\text{ for any $\beta\in I_1$}\bigg\}.\end{split}
\end{displaymath}

We present a relation between $M_S(Z_0,J,\lambda)$ and $M_H(Z_0,J,\rho)$.
 
\begin{Theorem}
	\label{thm:3.1}
	Assume that $M_H(Z_0,J,\rho)<+\infty$. Then
\begin{equation}
	\label{eq:0312a}M_S(Z_0,J,\lambda)\leq\frac{M_H(Z_0,J,\rho)}{n\pi^{n-1}}
		\end{equation}
 holds, and equality holds if and only if the following three statements hold 
 
 $(1)$ $\varphi_j=2u_j$ for any $1\le j\le n$, where $u_j$ is a harmonic function on $D_j$;
 
 $(2)$ $\prod_{1\le k\le m_j}\chi_{j,z_{j,k}}=\chi_{j,-u_j}$ for any $1\le j\le n$;
 
 $(3)$ For any $j$, there exists a constant $c_j\not=0$ such that 
 $$\lim_{z\rightarrow z_{j,k}}\frac{P_j^*\left(f_{u_j}\left(\prod_{1\le k\le m_j}f_{z_{j,k}}\right)\left(\sum_{1\le k\le m_j}\frac{df_{z_{j,k}}}{f_{z_{j,k}}}\right)\right)}{h_jdw_j}=c_j$$
  holds for any $1\le k\le m_j$.
 \end{Theorem}

\section{Preparations}

In this section, we do some preparations.

\subsection{Some results on Hardy space $H^2(D)$}\label{sec:2.1}
Let $D$ be a planar regular region with finite boundary components which are analytic Jordan curves  (see \cite{saitoh}, \cite{yamada}). In this section, we recall some propoerties related to Hardy space $H^2(D)$.

Let $H^2(D)$ (see \cite{saitoh}) denote the analytic Hardy class on $D$ defined as the set of all analytic functions $f(z)$ on $D$ such that the subharmonic functions $|f(z)|^2$ have harmonic majorants $U(z)$: 
$$|f(z)|^2\le U(z)\,\, \text{on}\,\,D.$$
Then each function $f(z)\in H^2(D)$ has Fatou's nontangential boundary value a.e. on $\partial D$ belonging to $L^2(\partial D)$ (see \cite{duren}). It is well know (see \cite{rudin55}) that if a subharmonic function has a harmonic majorant in $D$, then there exists a least harmonic majorant. Denote the least harmonic majorant of $|f|^2$ by $u_f$.

Let $z_0\in D$. Let $L^2(\partial D,\rho)$ be the space of complex valued measurable function $h$ on $\partial D$, normed by 
$$\|h\|_{\partial D,\rho}^2=\frac{1}{2\pi}\int_{\partial D}|h|^2\rho |dz|,$$
where $\rho=\frac{\partial G_D(z,z_0)}{\partial v_z}$ is a positive continuous function on $\partial D$ by the analyticity of $\partial D$, $G_D(z,z_0)$ is the Green function on $D$, and $\partial/\partial v_z$ denotes the derivative along the outer normal unit vector $v_z$.

The following lemma gives some properties related to the Hardy space $H^2(D)$.
\begin{Lemma}[\cite{rudin55}]
\label{l:0-1}
	$(a)$ If $f\in H^2(D)$, there is a function $f_*$ on $\partial D$ such that $f$  has nontangential boundary value $f_*$ almost everywhere on $\partial D$. The map $\gamma:f\mapsto f_*$ is an injective linear map from $H^2(D)$ into $L^2(\partial D,\rho)$ and
	$$\|f_*\|_{\partial D,\rho}^2=u_f(z_0)$$
	holds for any $f\in H^2(D)$, where $u_f$ is the least harmonic majorant of $|f|^2$. 
	
	$(b)$ $g\in \gamma(H^2(D))$ if and only if 
	$$\int_{\partial D}g(z)\phi(z)dz=0$$
	holds for any holomorphic function $\phi$ on a neighborhood of  $\overline D$.
	 
	 $(c)$ The inverse of $\gamma $ is given by 
	\begin{equation}
		\label{eq:0728a}f(w)=\frac{1}{2\pi\sqrt{-1}}\int_{\partial D}\frac{f_*(z)}{z-w}dz
	\end{equation}
	for any $z\in D$.
\end{Lemma}

Equality \eqref{eq:0728a} in Lemma \ref{l:0-1} implies the following lemma.
\begin{Lemma}
	\label{l:0-1b}If $\lim_{n\rightarrow+\infty}\|\gamma(f_n)\|_{\partial D,\rho}=0$ for $f_n\in H^2(D)$, then $f_n$ uniformly converges to $0$ on any compact subset of $D$.
\end{Lemma}

\begin{Lemma}
	\label{l:bounded}
	For any compact set $V\subset D$ and nonnegative integer $k$, there is a constant $C>0$, such that 
	$$|f^{(k)}(w)|^2\le C\int_{\partial D}|f_*|^2|dz|$$
	for any $w\in V$ and $f\in H^2(D)$.
\end{Lemma}
\begin{proof}
	By equality \eqref{eq:0728a}, we have
	$$f^{(k)}(w)=\frac{(-1)^{k+1}k!}{2\pi\sqrt{-1}}\int_{\partial D}\frac{f_*(z)}{(z-w)^{k+1}}dz$$
	for any $w\in D$ and $f\in H^2(D)$.
	Hence, there exists a constant $C>0$, such that 
	$|f^{(k)}(w)|^2\le C\int_{\partial D}|f_*|^2|dz|$
	for any $w\in V$.
	\end{proof}

\begin{Lemma}[\cite{rudin55}] \label{l:0-2}$ H^2(D)$ is a Hilbert space equipped with the inner product
	$$\ll f,g\gg_{\partial D,\rho}=\frac{1}{2\pi}\int_{\partial D} f_*\overline {g_*}\rho|dz|,$$
	where $\rho=\frac{\partial G_D(z,z_0)}{\partial v_z}$. Moreover,	$\mathcal{O}(D)\cap C(\overline D)$ is dense in $ H^2(D)$.
\end{Lemma}

\begin{Lemma}
	\label{l:0-3}Let $f_n\in H^2(D)$ for any $n\in \mathbb{Z}_{>0}$. Assume that $f_n$ uniformly converges to $0$ on any compact subset of $D$ and   there exists $f\in L^2(\partial D,\rho)$ such that  $\lim_{n\rightarrow+\infty}\|\gamma(f_n)-f\|_{\partial D,\rho}=0$. Then we have $f=0$.  
\end{Lemma}
\begin{proof}
	It follows from Lemma \ref{l:0-1} and Lemma \ref{l:0-2} that there exists $f_0\in H^2(D)$ such that $\gamma(f_0)=f$. Using Lemma \ref{l:0-1b}, we get that $f_n-f_0$ uniformly converges to $0$ on any compact subset of $D$, i.e. $f_0=0$, which implies that $f=0$.
\end{proof}

Let $\{D_k\}_{k\in\mathbb{Z}_{>0}}$ be an increasing sequence of domains with analytic boundaries, such that $z_0\in D_1$ and $\cup_{k=1}^{+\infty} D_k=D$. Let $G_{D_k}(\cdot,z_0)$ be the Green function of $D_k$.
\begin{Lemma}[see \cite{rudin55}] \label{l:0-4}
	$\|f_*\|^2_{\partial D,\rho}=\lim_{k\rightarrow+\infty}\frac{1}{2\pi}\int_{\partial D_k}|f|^2\frac{\partial G_{D_k}(z,z_0)}{\partial v_z}|dz|$
	holds for any $f\in H^2(D)$.
\end{Lemma}

We recall a well-known property for the Green function $G_{D}(\cdot,z_j)$ on $D$.

\begin{Lemma}[see \cite{GY-concavity3}]
	\label{l:green-sup2}Let $Z_0':=\{z_j:j\in\mathbb{Z}_{\ge1}\,\&\,j<\gamma \}$ be a discrete subset of $D$, where $\gamma\in\mathbb{Z}_{\geq1}\cup\{+\infty\}$. Let $\psi$ be a negative subharmonic function on $D$ such that $\frac{1}{2}v(dd^c\psi,z_j)\ge p_j>0$ for any $j$, where $p_j$ is a constant. Then $2\sum_{1\le j< \gamma}p_jG_{D}(\cdot,z_j)$ is a subharmonic function on $D$ satisfying that $2\sum_{1\le j<\gamma }p_jG_{D}(\cdot,z_j)\ge\psi$ and $2\sum_{1\le j<\gamma }p_jG_{D}(\cdot,z_j)$ is harmonic on $D\backslash Z_0'$.
\end{Lemma}

We recall the following basic formula.
\begin{Lemma}[see \cite{GY-saitohprodct}]
	\label{l:4}$\frac{\partial \psi}{\partial v_z}=\left(\left(\frac{\partial \psi}{\partial x}\right)^2+\left(\frac{\partial \psi}{\partial y}\right)^2\right)^{\frac{1}{2}}$ on $\partial D$, where $\partial/\partial v_z$ denotes the derivative along the outer normal unit vector $v_z$. 
\end{Lemma}

Let $\psi=\sum_{1\le j\le m}p_jG_{D}(\cdot,z_j)$, where $p_j>0$ and $\{z_j\}\subset D$ satisfying $z_j\not= z_k$ for $j\not=k$. Then there exist a neighborhood $U$ of $\partial D$ and $r_0\in(0,1)$ such that $\{z\in D:\psi(z)\ge\log r_0\}\Subset U$ and $dG_{D}(\cdot,z_j) \not=0$ on $U\cap\overline{D}$ for any $j$. 

The following lemma will be used in the proof of Theorem \ref{main theorem1}. 
\begin{Lemma} \label{l:0-4v2}Let $\varphi$ be a positive Lebesgue measurable function on $U\cap \overline D$ satisfying that  $\lim_{z\rightarrow\tilde z}\varphi(z)=\varphi(\tilde z)$ for any $\tilde z\in\partial D$. Then 
\begin{equation}
	\label{eq:1210e}\int_{\partial D}|f|^2\varphi|dz|=\lim_{r\rightarrow1-0}\int_{\partial D_r}|f|^2\varphi|dz|
\end{equation}
	holds for any $f\in H^2(D)$, where $D_r=\{z\in \overline D:\psi(z)<\log r\}$ for $r\in[r_0,1]$.
\end{Lemma}
\begin{proof}
	Following from Lemma \ref{l:green-sup2}, we know that $\psi-\log r=\sum_{1\le j\le m}p_jG_{D_r}(\cdot,z_j)$ on $D_r$. Thus, using Lemma \ref{l:0-4}, we get that 
	\begin{equation}
		\label{eq:1210b}
		\begin{split}
					\lim_{r\rightarrow1-0}\int_{\partial D_r}|\tilde f|^2\frac{\partial\psi}{\partial v_z}|dz|&=\sum_{1\le j\le m}p_j\lim_{r\rightarrow1-0}\int_{\partial D_r}|\tilde f|^2\frac{\partial G_{D_r}(z,z_j)}{\partial v_z}|dz|\\
					&=\sum_{1\le j\le m}p_j\int_{\partial D}|\tilde f|^2\frac{\partial G_{D}(z,z_j)}{\partial v_z}|dz|\\
					&=\int_{\partial D}|\tilde f|^2\frac{\partial \psi}{\partial v_z}|dz|
		\end{split}
			\end{equation}
	holds for any $\tilde f\in H^2(D)$. 
	As  $\lim_{z\rightarrow\tilde z}\varphi(z)=\varphi(\tilde z)$ for any $\tilde z\in\partial D_1$ and $dG_{D}(\cdot,z_j) \not=0$ on $U\cap\overline{D}$ for any $j$ , there exists a positive number $L_1$ such that 
	\begin{displaymath}
		\frac{1}{L_1}<\inf_{\{z\in\overline D:\psi(z)\ge\log r_0\}}\min\{|\bigtriangledown\psi|,\varphi\}\le\sup_{\{z\in\overline D:\psi(z)\ge\log r_0\}}\max\{|\bigtriangledown\psi|,\varphi\}<L_1,
	\end{displaymath}
	where $|\bigtriangledown\cdot|^2=\left(\frac{\partial \cdot}{\partial x}\right)^2+\left(\frac{\partial \cdot}{\partial y}\right)^2$.

		By Lemma \ref{l:0-2}, there exists $\{f_n\}_{n\in\mathbb{Z}_{>0}}\subset\mathcal{O}(D)\cap C(\overline D)$ such that 
	\begin{equation}
		\label{eq:1210a}\lim_{n\rightarrow+\infty}\int_{\partial D}|f_n-f|^2\varphi|dz|=0.
	\end{equation}
	 It follows from equality \eqref{eq:1210b} and Lemma \ref{l:4} that 
\begin{equation}
\label{eq:1210c}
\begin{split}
	\limsup_{r\rightarrow1-0}\int_{\partial D_r}|f_n-f|^2\varphi|dz|&\le L_1^2\limsup_{r\rightarrow1-0}\int_{\partial D_r}|f_n-f|^2\frac{\partial\psi}{\partial v_z}|dz|\\
	&\le L_1^2\int_{\partial D}|f_n-f|^2\frac{\partial\psi}{\partial v_z}|dz|\\
	&\le L_1^4\int_{\partial D}|f_n-f|^2\varphi |dz|.
\end{split}
\end{equation}
Using the dominated convergence theorem, we know that 
\begin{equation}
	\label{eq:1210d}
	\lim_{r\rightarrow1-0}\int_{\partial D_r}|f_n|^2\varphi |dz|=\int_{\partial D}|f_n|^2\varphi |dz|
\end{equation}	
	holds for any $n\in\mathbb{Z}_{>0}$. 
Following from equality \eqref{eq:1210a}, inequality \eqref{eq:1210c} and equality \eqref{eq:1210d}, we have
\begin{displaymath}
	\begin{split}
		&\limsup_{r\rightarrow1-0}\left(\int_{\partial D_r}|f|^2\varphi |dz|\right)^{\frac{1}{2}}\\\le&\liminf_{n\rightarrow+\infty}\left(\limsup_{r\rightarrow1-0}\left(\int_{\partial D_r}|f_n|^2\varphi |dz|\right)^{\frac{1}{2}}+\limsup_{r\rightarrow1-0}\left(\int_{\partial D_r}|f_n-f|^2\varphi |dz|\right)^{\frac{1}{2}}\right)\\
		\le &\liminf_{n\rightarrow+\infty}\left(\left(\int_{\partial D}|f_n|^2\varphi |dz|\right)^{\frac{1}{2}}+L_1^2\left(\int_{\partial D}|f_n-f|^2\varphi |dz|\right)^{\frac{1}{2}}\right)\\
		= &\left(\int_{\partial D}|f|^2\varphi |dz|\right)^{\frac{1}{2}}.
	\end{split}
\end{displaymath}
By Fatou's Lemma, we have 
$$\liminf_{r\rightarrow1-0}\left(\int_{\partial D_r}|f|^2\varphi |dz|\right)^{\frac{1}{2}}\ge\left(\int_{\partial D}|f|^2\varphi |dz|\right)^{\frac{1}{2}}.$$
Thus, equality \eqref{eq:1210e} holds.
\end{proof}

Let $Z_0:=\{z_j:1\le j\le m\}$ be a subset of $D$.
Let $\rho$ be a positive continuous function on $\partial D$.
 Let $\mathfrak{a}=(a_{j,l})$ $(1\le j\le m,0\le l\leq k_j )$, where $a_{j,l}\in\mathbb{C}$ such that $\sum_{1\le j\le m}\sum_{0\le l\le k_j}|a_{j,l}|\not=0$.
Denote that 
\begin{displaymath}
	\begin{split}
		M(Z_0,\mathfrak{a},\tilde\rho):=\inf\bigg\{&\int_{ D}|f|^2\tilde\rho:f\in \mathcal{O}(D) 
		\\ &\text{ s.t. } f^{(l)}(z_j)=l!a_{j,l}\text{ for any $0\le l\le k_j$ and any $1\le j\le m$}\bigg\}.	\end{split}
\end{displaymath}
and 
\begin{displaymath}
	\begin{split}
		M_H(Z_0,\mathfrak{a},\rho):=\inf\bigg\{\frac{1}{2\pi}&\int_{\partial D}|f|^2\rho|dz|:f\in H^2(D) 
		\\ &\text{ s.t. } f^{(l)}(z_j)=l!a_{j,l}\text{ for any $0\le l\le k_j$ and any $1\le j\le m$}\bigg\}.	\end{split}
\end{displaymath}

\begin{Lemma}
	\label{l:exists of $M_H$}If $M_H(Z_0,\mathfrak{a},\rho)<+\infty$, then there exists a unique $f\in H^2(D)$ such that $M_H(Z_0,\mathfrak{a},\rho)=\frac{1}{2\pi}\int_{\partial D}|f|^2\rho|dz|$, and $f^{(l)}(z_j)=l!a_{j,l}$ for any $0\le l\le k_j$ and any $1\le j\le m$. 
\end{Lemma}
\begin{proof}
Firstly, we prove the existence of $f$.
As $M_H(Z_0,\mathfrak{a},\rho)<+\infty,$
then there is $\{f_{s}\}_{s\in\mathbb{Z}_{>0}}\subset H^2(D)$  such that
$$\lim_{s\rightarrow+\infty}\frac{1}{2\pi}\int_{\partial D}|f_s|^2\rho|dz|=M_H(Z_0,\mathfrak{a},\rho),$$ 
and $f_s^{(l)}(z_j)=l!a_{j,l}$ for any $0\le l\le k_j$ and any $1\le j\le m$. 
Thus, there exists a subsequence of $\{f_{s}\}_{s\in\mathbb{Z}_{>0}}$ (denoted also by $\{f_{s}\}_{s\in\mathbb{Z}_{>0}}$), which satisfies that $\{f_{s}\}_{s\in\mathbb{Z}_{>0}}$ weakly converges to a function $g\in L^2(\partial D,\rho)$ in the Hilbert space $L^2(\partial D,\rho)$ and $\{f_{s}\}_{s\in\mathbb{Z}_{>0}}$ uniformly converges to a function $f\in \mathcal{O}(D)$ on any compact subset of $D$. 
Then we have 
\begin{equation}
	\label{eq:1229a}\frac{1}{2\pi}\int_{\partial D}|g|^2\rho|dz|\le\lim_{s\rightarrow+\infty}\frac{1}{2\pi}\int_{\partial D}|f_s|^2\rho|dz|=M_H(Z_0,\mathfrak{a},\rho).
\end{equation}
 By Lemma \ref{l:0-1}, we have
\begin{displaymath}
	\begin{split}
		\int_{\partial D}g(z)\phi(z)dz&=\int_{\partial D}g(z)\left(\phi(z)\frac{dz}{|dz|\rho(z)}\right)\rho(z)|dz|\\
		&=\lim_{s\rightarrow+\infty}\int_{\partial D}f_s(z)\left(\phi(z)\frac{dz}{|dz|\rho(z)}\right)\rho(z)|dz|\\
		&=\lim_{s\rightarrow+\infty}\int_{\partial D}f_s(z)\phi(z)dz\\
		&=0
	\end{split}
\end{displaymath} 
for any holomorphic function $\phi$ on a neighborhood of $\overline D$, and 
\begin{displaymath}
	\begin{split}
		f(w)&=\lim_{s\rightarrow+\infty}f_s(w)\\
		&=\lim_{s\rightarrow+\infty}\frac{1}{2\pi\sqrt{-1}}\int_{\partial D}\frac{f_s(z)}{z-w}dz\\
		&=\lim_{s\rightarrow+\infty}\int_{\partial D}f_s(z)\left(\frac{dz}{|dz|(z-w)\rho(z)}\right)\rho(z)|dz|\\
		&=\frac{1}{2\pi\sqrt{-1}}\int_{\partial D}\frac{g(z)}{z-w}dz.
	\end{split}
\end{displaymath}
Thus, it follows from Lemma \ref{l:0-1} that $f\in H^2(D)$ and $\gamma(f)=g$. By inequality \eqref{eq:1229a} and the definition of $M_H(Z_0,\mathfrak{a},\rho)$, we get 
$$\frac{1}{2\pi}\int_{\partial D}|f|^2\rho|dz|=M_H(Z_0,\mathfrak{a},\rho).$$
Thus,  we obtain the existence of $f$.

Secondly, we prove the uniqueness of $f$ by contradiction:
if not, there exist two different $g_{1}\in H^2(D)$ and $g_{2}\in H^2(D)$ satisfying that $\frac{1}{2\pi}\int_{\partial D}|g_s|^2\rho|dz|=M_H(Z_0,\mathfrak{a},\rho)$, and $g_s^{(l)}(z_j)=l!a_{j,l}$ for any $0\le l\le k_j$ and any $1\le j\le m$, where $s=1,2$. Note that
\begin{equation}\nonumber
\begin{split}
\frac{1}{2\pi}\int_{\partial D}|\frac{g_1+g_2}{2}|^2\rho|dz|+\frac{1}{2\pi}\int_{\partial D}|\frac{g_1-g_2}{2}|^2\rho|dz|=
M_H(Z_0,\mathfrak{a},\rho),
\end{split}
\end{equation}
hence we obtain that
$$\frac{1}{2\pi}\int_{\partial D}|\frac{g_1+g_2}{2}|^2\rho|dz|<M_H(Z_0,\mathfrak{a},\rho),$$
 which contradicts the definition of $M_H(Z_0,\mathfrak{a},\rho)$.

Thus, Lemma \ref{l:exists of $M_H$} has been proved.
\end{proof}

In the following, let $\psi$ be as in Theorem \ref{main theorem1}. 
The following lemma will be used in the proof of Lemma \ref{l:01a}.

\begin{Lemma}[\cite{GY-weightedsaitoh}]
	\label{l:2}
	Let $f$ be a holomorphic function on $D$.   Assume that 
	\begin{equation}
		\nonumber \liminf_{r\rightarrow1-0}\frac{\int_{\{z\in D:\psi(z)\ge\log r\}}|f(z)|^2}{1-r}<+\infty,
	\end{equation}
	then we have $f\in H^{2}(D)$.
	\end{Lemma}

We recall the following coarea formula.
\begin{Lemma}[see \cite{federer}]
	\label{l:3}Suppose that $\Omega$ is an open set in $\mathbb{R}^n$ and $u\in C^1(\Omega)$. Then for any $g\in L^1(\Omega)$, 
	$$\int_{\Omega}g(x)|\bigtriangledown u(x)|dx=\int_{\mathbb{R}}\left(\int_{u^{-1}(t)}g(x)dH_{n-1}(x)\right)dt,$$
	where $H_{n-1}$ is the $(n-1)$-dimensional Hausdorff measure.
\end{Lemma}

Let $\tilde\rho$ be a Lebesgue measurable function on $\overline D$, which satisfies that $\inf_{\overline D}\tilde\rho>0$ and $\tilde\rho(z)\leq\liminf_{w\rightarrow z}\tilde\rho(w)$ for any $z\in\partial D$. Denote that 
$$\rho=\left(\frac{\partial \psi}{\partial v_z}\right)^{-1}\tilde\rho$$
on $\partial D$. In the following, we give a sufficient condition for $f\in H^{2}(D)$ .

\begin{Lemma}
	\label{l:01a}
	Let $f$ be a holomorphic function on $D$.   Assume that 
	\begin{equation}
		\nonumber \liminf_{r\rightarrow1-0}\frac{\int_{\{z\in D:\psi(z)\ge\log r\}}|f(z)|^2\tilde\rho}{1-r}<+\infty,
	\end{equation}
	then we have $f\in H^{2}(D)$ and 
	\begin{equation}
		\label{eq:221201a}\int_{\partial D}|f|^2\rho|dz| \le \liminf_{r\rightarrow1-0}\frac{\int_{\{z\in D:\psi(z)\ge\log r\}}|f(z)|^2\tilde\rho}{1-r}.	\end{equation}
\end{Lemma}
\begin{proof}
Note that $\inf_{\overline D}\tilde\rho>0$, then Lemma \ref{l:2} tells us that $f\in H^2(D)$. Thus, it suffices to prove
 inequality \eqref{eq:221201a}.

Note that $f$ has Fatou's nontangential boundary value on $\partial D$.
By Lemma \ref{l:4}, we have 
$$\int_{\partial D}|f|^2\rho|dz|
		=\int_{\partial D}|f|^2\left(\frac{\partial \psi}{\partial v_z}\right)^{-1}\tilde\rho|dz|
		=\int_{\partial D}|f|^2\tilde\rho\left|\bigtriangledown\psi\right|^{-1}|dz|.$$
As $\frac{\partial\psi}{\partial v_z}>0$ on $\partial D$, $\psi=0$ on $\partial D$ and $\tilde\rho(z)\leq\liminf_{w\rightarrow z}\tilde\rho(w)$ for any $z\in\partial D$,  it follows from Fatou's Lemma and Lemma \ref{l:3} that 
\begin{equation}
	\nonumber \begin{split}
		&\int_{\partial D}|f|^2\tilde\rho\left|\bigtriangledown\psi\right|^{-1}|dz|\\
		\le&\liminf_{r\rightarrow1-0}\frac{\int_{\log r}^0\left(\int_{\{z\in D:\psi(z)=t\}}|f|^2\rho\left|\bigtriangledown\psi\right|^{-1}|dz|\right)dt}{-\log r}\\
		=&\liminf_{r\rightarrow1-0}\frac{\int_{\{z\in D:\psi(z)\ge\log r\}}|f|^2\tilde\rho}{1-r}\times\frac{1-r}{-\log r}\\
		=&\liminf_{r\rightarrow1-0}\frac{\int_{\{z\in D:\psi(z)\ge\log r\}}|f|^2\tilde\rho}{1-r}.
	\end{split}
\end{equation}
Thus, inequality \eqref{eq:221201a} holds.
\end{proof}

\subsection{The Hardy space over $\partial M$}
\label{sec:hardy}

Let $D_j$ be a planar regular region with finite boundary components which are analytic Jordan curves   for any $1\le j\le n$. Let $$M=\prod_{1\le j\le n}D_j$$ be a bounded domain in $\mathbb{C}^n$. In this section, we recall and give some properties on the Hardy space over $\partial M$, which will be used in the proofs of the main theorems.

\subsubsection{Some results on $H^2_{\rho}(M,\partial D_j\times M_j)$}
\

 Let  $M_j=\prod_{1\le l\le n,l\not=j}D_l$, then $M=D_j\times M_j$. Let $z_j\in D_j$ for any $1\le j\le n$. Recall that $H^2(D_j)$ denotes  the Hardy space on $D_j$ and there exists a norm-preserving linear map  $\gamma_j:H^2(D_j)\rightarrow L^2(\partial D_j,\frac{\partial G_{D_j}(z,z_j)}{\partial v_z})$ (see Section \ref{sec:2.1}) satisfying that $\gamma_j(f)$ denotes the nontangential boundary value of $f$ a.e. on $\partial D_j$ for any $f\in H^2(D_j)$, where $G_{D_j}(\cdot,z_j)$ is the Green function on $D_j$.

Let $d\mu_j$ be the Lebesgue measure on $M_j$ for any $1\le j\le n$, and let $d\mu$ be a measure on $\partial M$ defined by
$$\int_{\partial M}hd\mu=\sum_{1\le j\le n}\frac{1}{2\pi}\int_{M_j}\int_{\partial D_j}h(w_j,\hat w_j)|dw_j|d\mu_j(\hat w_j)$$
 for any $h\in L^1(\partial M)$, where $\hat w_j:=(w_1,\ldots,w_{j-1},w_{j+1},\ldots,w_n)$. For simplicity, denote $d\mu|_{\partial D_j\times M_j}$ by $d\mu$. Let us consider a space over $\partial D_j\times M_j$.
Denote 
\begin{displaymath}\begin{split}
	\{f\in L^2(\partial D_j\times M_j,d\mu):\exists f^*\in \mathcal{O}(M)\text{, s.t.  }&f^*(\cdot,\hat w_j)\in H^2(D_j) \text{ for any  $\hat w_j\in M_j$}\\
	&\,\&\, f=\gamma_j(f^*) \text{ a.e. on } \partial D_j\times M_j\}	
\end{split}\end{displaymath}
by  $H^2(M,\partial D_j\times M_j)$. In \cite{GY-saitohprodct}, we proved that there exists a unique linear injective map $P_{\partial M,j}$ from $H^2(M,\partial D_j\times M_j)$ to $\mathcal{O}(M)$ such that $P_{\partial M,j}(f)$ (denoted by $f^*$ for simplicity) satisfies the following conditions for any $f\in H^2(M,\partial D_j\times M_j)$:
	
	$(1)$ $P_{\partial M,j}(f)(\cdot,\hat w_j)\in H^2(D_j)$ for any  $\hat w_j\in M_j$;

	$(2)$ $f=\gamma_j(P_{\partial M,j}(f))$ a.e. on $\partial D_j\times M_j$.
	
	\

Let $\rho$ be a Lebesgue measurable function on $\partial M$ such that $\inf_{\partial M}\rho>0$. Denote that  
$$\ll f,g\gg_{\partial D_j\times M_j,\rho}:=\frac{1}{2\pi}\int_{M_j}\int_{\partial D_j}f(w_j,\hat w_j)\overline{g(w_j,\hat w_j)}\rho |dw_j|d\mu_j(\hat w_j)$$
for any $f,g\in L^2(\partial D_j\times M_j,\rho d\mu)\subset L^2(\partial D_j\times M_j,d\mu)$.
 Denote that
 \begin{displaymath}
 	H^2_{\rho}(M,\partial D_j\times M_j):=\{f\in H^2(M,\partial D_j\times M_j):\|f\|_{\partial D_j\times M_j,\rho}<+\infty \}.
 \end{displaymath}
 $H^2_{\rho}(M,\partial D_j\times M_j)$ is a Hilbert space equipped with the inner product $\ll \cdot,\cdot\gg_{\partial D_j\times M_j,\rho}$ (see \cite{GY-saitohprodct}).
 
 We recall the following lemma.
 \begin{Lemma}[\cite{GY-saitohprodct}]
 	\label{l:b1-p}For any compact subset $K$ of $M$, there exists a positive constant $C_K$ such that 
 	$$|f^*(z)|\le C_K\|f\|_{\partial D_j\times M_j,\rho}$$
 	holds for any $z\in K$ and $f\in H_{\rho}^2(M,\partial D_j\times M_j)$.
 \end{Lemma}

In the following, assume that $\rho|_{\partial D_1\times M_1}=\rho_1\times\lambda_1$,  where $\rho_1$ is a positive Lebesgue measurable function on $\partial D_1$, and $\lambda_1$ is a positive Lebesgue measurable function on $M_1$ such that $A^2(M_1,\lambda_1):=\{f\in\mathcal{O}(M_1):\int_{M_1}|f|^2\lambda_1<+\infty\}$ is a Hilbert space with the inner $\ll f,g\gg_{M_1,\lambda_1}:=\int_{M_1}f\overline g\lambda_1$, i.e., $\lambda_1$ is an admissible weight on $M_1$ (see Section \ref{sec:other}).
\begin{Lemma}[\cite{GY-saitohprodct}]\label{l:prod-d1xm1}
	 Assume that $H^2_{\rho}(M,\partial D_1\times M_1)\not=\{0\}$. Then we have $H^2_{\rho_1}(D_1,\partial D_1)\not=\{0\}$ and $A^2(M_1,\lambda_1)\not=\{0\}$. Furthermore,
	$\{e_m(z)\tilde e_l(w)\}_{m,l\in\mathbb{Z}_{>0}}$ is a complete orthonormal basis for $H^2_{\rho}(M,\partial D_1\times M_1)$, where $\{e_m\}_{m\in\mathbb{Z}_{>0}}$ is a complete orthonormal basis for $H^2_{\rho_1}(D_1,\partial D_1)$, and $\{\tilde e_m\}_{m\in\mathbb{Z}_{>0}}$ is a complete orthonormal basis for $A^2(M_1,\lambda_1)$. 
\end{Lemma}

Let $p_{j,k}$ be positive real number for any $1\le j\le n$ and $1\le k\le m_j$. Let $\psi_1=\sum_{1\le k\le m_1}p_{1,k}G_{D_1}(\cdot,z_{1,k})$ on $D_1$, and let $\hat\psi_1=\max_{2\le j\le n}\{\sum_{1\le k\le m_j}p_{j,k}G_{D_j}(\cdot,z_{j,k})\}$ on $M_1$.	

\begin{Lemma} \label{l:0-4v3}Assume that $\rho_1$ is a positive  Lebesgue measurable function on $U\cap \overline D_1$ satisfying that $\lim_{z\rightarrow\tilde z}\rho_1(z)=\rho_1(\tilde z)$ for any $\tilde z\in\partial D_1$, where $U$ is a neighborhood of $\partial D_1$. Then 
\begin{equation}
	\label{eq:0213d}\|f\|_{\partial D_1\times M_1,\rho}=\frac{1}{2\pi}\lim_{r\rightarrow1-0}\int_{M_{1,r}}\int_{\partial D_{1,r}}|f^*(w_1,\hat w_1)|^2\rho_1(w_1)|dw_1|\lambda_1(\hat w_1)d\mu_1(\hat w_1)
\end{equation}
	holds for any $f\in H_{\rho}^2(M,\partial D_1\times M_1)$, where $D_{1,r}=\{z\in \overline D_1:\psi_1(z)<\log r\}$ and $M_{1,r}=\{z\in M_1:\hat\psi_1(z)<\log r\}$ for $r\in[0,1]$.
\end{Lemma}
\begin{proof}
	Following from Lemma \ref{l:green-sup2}, we know that 
	\begin{equation}
		\label{eq:02100}
		\psi_1-\log r=\sum_{1\le k\le m_1}p_{1,k}G_{D_{1,r}}(\cdot,z_{1,k})
	\end{equation}
	 on $D_{1,r}$, where $G_{D_{1,r}}(\cdot,z_{1,k})$ is the Green function on $D_{1,r}$. For any holomorphic function $h$ on $D_1$ and any $\tilde z\in D_1$, as $|h|^2$ is subharmonic, we know that $\int_{\partial D_{1,r}}|h(z)|^2\frac{G_{D_{1,r}}(z,\tilde z)}{\partial v_z} |dz|$ is increasing with respect to $r$. Combining Lemma \ref{l:0-4}, we have 
	$$\lim_{r\rightarrow+\infty}\int_{\partial D_{1,r}}|h(z)|^2\frac{G_{D_{1,r}}(z,\tilde z)}{\partial v_z} |dz|=\int_{\partial D_{1}}|h(z)|^2\frac{G_{D_{1}}(z,\tilde z)}{\partial v_z} |dz|$$
	for $h\in H^2(D_1)$.  Thus, it follows from Levi's Theorem and equality \eqref{eq:02100} that 
	\begin{equation}
		\label{eq:0210a}
		\begin{split}
					&\lim_{r\rightarrow1-0}\int_{M_{1,r}}\int_{\partial D_{1,r}}|\tilde f^*(w_1,\hat w_1)|^2\frac{\partial\psi_1}{\partial v_z}|dw_1|\lambda_{1}(\hat w_1)d\mu_1(\hat w_1) \\
					=&\lim_{r\rightarrow1-0}\sum_{1\le k\le m_1}p_{1,k}\int_{M_{1,r}}\int_{\partial D_{1,r}}|\tilde f^*(w_1,\hat w_1)|^2\frac{\partial G_{D_{1,r}}(w_1,z_{1,k})}{\partial v_z}|dw_1|\lambda_{1}(\hat w_1)d\mu_1(\hat w_1)\\
					=&\sum_{1\le k\le m_1}p_{1,k}\int_{M_1}\int_{\partial D_{1}}|\tilde f(w_1,\hat w_1)|^2\frac{\partial G_{D_{1}}(w_1,z_{1,k})}{\partial v_z}|dw_1|\lambda_{1}(\hat w_1)d\mu_1(\hat w_1)\\
					=&\int_{M_1}\int_{\partial D_{1}}|\tilde f(w_1,\hat w_1)|^2\frac{\partial \psi_1}{\partial v_z}\lambda_{1}(\hat w_1)|dw_1|d\mu_1(\hat w_1)
		\end{split}
			\end{equation}
	holds for any $\tilde f\in H_{\frac{\partial \psi_1}{\partial v_z}\lambda_{1}}^2(M,\partial D_1\times M_1)$. 
There exist positive numbers $L_1$ and $r_0\in[0,1]$ such that 
	\begin{displaymath}
		\frac{1}{L_1}<\inf_{\{z\in\overline D_1:\psi_1(z)\ge\log r_0\}}\min\{|\bigtriangledown\psi_1|,\rho_1\}\le\sup_{\{z\in\overline D_1:\psi_1(z)\ge\log r_0\}}\max\{|\bigtriangledown\psi_1|,\rho_1\}<L_1.
	\end{displaymath}

		By Lemma \ref{l:prod-d1xm1}, there exist $\{f_l\}_{l\in\mathbb{Z}_{>0}}\subset H_{\rho_1}^2(D_1,\partial D_1)$ and $\{g_l\}_{l\in\mathbb{Z}_{>0}}\subset A^2(M_1,\lambda_1)$ such that 
	\begin{equation}
		\label{eq:0213a}f=\sum_{l=1}^{+\infty}f_lg_l.
	\end{equation}
	Denote that $F_m:=\sum_{l=m+1}^{+\infty}f_lg_l\in H^2_{\rho}(M,\partial D_1\times M_1)$.
	 It follows from equality \eqref{eq:0210a} and Lemma \ref{l:4} that 
\begin{equation}
\label{eq:0213b}
\begin{split}
	&\limsup_{r\rightarrow1-0}\int_{M_{1,r}}\int_{\partial D_{1,r}}|F^*_m(w_1,\hat w_1)|^2\rho_1(w_1)|dw_1|\lambda_{1}(\hat w_1)d\mu_1(\hat w_1)\\
	\le& L_1^2\limsup_{r\rightarrow1-0}\int_{M_{1,r}}\int_{\partial D_{1,r}}|F^*_m(w_1,\hat w_1)|^2\frac{\partial\psi_1}{\partial v_z}|dw_1|\lambda_{1}(\hat w_1)d\mu_1(\hat w_1)\\
	\le& L_1^2\int_{M_1}\int_{\partial D_{1}}|F_m(w_1,\hat w_1)|^2\frac{\partial \psi_1}{\partial v_z}\lambda_{1}(\hat w_1)|dw_1|d\mu_1(\hat w_1)\\
	\le& L_1^4\int_{M_1}\int_{\partial D_{1}}|F_m(w_1,\hat w_1)|^2\rho|dw_1|d\mu_1(\hat w_1).
\end{split}
\end{equation}
Using Lemma \ref{l:0-4v2}, we have 
\begin{equation}
\label{eq:0213c}
\begin{split}
	&\limsup_{r\rightarrow1-0}\left(\int_{M_{1,r}}\int_{\partial D_{1,r}}|\sum_{1\le l\le m}f_l(w_1)g_l(\hat w_1)|^2\rho_1(w_1)|dw_1|\lambda_{1}(\hat w_1)d\mu_1(\hat w_1)\right)^{\frac{1}{2}}\\
	\le&\limsup_{r\rightarrow1-0}\sum_{1\le l\le m}\left(\int_{M_{1,r}}\int_{\partial D_{1,r}}|f_l(w_1)g_l(\hat w_1)|^2\rho_1(w_1)|dw_1|\lambda_{1}(\hat w_1)d\mu_1(\hat w_1)\right)^{\frac{1}{2}}\\
	=&\limsup_{r\rightarrow1-0}\sum_{1\le l\le m}\left(\int_{M_{1,r}}|g_l(\hat w_1)|^2\lambda_{1}(\hat w_1)d\mu_1(\hat w_1) \int_{\partial D_{1,r}}|f_l(w_1)|^2\rho_1(w_1)|dw_1|\right)^{\frac{1}{2}}\\
	=&\sum_{1\le l\le m}\left(\int_{M_{1}}|g_l(\hat w_1)|^2\lambda_{1}(\hat w_1)d\mu_1(\hat w_1) \int_{\partial D_{1}}|f_l(w_1)|^2\rho_1(w_1)|dw_1|\right)^{\frac{1}{2}}\\
	=&\left(\int_{M_{1}}\int_{\partial D_{1}}|\sum_{1\le l\le m}f_l(w_1)g_l(\hat w_1)|^2\rho_1(w_1)|dw_1|\lambda_{1}(\hat w_1)d\mu_1(\hat w_1)\right)^{\frac{1}{2}}
	\end{split}
\end{equation}
 for any $m\in\mathbb{Z}_{>0}$. 
Following from equality \eqref{eq:0213a}, inequality \eqref{eq:0213b} and \eqref{eq:0213c}, we have
\begin{displaymath}
	\begin{split}
&\limsup_{r\rightarrow1-0}\left(\int_{M_{1,r}}\int_{\partial D_{1,r}}|f^*(w_1,\hat w_1)|^2\rho_1(w_1)|dw_1|\lambda_1(\hat w_1)d\mu_1(\hat w_1)\right)^{\frac{1}{2}}
		\\\le&\liminf_{m\rightarrow+\infty}\Bigg(\limsup_{r\rightarrow1-0}\Big(\int_{M_{1,r}}\int_{\partial D_{1,r}}|\sum_{1\le l\le m}f_lg_l|^2\rho_1(w_1)|dw_1|\lambda_1(\hat w_1)d\mu_1(\hat w_1)\Big)^{\frac{1}{2}}\\
		&+\limsup_{r\rightarrow1-0}\Big(\int_{M_{1,r}}\int_{\partial D_{1,r}}|F_m^*|^2\rho_1(w_1)|dw_1|\lambda_1(\hat w_1)d\mu_1(\hat w_1)\Big)^{\frac{1}{2}}\Bigg)\\
		\le &\liminf_{m\rightarrow+\infty}\Bigg(\Big(\int_{M_{1}}\int_{\partial D_{1}}|\sum_{1\le l\le m}f_l(w_1)g_l(\hat w_1)|^2\rho_1(w_1)|dw_1|\lambda_{1}(\hat w_1)d\mu_1(\hat w_1)\Big)^{\frac{1}{2}}\\
		&+\Big(L_1^4\int_{M_1}\int_{\partial D_{1}}|F_m(w_1,\hat w_1)|^2\rho|dw_1|d\mu_1(\hat w_1)\Big)^{\frac{1}{2}}\Bigg)\\
		= &(2\pi)^{\frac{1}{2}}\|f\|_{\partial D_1\times M_1,\rho}.
	\end{split}
\end{displaymath}
By Fatou's Lemma, we have 
$$2\pi\|f\|^2_{\partial D_1\times M_1,\rho}\le\liminf_{r\rightarrow1-0}\int_{M_{1,r}}\int_{\partial D_{1,r}}|f^*(w_1,\hat w_1)|^2\rho_1(w_1)|dw_1|\lambda_1(\hat w_1)d\mu_1(\hat w_1).$$
Thus, equality \eqref{eq:0213d} holds.
\end{proof}

\subsubsection{Some results on $H^2_{\rho}(M,\partial M)$}

Let $\rho$ be a Lebesgue measurable function on $\partial M$ such that $\inf_{\partial M}\rho>0$. Denote that  
$$\ll f,g\gg_{\partial M,\rho}:=\sum_{1\le j\le n}\frac{1}{2\pi}\int_{M_j}\int_{\partial D_j}f(w_j,\hat w_j)\overline{g(w_j,\hat w_j)}\rho |dw_j|d\mu_j(\hat w_j)$$
for any $f,g\in L^2(\partial M,\rho d\mu)\subset L^2(\partial M,d\mu)$.
 The weighted Hardy space over $\partial M$ is defined as follows:
 
 \emph{For any $f\in L^2(\partial M,\rho d\mu)$, we call $f\in H^2_{\rho}(M,\partial M)$ if $f\in H^2_{\rho}(M,\partial D_j\times M_j)$ for any $1\le j\le n$ and $P_{\partial M,j}(f)=P_{\partial M,k}(f)$ for any $j\not=k$.} 
 
Denote that $P_{\partial M}(f):=P_{\partial M,j}(f)$ for any $f\in H^2_{\rho}(M,\partial M)$ (denote also by $f^*$ for simplicity), and $P_{\partial M}$ is a linear injective map from $H^2_{\rho}(M,\partial M)$ to $\mathcal{O}(M)$. 
$H^2_{\rho}(M,\partial M)$ is a Hilbert space with the inner product $\ll\cdot,\cdot\gg_{\partial M,\rho}$ (see \cite{GY-saitohprodct}).
Let $Z_0$ be any subset of $M$, and let $J_z$ be an ideal of $\mathcal{O}_z$ for any $z\in Z_0$. Let $f_0$ be a holomorphic function on a neighborhood of $Z_0$.
Denote that
\begin{displaymath}
	\begin{split}
		M_H(Z_0,J,\rho):=\inf\bigg\{&\|f\|^2_{\partial M,\rho}:f\in H_{\rho}^2(M,\partial M) 
		\\ &\text{ s.t. } (f^*-f_0,z)\in J_{z}\text{ for any $z\in Z_0$}\bigg\}.\end{split}
\end{displaymath}

The following Lemma will be used in the proof of Lemma \ref{l:exists of $M_H$2}.

\begin{Lemma}[see \cite{G-R}]
\label{l:closedness}
Let $N$ be a submodule of $\mathcal O_{\mathbb C^n,o}^q$, $1\leq q<+\infty$, and let $f_j\in\mathcal O_{\mathbb C^n}(U)^q$ be a sequence of $q-$tuples holomorphic in an open neighborhood $U$ of the origin $o$. Assume that the $f_j$ converge uniformly in $U$ towards  a $q-$tuple $f\in\mathcal O_{\mathbb C^n}(U)^q$, assume furthermore that all germs $(f_{j},o)$ belong to $N$. Then $(f,o)\in N$.	
\end{Lemma}

\begin{Lemma}
	\label{l:exists of $M_H$2}Assume that $M_H(Z_0,J,\rho)<+\infty$. Then  there is a unique $f\in H^2_{\rho}(M,\partial M)$ satisfying that $(f^*-f_0,z)\in J_z$ for any $z\in Z_0$ and $M_H(Z_0,J,\rho)=\|f\|_{\partial M,\rho}^2$.\end{Lemma}
\begin{proof}
Firstly, we prove the existence of $f$.
As $M_H(Z_0,J,\rho)<+\infty,$ there is $\{f_{j}\}_{j\in\mathbb{Z}_{>0}}\subset H^2_{\rho}(M,\partial M)$  such that
$\lim_{j\rightarrow+\infty}\|f_j\|^2_{\partial M,\rho}=M_H(Z_0,J,\rho)<+\infty$ 
and $(f^*_{j}-f_0,z)\in J_z$ for any $z\in Z_0$ and any $j$. 
Then there is a subsequence of $\{f_j\}_{j\in\mathbb{Z}_{>0}}$ denoted also by $\{f_j\}_{j\in\mathbb{Z}_{>0}}$, which weakly converges to an element $f\in H_{\rho}^2(M,\partial M)$, i.e.,
	\begin{equation}\label{eq:0803a}
		\lim_{j\rightarrow+\infty}\ll f_j,g\gg_{\partial M,\rho}=\ll f,g\gg_{\partial M,\rho}
	\end{equation}
holds for any $g\in H^2_{\rho}(M,\partial M)$. Hence we have 
\begin{equation}
	\label{eq:220807g}\|f\|_{\partial M,\rho}^2\le \lim_{j\rightarrow+\infty}\|f_j\|^2_{\partial M,\rho}=M_H(Z_0,J,\rho).
\end{equation}
It follows from Lemma \ref{l:b1-p} that there is a subsequence of $\{f_j\}_{j\in\mathbb{Z}_{>0}}$ denoted also by $\{f_j\}_{j\in\mathbb{Z}_{>0}}$, which satisfies that $f_j^*$ uniformly converges to a holomorphic function $g_0$ on $M$ on any compact subset of $M$. 
Following from Lemma \ref{l:b1-p}, for any $z\in M$, there exists $g_z\in H^2_{\rho}(M,\partial M)$ such that 
\begin{equation}
	\label{eq:0217a}\ll g,g_z\gg_{\partial M,\rho}=g(z)
\end{equation}
holds for any $g\in H^2_{\rho}(M,\partial M)$.
By equality \eqref{eq:0803a} and \eqref{eq:0217a}, we get that 
$$\lim_{j\rightarrow+\infty}f_j^*(z)=f^*(z)$$
for any $z\in M$, hence we know that $f^*=g_0$ and $f_j^*$ uniformly converges to $f^*$ on any compact subset of $M$. Following from Lemma \ref{l:closedness} and $(f^*_{j}-f_0,z)\in J_z$ for any $z\in Z_0$ and any $j$, we get 
$$(f^*-f_0,z)\in J_z$$
for any $z\in Z_0$.
By definition of $M_H(Z_0,J,\rho)$ and inequality \eqref{eq:220807g}, we have 
$$\|f\|_{\partial M,\rho}^2=M_H(Z_0,J,\rho).$$
Thus,  we obtain the existence of $f$.

Now, we prove the uniqueness of $f$ by contradiction:
if not, there exist two different $g_{1}\in H^2_{\rho}(M,\partial M)$ and $g_{2}\in H^2_{\rho}(M,\partial M)$ satisfying that $\|g_1\|_{\partial M,\rho}^2=\|g_1\|_{\partial M,\rho}^2=M_H(Z_0,J,\rho)$, 
$(g_{1}^*-f_0,z)\in J_z$ and $(g_{2}^*-f_0,z)\in J_z$ for any $z\in Z_0$. It is clear that 
$$(\frac{g^*_{1}+g^*_{2}}{2}-f_0,z)\in J_z.$$
Note that
\begin{equation}\nonumber
\begin{split}
\|\frac{g_1+g_2}{2}\|_{\partial M,\rho}^2+\|\frac{g_1-g_2}{2}\|_{\partial M,\rho}^2=
\frac{\|g_1\|_{\partial M,\rho}^2+\|g_2\|_{\partial M,\rho}^2}{2}=M_H(Z_0,J,\rho),
\end{split}
\end{equation}
then we obtain that
$$\|\frac{g_1+g_2}{2}\|_{\partial M,\rho}^2<M_H(Z_0,J,\rho),$$
 which contradicts the definition of $M_H(Z_0,J,\rho)$.

Thus, Lemma \ref{l:exists of $M_H$2} has been proved.
\end{proof}

Let $Z_j=\{z_{j,1},z_{j,2},...,z_{j,m_j}\}\subset D_j$ for any  $j\in\{1,2,...,n\}$, where $m_j$ is a positive integer. Denote that 
$$Z_0:=\prod_{1\le j\le n}Z_j\subset M.$$
Let $\psi=\max_{1\le j\le n}\{\sum_{1\le k\le m_j}p_{j,k}G_{D_j}(\cdot,z_{j,k})\}$.

Let $\hat\rho$ be a Lebesgue measurable function on $\overline M$, which satisfies that $\inf_{\overline{M}}\hat\rho>0$ and $\hat\rho(w_j,\hat w_j)\leq \liminf_{w\rightarrow w_j}\hat\rho(w,\hat w_j)$ for any $(w_j,\hat w_j)\in \partial D_j\times M_j\subset \partial M$ and any $1\le j\le n$, where $M_j=\prod_{l\not=j}D_l$. Let $\rho$ be a Lebesgue measurable function on $\partial M$ such that
$$\rho(w_1,\ldots,w_n):=\left(\sum_{1\le k\le m_j}p_{j,k}\frac{\partial G_{D_j}(w_j,z_{j,k})}{\partial v_{w_j}}\right)^{-1}\hat\rho$$
on $\partial D_j\times {M_j}$ for any $1\le j\le n$, thus we have $\inf_{\partial M}\rho>0$.

The following proposition gives an sufficient condition for $f\in H_{\rho}^2(M,\partial M)$.

\begin{Proposition}
	\label{p:b6}Let $g$ be a holomorphic function on $M$. Assume that 
	$$\liminf_{r\rightarrow1-0}\frac{\int_{\{z\in M:2\psi(z)\ge\log r\}}|g(z)|^2\hat\rho}{1-r}<+\infty,$$
	then there is $f\in H^2_{\rho}(M,\partial M)$ such that $f^*=g$ and 
	$$\|f\|_{\partial M,\rho}^2\le \frac{1}{\pi} \liminf_{r\rightarrow1-0}\frac{\int_{\{z\in M:2\psi(z)\ge\log r\}}|g(z)|^2\hat\rho}{1-r}.$$
\end{Proposition}

\begin{proof}If $Z_0$ is a single point set, Proposition \ref{p:b6} can be referred to \cite{GY-saitohprodct}. 
Let $\tilde\psi(w_1,\ldots,w_n)=\max_{1\le j\le n}\{2G_{D_j}(w_j,z_{j,1})\}$ on $M$. 
	There exist $t_0>0$ and $C$ such that 
	$$\{z\in M:\tilde\psi(z)\ge-t\}\subset\{z\in M:\psi(z)\ge-Ct\}$$
	for any $t\in(0,t_0)$ (see \cite{GY-weightedsaitoh}), which implies that 
	$$\liminf_{r\rightarrow1-0}\frac{\int_{\{z\in M:\tilde\psi(z)\ge\log r\}}|g(z)|^2\hat\rho}{1-r}<+\infty.$$
	As Proposition \ref{p:b6} holds when $Z_0$ is a single point set,  there is $f\in H^2(M,\partial M)$ such that $f^*=g$.

In the following part, we will prove that 
$$\sum_{1\le j\le n}\int_{M_j}\int_{\partial D_j}|f|^2\rho|dw_j|d\mu_j(\hat w_j)\le 2\liminf_{r\rightarrow1-0}\frac{\int_{\{z\in M:2\psi(z)\ge\log r\}}|g(z)|^2\hat\rho}{1-r}.$$
	
Choose any compact subset $K_j$ of $M_j$ for any $1\le j\le n$, and denote that
$$\Omega_{j,r}:=\{z\in D_j:2\sum_{1\le k\le m_j}p_{j,k}G_{D_j}(z,z_{j,k})\geq \log r\}\times K_j\subset \{z\in M:2\psi(z)\ge \log r\}$$
for any $1\le j\le n$. There exists $r_1\in(0,1)$ such that $\Omega_{j,r_1}\cap\Omega_{j',r_1}=\emptyset$ for any $j\not=j'$ and $|\bigtriangledown\psi_j|\not=0$ on $\{z\in D_j:2\psi_j\ge\log r\}$, where $\psi_j=\sum_{1\le k\le m_j}p_{j,k}G_{D_j}(\cdot,z_{j,k})$ on $D_j$.
Note that $f(\cdot,\hat w_1)=\gamma_1(g(\cdot,\hat w_1))$ denotes the nontangential boundary value of $g(\cdot,\hat w_1)$ a.e. on $\partial D_1$ for any $\hat w_1\in M_1$. As
 $$\hat\rho(w_1,\hat w_1)\leq \liminf_{w\rightarrow w_1}\hat\rho(w,\hat w_1)$$
  for any $(w_1,\hat w_1)\in \partial D_1\times M_1$,  
it follows from Fatou's lemma, Lemma \ref{l:4} and Lemma \ref{l:3} that
\begin{equation}
\nonumber \begin{split}
	&\int_{K_1}\int_{\partial D_1}|f(w_1,\hat w_1)|^2\rho|dw_1|d\mu_1(\hat w_1)\\
	=&\int_{K_1}\left(\int_{\partial D_1}\frac{|\gamma_1(g(\cdot,\hat w_1))|^2}{\sum_{1\le k\le m_j}p_{1,k}\frac{\partial G_{D_1}(w_1,z_{1,k})}{\partial v_{w_1}}}\hat\rho|dw_1|\right)d\mu_1(\hat w_1)\\
	\le& \liminf_{r\rightarrow1-0}\frac{\int_{\log r}^0\left(\int_{K_1}\left(\int_{\{\psi_1=s\}}\frac{|g|^2\hat\rho}{|\bigtriangledown \psi_1|}|dw_1|\right)d\mu_1(\hat w_1)\right)ds}{-\log r}\\
	=&\liminf_{r\rightarrow1-0}\frac{\int_{\Omega_{1,r^2}}|g|^2\hat\rho}{-\log r}\\
	=&2\liminf_{r\rightarrow1-0}\frac{\int_{\Omega_{1,r}}|g|^2\hat\rho}{-\log r}.
\end{split}
\end{equation}	
By similar discussion, we have 
\begin{equation}
	\label{eq:220731a}
	\int_{K_j}\int_{\partial D_j}|f(w_j,\hat w_j)|^2\rho|dw_j|d\mu_j(\hat w_j)
	\le2\liminf_{r\rightarrow1-0}\frac{\int_{\Omega_{j,r}}|g|^2\hat\rho}{-\log r}
\end{equation} 
for any $1\le j\le n$. As $\Omega_{j,r}\cap\Omega_{j',r}=\emptyset$ for any $j\not=j'$ and $r\in(r_1,1)$, following from the arbitrariness of $K_j$ and inequality \eqref{eq:220731a} that 
\begin{equation}\nonumber
	\begin{split}
		&\sum_{1\le j\le n}	\int_{M_j}\int_{\partial D_j}|f(w_j,\hat w_j)|^2\rho|dw_j|d\mu_j(\hat w_j)\\
		\le&2\liminf_{r\rightarrow1-0}\frac{\int_{\{z\in M:2\psi(z)\ge\log r\}}|g|^2\hat\rho}{-\log r}\\
		=&2\liminf_{r\rightarrow1-0}\frac{\int_{\{z\in M:2\psi(z)\ge\log r\}}|g|^2\hat\rho}{1-r}\\
		<&+\infty.
	\end{split}
\end{equation}

Thus, Proposition \ref{p:b6} holds.
\end{proof}

The following lemma will be used in the proof of Lemma \ref{l:integral}.

\begin{Lemma}[\cite{skoda1972}]
\label{l:skoda}
	Let $u$ is a subharmonic function on $\Omega$. If $v(dd^cu,z_0)<1$, then $e^{-2u}$ is $L^1$ on a neighborhood of $z_0$.
\end{Lemma}	
	
 Let $\varphi_j$ be a subharmonic function on $D_j$, which satisfies that  $\varphi_j$ is continuous at $z$ for any $z\in \partial D_j$. Let $\rho=\prod_{1\le j\le n}e^{-\varphi_j}$ on $\overline M$.
 \begin{Lemma}
 	\label{l:integral}
 	Assume that $n>1$. Let $f\in H_{\rho}^2(M,\partial M)$. Then for any compact subset $K$ of $M$, we have $\int_{K}|f^*|^2e^{-\varphi}<+\infty$.
 \end{Lemma}
	\begin{proof}
	Since $\varphi_j$ is continuous at $z$ for any $z\in \partial D_j$, for any $j$, it follows from Weierstrass theorem (see \cite{OF81}) and Siu's Decomposition Theorem that there exists a holomorphic function $g_j$ on $\mathbb{C}$ such that $\varphi_j-2\log|g_j|$ is subharmonic on $D_j$ and the Lelong number
	$$v(dd^c(\varphi_j-2\log|g_j|),z)\in[0,2)$$
	holds for any $z\in D_j$. Lemma \ref{l:skoda} shows that $\prod_{1\le j\le n}e^{-(\varphi_j-2\log|g_j|)}$ is locally integrable on $M$. Thus, it suffices to prove that $\frac{f^*}{\prod_{1\le j\le n}g_j}$ is holomorphic.
	
		As $f^*(\cdot,\hat w_1)\in H^2(D_1)$ for any $\hat w_1\in M_1$ and $\gamma_1(f^*)=f$ a.e. on $\partial D_1\times M_1$, it follows from Lemma \ref{l:0-1} that for any $K_1\Subset D_1$, there is $C_{K_1}>0$ such that 
	\begin{equation}
		\label{eq:0215a}
		\sup_{w_1\in K_1}\left|f^*(w_1,\hat w_1)\right|^2\le C_{K_1}\frac{1}{2\pi} \int_{\partial D_1}\left|f(z_1,\hat w_1)\right|^2|dz_1|,
	\end{equation}
	holds for a.e. $\hat w_1\in M_1$. Note that $\inf_{M} \frac{\rho}{e^{-2\log|g_l|}}>0$. As $f\in H_{\rho}^2(M,\partial M),$ we have 
	\begin{equation}\label{eq:0330}
		\begin{split}
			&\int_{M_j}\int_{\partial D_j}\left|\frac{f^*}{g_l}\right|^2|dz_j|d\mu_j(\hat w_j)\\
			\le& C_0\int_{M_j}\int_{\partial D_j}\left|f^*\right|^2e^{-2\log|g_l|}\frac{\rho}{e^{-2\log|g_l|}}|dz_j|d\mu_j(\hat w_j)\\
			\le&C_0\|f\|^2_{\partial M,\rho}\\
			<&+\infty
		\end{split}
	\end{equation}
for any $1\le j,l\le n$.
	 Since $g_2\not\equiv0$, inequality \eqref{eq:0215a} and \eqref{eq:0330} imply that 
	\begin{equation}
\nonumber		\begin{split}&\int_{M_1}\int_{K_1}\left|\frac{f^*}{g_2}(w_1,\hat w_1)\right|^2 \\
		\le&		C_1\int_{M_1}\sup_{w_1\in K_1}\left|\frac{f^*}{g_2}(w_1,\hat w_1)\right|^2d\mu_1(\hat w_1)\\
		\le& C_1C_{K_1}\int_{M_1}\left(\frac{1}{2\pi} \int_{\partial D_1}\left|\frac{f^*}{g_2}(z_1,\hat w_1)\right|^2|dz_1|\right)d\mu_1(\hat w_1)\\
					<&+\infty
		\end{split}
	\end{equation}
for any $K_1$,
	which implies that $\frac{f^*}{g_2}$ is holomorphic on $M$.	Thus, $\frac{f^*}{\prod_{1\le j\le n}g_j}\in\mathcal{O}(M)$. 
	\end{proof}

	\subsection{Concavity property of minimal $L^2$ integrals} 
In this section, we recall some results about the concavity property of minimal $L^2$ integrals (see \cite{GY-concavity,GY-concavity3, GY-concavity4}).

Let $M$ be an $n-$dimensional Stein manifold,  and let $K_{M}$ be the canonical (holomorphic) line bundle on $M$.
Let $\psi$ be a plurisubharmonic function on $M$,
and let  $\varphi$ be a Lebesgue measurable function on $M$,
such that $\varphi+\psi$ is a plurisubharmonic function on $M$. Take $T=-\sup_{M}\psi>-\infty$.

\begin{Definition}
\label{def:gain}
We call a positive measurable function $c$  on $(T,+\infty)$ in class $\mathcal{P}_T$ if the following two statements hold:

$(1)$ $c(t)e^{-t}$ is decreasing with respect to $t$;

$(2)$ there is a closed subset $E$ of $M$ such that $E\subset \{z\in Z:\psi(z)=-\infty\}$ and for any compact subset $K\subseteq M\backslash E$, $e^{-\varphi}c(-\psi)$ has a positive lower bound on $K$, where $Z$ is some analytic subset of $M$.	
\end{Definition}

Let $Z_{0}$ be a subset of $\{\psi=-\infty\}$ such that $Z_{0}\cap Supp(\{\mathcal{O}/\mathcal{I(\varphi+\psi)}\})\neq\emptyset$.
Let $U\supseteq Z_{0}$ be an open subset of $M$
and let $f$ be a holomorphic $(n,0)$ form on $U$.
Let $\mathcal{F}\supseteq\mathcal{I}(\varphi+\psi)|_{U}$ be an  analytic subsheaf of $\mathcal{O}$ on $U$.

Denote
\begin{equation*}
\begin{split}
\inf\Bigg\{\int_{\{\psi<-t\}}|\tilde{f}|^{2}e^{-\varphi}c(-\psi):(\tilde{f}-f)\in H^{0}(Z_0,&
(\mathcal{O}(K_{M})\otimes\mathcal{F})|_{Z_0})\\&\&{\,}\tilde{f}\in H^{0}(\{\psi<-t\},\mathcal{O}(K_{M}))\Bigg\},
\end{split}
\end{equation*}
by $G(t;c)$ (without misunderstanding, we denote $G(t;c)$ by $G(t)$),  where $t\in[T,+\infty)$,  $c\in\mathcal{P}_{T}$ satisfying $\int_T^{+\infty}c(l)e^{-l}dl<+\infty$,
$|f|^{2}:=\sqrt{-1}^{n^{2}}f\wedge\bar{f}$ for any $(n,0)$ form $f$ and $(\tilde{f}-f)\in H^{0}(Z_0,
(\mathcal{O}(K_{M})\otimes\mathcal{F})|_{Z_0})$ means $(\tilde{f}-f,z_0)\in(\mathcal{O}(K_{M})\otimes\mathcal{F})_{z_0}$ for all $z_0\in Z_0$.

We recall some results about the concavity for $G(t)$.
\begin{Theorem} [\cite{GY-concavity}]
\label{thm:general_concave}
 Assume that $G(T)<+\infty$. Then $G(h^{-1}(r))$ is concave with respect to $r\in(0,\int_{T}^{+\infty}c(l)e^{-l}dl)$, $\lim_{t\rightarrow T+0}G(t)=G(T)$ and $\lim_{t\rightarrow +\infty}G(t)=0$, where $h(t)=\int_{t}^{+\infty}c(l)e^{-l}dl$.
\end{Theorem}

The following corollary gives a necessary condition for the concavity of $G(h^{-1}(r))$ degenerating to linearity.

\begin{Corollary}[\cite{GY-concavity}]	\label{c:linear}
 Assume that $G(T)\in(0,+\infty)$. If $G( {h}^{-1}(r))$ is linear with respect to $r\in(0,\int_{T}^{+\infty}c(l)e^{-l}dl)$, where $ {h}(t)=\int_{t}^{+\infty}c(l)e^{-l}dl$,
 then there is a unique holomorphic $(n,0)$ form $F$ on $M$ satisfying $(F-f)\in H^{0}(Z_0,(\mathcal{O}(K_{M})\otimes\mathcal F)|_{Z_0})$ and $G(t;c)=\int_{\{\psi<-t\}}|F|^2e^{-\varphi}c(-\psi)$ for any $t\geq T$. Furthermore,
\begin{equation}
\nonumber	\int_{\{-t_1\leq\psi<-t_2\}}|F|^2e^{-\varphi}a(-\psi)=\frac{G(T_1;c)}{\int_{T_1}^{+\infty}c(l)e^{-l}dl}\int_{t_2}^{t_1} a(t)e^{-t}dt
\end{equation}
for any nonnegative measurable function $a$ on $(T,+\infty)$, where $+\infty\geq t_1>t_2\geq T$.
\end{Corollary}

We recall the existence and uniqueness of the holomorphic $(n,0)$ form related to $G(t)$.
\begin{Lemma}[\cite{GY-concavity}]
\label{l:unique}
Assume that $G(t)<+\infty$ for some $t\in[T,+\infty)$.
Then there exists a unique holomorphic $(n,0)$ form $F_{t}$ on
$\{\psi<-t\}$ satisfying $(F_{t}-f)\in H^{0}(Z_0,(\mathcal{O}(K_{M})\otimes\mathcal{F})|_{Z_0})$ and $\int_{\{\psi<-t\}}|F_{t}|^{2}e^{-\varphi}c(-\psi)=G(t)$.
Furthermore,
for any holomorphic $(n,0)$ form $\hat{F}$ on $\{\psi<-t\}$ satisfying $(\hat{F}-f)\in H^{0}(Z_0,(\mathcal{O}(K_{M})\otimes\mathcal{F})|_{Z_0})$ and
$\int_{\{\psi<-t\}}|\hat{F}|^{2}e^{-\varphi}c(-\psi)<+\infty$,
we have the following equality
\begin{equation}
\nonumber \begin{split}
&\int_{\{\psi<-t\}}|F_{t}|^{2}e^{-\varphi}c(-\psi)+\int_{\{\psi<-t\}}|\hat{F}-F_{t}|^{2}e^{-\varphi}c(-\psi)
\\=&
\int_{\{\psi<-t\}}|\hat{F}|^{2}e^{-\varphi}c(-\psi).
\end{split}
\end{equation}
\end{Lemma}

In the following, we recall some characterizations for the concavity of $G(h^{-1}(r))$ degenerating to linearity.

Assume that $M=\Omega$ is an open Riemann surface, which admitted a nontrivial Green function $G_{\Omega}$.  Let $Z_0=\{z_1,z_2,\ldots ,z_m\}\subset\Omega$ be a finite subset of $\Omega$ satisfying that $z_j\not=z_k$ for any $j\not=k$.

We recall some notations (see \cite{OF81}, see also \cite{guan-zhou13ap,GY-concavity,GMY-concavity2}).
  Let $p:\Delta\rightarrow\Omega$ be the universal covering from unit disc $\Delta$ to $\Omega$.
 we call the holomorphic function $f$  on $\Delta$ a multiplicative function,
 if there is a character $\chi$, which is the representation of the fundamental group of $\Omega$, such that $g^{\star}f=\chi(g)f$,
 where $|\chi|=1$ and $g$ is an element of the fundamental group of $\Omega$. It is known that for any harmonic function $u$ on $\Omega$,
there exists a $\chi_{u}$ and a multiplicative function $f_u\in\mathcal{O}^{\chi_u}(\Omega)$,
such that $|f_u|=p^{\star}\left(e^{u}\right)$.
Recall that for the Green function $G_{\Omega}(z,z_j)$,
there exist a $\chi_{z_j}$ and a multiplicative function $f_{z_j}\in\mathcal{O}^{\chi_{z_j}}(\Omega)$, such that $|f_{z_j}(z)|=p^{\star}\left(e^{G_{\Omega}(z,z_j)}\right)$ (see \cite{yamada,suita}).

The following Theorem gives a characterization of the concavity of $G(h^{-1}(r))$ degenerating to linearity.
\begin{Theorem}[\cite{GY-concavity3}, see also \cite{GMY-boundary3}]
	\label{thm:m-points}
 Let $G(0)\in(0,+\infty)$ and $p_j=\frac{1}{2}v(dd^c(\psi),z_j)>0$ for any $j\in\{1,2,\ldots,m\}$. For any $j\in\{1,2,\ldots,m\}$, assume that one of the following conditions holds:
	
	$(A)$ $\varphi+a\psi$ is  subharmonic near $z_j$ for some $a\in[0,1)$;
	
	$(B)$ $(\psi-2p_jG_{\Omega}(\cdot,z_j))(z_j)>-\infty$.
	
 Then $G(h^{-1}(r))$ is linear with respect to $r$ if and only if the following statements hold:
	
	$(1)$ $\psi=2\sum_{1\le j\le m}p_jG_{\Omega}(\cdot,z_j)$;
	
	$(2)$ $\varphi+\psi=2\log|g|+2\sum_{1\le j\le m}G_{\Omega}(\cdot,z_j)+2u$ and $\mathcal{F}_{z_j}=\mathcal{I}(\varphi+\psi)_{z_j}$ for any $j\in\{1,2,\ldots,m\}$, where $g$ is a holomorphic function on $\Omega$ such that $ord_{z_j}(g)=ord_{z_j}(f)$ for any $j\in\{1,2,\ldots,m\}$ and $u$ is a harmonic function on $\Omega$;
	
	$(3)$ $\prod_{1\le j\le m}\chi_{z_j}=\chi_{-u}$, where $\chi_{-u}$ and $\chi_{z_j}$ are the  characters associated to the functions $-u$ and $G_{\Omega}(\cdot,z_j)$ respectively;
	
	$(4)$  $\lim_{z\rightarrow z_k}\frac{f}{gp_*\left(f_u\left(\prod_{1\le j\le m}f_{z_j}\right)\left(\sum_{1\le j\le m}p_{j}\frac{d{f_{z_{j}}}}{f_{z_{j}}}\right)\right)}=c_0$ for any $k\in\{1,2,\ldots,m\}$, where $c_0\in\mathbb{C}\backslash\{0\}$ is a constant independent of $k$,  $f_{u}$ is a holomorphic function on $\Delta$ such that $|f_{u}|=p^*(e^{u})$ and $f_{z_j}$ is a holomorphic function on $\Delta$ such that $|f_{z_j}|=p^*\left(e^{G_{\Omega}(\cdot,z_j)}\right)$
	\end{Theorem}

\begin{Remark}[\cite{GY-concavity3}]\label{rem:finite uniform section}
When the four statements in Theorem \ref{thm:m-points} hold,
$$c_0gp_*(f_u(\Pi_{1\le j\le m}f_{z_j})(\sum_{1\le j\le m}p_{j}\frac{d{f_{z_{j}}}}{f_{z_{j}}}))$$
 is the unique holomorphic $(1,0)$ form $F$ on $\Omega$ such that $(F-f,z_j)\in(\mathcal{O}(K_{\Omega}))_{z_j}\otimes\mathcal{F}_{z_j}$ for any $j\in\{1,2,...,m\}$ and
	$G(t)=\int_{\{\psi<-t\}}|F|^2e^{-\varphi}c(-\psi)$ for any $t\ge0$.
\end{Remark}

We recall the following characterization for the holding of the equality in the optimal $L^2$ extension problem.

\begin{Theorem}[\cite{GY-concavity3}]
\label{c:L2-1d-char}
Let $k_j$ be a nonnegative integer for any $j\in\{1,2,...,m\}$. Let $\psi$ be a negative  subharmonic function on $\Omega$ satisfying that   $\frac{1}{2}v(dd^{c}\psi,z_j)=p_j>0$ for any $j\in\{1,2,...,m\}$. Let $\varphi$ be a Lebesgue measurable function on $\Omega$  such that $\varphi+\psi$ is subharmonic on $\Omega$, $\frac{1}{2}v(dd^c(\varphi+\psi),z_j)=k_j+1$ and $\alpha_j:=(\varphi+\psi-2(k_j+1)G_{\Omega}(\cdot,z_j))(z_j)>-\infty$ for any $j$. Let $c(t)$ be a positive measurable function on $(0,+\infty)$ satisfying $c(t)e^{-t}$ is decreasing on $(0,+\infty)$ and $\int_{0}^{+\infty}c(s)e^{-s}ds<+\infty$. Let $a_j$ be a constant for any $j$.

Let $f$ be a holomorphic $(1,0)$ form on $V_0$ satisfying that $f=a_jw_j^{k_j}dw_j$ on $V_{z_j}$. Then there exists a holomorphic $(1,0)$ form $F$ on $\Omega$ such that $(F-f,z_j)\in(\mathcal{O}(K_{\Omega})\otimes\mathcal{I}(2(k_j+1)G_{\Omega}(\cdot,z_j)))_{z_j}$ and
 \begin{equation}
 	\label{eq:210902a}
 	\int_{\Omega}|F|^2e^{-\varphi}c(-\psi)\leq(\int_0^{+\infty}c(s)e^{-s}ds)\sum_{1\le j\le m}\frac{2\pi|a_j|^2e^{-\alpha_j}}{p_jc_{\beta}(z_j)^{2(k_j+1)}}.
 \end{equation}

 Moreover, equality $(\int_0^{+\infty}c(s)e^{-s}ds)\sum_{1\le j\le m}\frac{2\pi|a_j|^2e^{-\alpha_j}}{p_jc_{\beta}(z_j)^{2(k_j+1)}}=\inf\{\int_{\Omega}|\tilde{F}|^2e^{-\varphi}c(-\psi):\tilde{F}$ is a holomorphic $(1,0)$ form on $\Omega$ such that $(\tilde{F}-f,z_j)\in(\mathcal{O}(K_{\Omega})\otimes\mathcal{I}(2(k_j+1)G_{\Omega}(\cdot,z_j)))_{z_j}$ for any $j\}$ holds if and only if the following statements hold:

	$(1)$ $\psi=2\sum_{1\le j\le m}p_jG_{\Omega}(\cdot,z_j)$;
	
	$(2)$ $\varphi+\psi=2\log|g|+2\sum_{1\le j\le m}(k_j+1)G_{\Omega}(\cdot,z_j)+2u$, where $g$ is a holomorphic function on $\Omega$ such that $g(z_j)\not=0$ for any $j\in\{1,2,...,m\}$ and $u$ is a harmonic function on $\Omega$;
	
	$(3)$ $\Pi_{1\le j\le m}\chi_{z_j}^{k_j+1}=\chi_{-u}$, where $\chi_{-u}$ and $\chi_{z_j}$ are the  characters associated to the functions $-u$ and $G_{\Omega}(\cdot,z_j)$ respectively;
	
	$(4)$  $\lim_{z\rightarrow z_k}\frac{f}{gp_*(f_u(\Pi_{1\le j\le m}f_{z_j}^{k_j+1})(\sum_{1\le j\le m}p_{j}\frac{d{f_{z_{j}}}}{f_{z_{j}}}))}=c_0$ for any $k\in\{1,2...,m\}$, where $c_0\in\mathbb{C}\backslash\{0\}$ is a constant independent of $k$.
\end{Theorem}

In the following, we consider the case $M$ is a product manifold of open Riemann surfaces.

Let $\Omega_j$  be an open Riemann surface, which admits a nontrivial Green function $G_{\Omega_j}$ for any  $1\le j\le n$. Let 
$$M=\prod_{1\le j\le n}\Omega_j$$ be an $n-$dimensional complex manifold, and let $\pi_j$ be the natural projection from $M$ to $\Omega_j$. Let $K_M$ be the canonical (holomorphic) line bundle on $M$. Let
 $\varphi_j$ be a subharmonic function on $\Omega_j$, and let 
 $$\varphi=\sum_{1\le j\le n}\pi_j^*(\varphi_j).$$ 
Let $Z_j=\{z_{j,1},z_{j,2},...,z_{j,m_j}\}\subset \Omega_j$ for any  $j\in\{1,2,...,n\}$, where $m_j$ is a positive integer. Denote that 
$$Z_0:=\prod_{1\le j\le n}Z_j\subset M.$$
Let $\psi=\max_{1\le j\le n}\{\pi_j^*(2\sum_{1\le k\le m_j}p_{j,k}G_{D_j}(\cdot,z_{j,k}))\}$, where $p_{j,k}>0$ is a constant. Let $\mathcal{F}_{z}=\mathcal{I}(\psi)_z$ for any $z\in Z_0$.

Let $w_{j,k}$ be a local coordinate on a neighborhood $V_{z_{j,k}}\Subset\Omega_{j}$ of $z_{j,k}\in\Omega_j$ satisfying $w_{j,k}(z_{j,k})=0$ for any $j\in\{1,2,...,n\}$ and $k\in\{1,2,...,m_j\}$, where $V_{z_{j,k}}\cap V_{z_{j,k'}}=\emptyset$ for any $j$ and $k\not=k'$. Denote that $I_1:=\{(\beta_1,\beta_2,...,\beta_n):1\le \beta_j\le m_j$ for any $j\in\{1,2,...,n\}\}$, $V_{\beta}:=\prod_{1\le j\le n}V_{z_{j,\beta_j}}$ for any $\beta=(\beta_1,\beta_2,...,\beta_n)\in I_1$ and $w_{\beta}:=(w_{1,\beta_1},w_{2,\beta_2},...,w_{n,\beta_n})$ is a local coordinate on $V_{\beta}$ of $z_{\beta}:=(z_{1,\beta_1},z_{2,\beta_2},...,z_{n,\beta_n})\in M$.

Let $f_0$ be a holomorphic $(n,0)$ form on $\cup_{\beta\in I_1}V_{\beta}$.   Let $c$ be a positive function on $[0,+\infty)$, which satisfies that $c(t)e^{-t}$ is decreasing on $[0,+\infty)$, $\lim_{t\rightarrow0+0}c(t)=c(0)=1$ and $\int_{0}^{+\infty}c(t)e^{-t}dt<+\infty$.

\begin{Theorem}[\cite{GY-concavity4}]
	\label{thm:prod-finite-point}Assume that $G(0)\in(0,+\infty)$ and $\varphi(z_{\beta})>-\infty$ for any $\beta\in I_1$.  $G(h^{-1}(r))$ is linear with respect to $r\in(0,\int_0^{+\infty} c(s)e^{-s}ds]$ if and only if the following statements hold:

	$(1)$ $\varphi_j=2\log|g_j|+2u_j$ for any $j\in\{1,2,...,n\}$, where $u_j$ is a harmonic function on $\Omega_j$ and $g_j$ is a holomorphic function on $\Omega_j$ satisfying $g_j(z_{j,k})\not=0$ for any $k\in\{1,2,...,m_j\}$;
	
	$(2)$ There exists a nonnegative integer $\gamma_{j,k}$ for any $j\in\{1,2,...,n\}$ and $k\in\{1,2,...,m_j\}$, which satisfies that $\Pi_{1\le k\leq m_j}\chi_{j,z_{j,k}}^{\gamma_{j,k}+1}=\chi_{j,-u_j}$ and $\sum_{1\le j\le n}\frac{\gamma_{j,\beta_j}+1}{p_{j,\beta_j}}=1$ for any $\beta\in I_1$, where $\chi_{-u_j}$ and $\chi_{z_{j,k}}$ are the  characters associated to the functions $-u_j$ and $G_{\Omega_j}(\cdot,z_{j,k})$ respectively;
	
	$(3)$ $f=(c_{\beta}\Pi_{1\le j\le n}w_{j,\beta_j}^{\gamma_{j,\beta_j}}+g_{\beta})dw_{1,\beta_1}\wedge dw_{2,\beta_2}\wedge...\wedge dw_{n,\beta_n}$ on $V_{\beta}$ for any $\beta\in I_1$, where $c_{\beta}$ is a constant and $g_{\beta}$ is a holomorphic function on $V_{\beta}$ such that $(g_{\beta},z_{\beta})\in\mathcal{I}(\psi)_{z_{\beta}}$;
	
	$(4)$ $\lim_{z\rightarrow z_{\beta}}\frac{c_{\beta}\Pi_{1\le j\le n}w_{j,\beta_j}^{\gamma_{j,\beta_j}}dw_{1,\beta_1}\wedge dw_{2,\beta_2}\wedge...\wedge dw_{n,\beta_n}}{\wedge_{1\le j\le n}\pi_{j}^*(g_j(P_{j})_*(f_{u_j}(\Pi_{1\le k\le m_j}f_{z_{j,k}}^{\gamma_{j,k}+1})(\sum_{1\le k\le m_j}p_{j,k}\frac{df_{z_{j,k}}}{f_{z_{j,k}}})))}=c_0$ for any $\beta\in I_1$, where $c_0\in\mathbb{C}\backslash\{0\}$ is a constant independent of $\beta$, $P_j:\Delta\rightarrow\Omega_j$ is the universal covering, $f_{u_j}$ is a holomorphic function $\Delta$ such that $|f_{u_j}|=P_j^*(e^{u_j})$ and $f_{z_{j,k}}$ is a holomorphic function on $\Delta$ such that $|f_{z_{j,k}}|=P_j^*(e^{G_{\Omega_j}(\cdot,z_{j,k})})$ for any $j\in\{1,2,...,n\}$ and $k\in\{1,2,...,m_j\}$.
\end{Theorem}

The following Lemma will be used in the proof of Remark \ref{r:var=-infty}.

\begin{Lemma}[see \cite{GY-concavity4}] \label{l:m2}
Let $\psi=\max_{1\le j\le n}\{2p_j\log|w_j|\}$ be a plurisubharmonic function on $\mathbb{C}^n$, where $p_j>0$.
	Let $f=\sum_{\alpha\in \mathbb{Z}_{\ge0}^n}b_{\alpha}w^{\alpha}$ (Taylor expansion) be a holomorphic function on $\{\psi<-t_0\}$, where $t_0>0$. Then
	$$\int_{\{\psi<-t\}}|f|^2d\lambda_n=\sum_{\alpha\in\mathbb{Z}_{\ge0}^n}e^{-\sum_{1\le j\le n}\frac{\alpha_j+1}{p_j}t}\frac{|b_{\alpha}|^2\pi^n}{\Pi_{1\le j\le n}(\alpha_j+1)}$$
	holds for any $t\ge t_0$. \end{Lemma}

\begin{Remark}
	\label{r:var=-infty}The requirement ``$\varphi(z_{\beta})>-\infty$ for any $\beta\in I_1$" in Theorem \ref{thm:prod-finite-point} can be removed.
\end{Remark}
\begin{proof}It suffices to prove that the linearity of $G(h^{-1}(r))$ can deduce $\varphi(z_{\beta})>-\infty$ for any $\beta\in I_1$. 

Assume that $G(h^{-1}(r))$ is linear with respect to $r\in(0,\int_{0}^{+\infty}c(t)e^{-t}dt]$.
	It follows from Corollary \ref{c:linear} that there is a holomorphic $(n,0)$ form $F$ on $M$ such that $(F-f,z_\beta)\in\mathcal{I}(\psi)_{z_\beta}$ for any $\beta\in I_1$, and
	\begin{equation}
		\label{eq:220807a}\int_{\{\psi<-t\}}|F|^2e^{-\varphi}=\frac{G(0)}{\int_0^{+\infty}c(t)e^{-t}dt}e^{-t}
	\end{equation}
	for any $t\ge0$.
	
	Firstly, we prove that $(F,z_\beta)\not\in \mathcal{I}(\psi)_{z_\beta}$ for any $\beta\in I_1$. We prove this by contradiction: if not, there exists $\beta_0\in I_1$ such that $(F,z_{\beta_0})\in \mathcal{I}(\psi)_{z_{\beta_0}}$. Then we have 
	$$(f,z_{\beta_0})\in \mathcal{I}(\psi)_{z_{\beta_0}}.$$
	There exists $t>0$ such that
	$\{\psi<-t\}\cap V_{{\beta_0}}\Subset V_{\beta_0}.$
 Corollary \ref{c:linear} tells us that $F$ is the unique ``minimal form" on any sublevel set of $\psi$, thus we have 
 $F\equiv0$
 on $\{\psi<-t\}\cap V_{{\beta_0}}$, which implies that 
 $$F\equiv0$$
 on $M$. Then we get that $(f,z_{\beta})\in \mathcal{I}(\psi)_{z_\beta}$ for any $\beta\in I_1$, which contradict to $G(0)>0$.
	
	Now, we prove $\varphi_j(z_{\beta})>-\infty$ for any $\beta\in I_1$. Fixed any $\beta\in I_1$, without loss of generality, assume that $|w_j(z)|=e^{\sum_{1\le k\le m_j}\frac{p_{j,k}}{p_{j,\beta_j}}G_{\Omega_j}(z,z_{j,k})}$ on $V_{z_{j,\beta_j}}$, hence $\psi=\max_{1\le j\le n}\{2p_{j,\beta_j}\log|w_j|\}$ on $V_\beta$. There is $t_0>0$ such that
	$$\{\psi<-t_0\}\cap V_{\beta}\Subset V_\beta.$$
	 Denote that 
	 $$c_t:=\sup_{\{\psi<-t\}\cap V_{\beta}}\varphi<+\infty$$
	for any $t\ge t_0$.
	As $\varphi=\sum_{1\le j\le n}\varphi_j$ is plurisubharmonic, we know that 
	$$\lim_{t\rightarrow+\infty}c_t=\varphi(z_\beta)=\sum_{1\le j\le n}\varphi_j(z_{j,\beta_j}).$$
	Let $F=\sum_{\alpha\in \mathbb{Z}_{\geq0}^n}d_{\alpha}w^{\alpha}dw_1\wedge\ldots\wedge dw_n$ near $z_\beta$. Denote that $E_\beta=\{\alpha\in \mathbb{Z}_{\geq0}^n:\sum_{1\le j\le n}\frac{\alpha_j+1}{p_{j,\beta_j}}\le 1\}$. Since $(F,z_\beta)\not\in\mathcal{I}(\psi)_{z_\beta}$, we have 
	$$\sum_{\alpha\in E_\beta}|d_{\alpha}|^2>0.$$
	  Lemma \ref{l:m2} tells us that 
	  \begin{equation}
	  	\label{eq:220807b}
	  	\begin{split}
	  			  	\int_{\{\psi<-t\}}|F|^2e^{-\varphi}&\ge e^{-c_t}\int_{\{\psi<-t\}\cap V_{\beta}}|F|^2\\
	  			  	&=e^{-c_t}\sum_{\alpha\in\mathbb{Z}_{\ge0}^n}e^{-\sum_{1\le j\le n}\frac{\alpha_j+1}{p_{j,\beta_j}}t}\frac{|d_{\alpha}|^2(2\pi)^n}{\Pi_{1\le j\le n}(\alpha_j+1)}\\
	  			  	&\geq e^{-c_t}\sum_{\alpha\in E_\beta}e^{-\sum_{1\le j\le n}\frac{\alpha_j+1}{p_{j,\beta_j}}t}\frac{|d_{\alpha}|^2(2\pi)^n}{\Pi_{1\le j\le n}(\alpha_j+1)}
	  	\end{split}
	  \end{equation} 
	  for any $t\ge t_0$. It follows from equality \eqref{eq:220807a} and inequality \eqref{eq:220807b} that 
	  \begin{equation}\label{eq:220807c}
	  \begin{split}
	  		  	\frac{G(0)}{\int_0^{+\infty}c(t)e^{-t}dt}&=\lim_{t\rightarrow+\infty}e^t \int_{\{\psi<-t\}}|F|^2e^{-\varphi}\\
	  		  	&\geq \lim_{t\rightarrow+\infty}e^{-c_t}\sum_{\alpha\in E_\beta}e^{\left(1-\sum_{1\le j\le n}\frac{\alpha_j+1}{p_{j,\beta_j}}\right)t}\frac{|d_{\alpha}|^2(2\pi)^n}{\Pi_{1\le j\le n}(\alpha_j+1)},
	  \end{split}
	  \end{equation}
	  Note that $\frac{G(0)}{\int_0^{+\infty}c(t)e^{-t}dt}\in(0,+\infty)$ and $1-\sum_{1\le j\le n}\frac{\alpha_j+1}{p_{j,\beta_j}}\ge0$ for any $\alpha\in E_\beta$, inequality \eqref{eq:220807c} shows that 
	  $$\lim_{t\rightarrow+\infty}c_t>-\infty,$$
	  hence we have $\varphi(z_{\beta})>-\infty$.
\end{proof}

Denote that
\begin{equation*}
c_{j,k}:=\exp\lim_{z\rightarrow z_{j,k}}\left(\frac{\sum_{1\le k_1\le m_j}p_{j,k_1}G_{\Omega_j}(z,z_{j,k_1})}{p_{j,k}}-\log|w_{j,k}(z)|\right)
\end{equation*}
 for any $j\in\{1,2,...,n\}$ and $k\in\{1,2,...,m_j\}$.
\begin{Remark}[\cite{GY-concavity4}]
	\label{r:1.2}When the four statements in Theorem \ref{thm:prod-finite-point} hold,
$$c_0\wedge_{1\le j\le n}\pi_{j}^*(g_j(P_{j})_*(f_{u_j}(\Pi_{1\le k\le m_j}f_{z_{j,k}}^{\gamma_{j,k}+1})(\sum_{1\le k\le m_j}p_{j,k}\frac{df_{z_{j,k}}}{f_{z_{j,k}}})))$$
 is the unique holomorphic $(n,0)$ form $F$ on $M$ such that $(F-f,z_\beta)\in(\mathcal{O}(K_{M}))_{z_\beta}\otimes\mathcal{I}(\psi)_{z_\beta}$ for any $\beta\in I_1$ and
	$$G(t)=\int_{\{\psi<-t\}}|F|^2e^{-\varphi}c(-\psi)=\left(\int_{t}^{+\infty}c(s)e^{-s}ds\right)\sum_{\beta\in I_1}\frac{|c_{\beta}|^2(2\pi)^ne^{-\varphi(z_{\beta})}}{\Pi_{1\le j\le n}(\gamma_{j,\beta_j}+1)c_{j,\beta_j}^{2\gamma_{j,\beta_j}+2}}$$
	 for any $t\ge0$. 
	 \end{Remark}

 Denote that $E_{\beta}:=\{(\alpha_1,\alpha_2,...,\alpha_n):\sum_{1\le j\le n}\frac{\alpha_j+1}{p_{j,\beta_j}}=1\,\&\,\alpha_j\in\mathbb{Z}_{\ge0}\}$.
 Let $f$ be a holomorphic $(n,0)$ form on $\cup_{\beta\in I_1}V_{\beta}$ such that $f=\sum_{\alpha\in E_{\beta}}d_{\beta,\alpha}w_{\beta}^{\alpha}dw_{1,\beta_1}\wedge dw_{2,\beta_2}\wedge...\wedge dw_{n,\beta_n}$ on $V_{\beta}$ for any $\beta\in I_1$.
 
 We recall the following characterization for the holding of the equality in the optimal $L^2$ extension problem.

 \begin{Theorem}[\cite{GY-concavity4}]
 	\label{thm:prod-finite-jet}If $\sum_{\beta\in I_1}\sum_{\alpha\in E_{\beta}}\frac{|d_{\beta,\alpha}|^2(2\pi)^ne^{-\varphi(z_{\beta})}}{\Pi_{1\le j\le n}(\alpha_j+1)c_{j,\beta_j}^{2\alpha_{j}+2}}\in(0,+\infty)$, there exists a holomorphic $(n,0)$ form $F$ on $M$, which satisfies that $(F-f,z_{\beta})\in(\mathcal{O}(K_{M})\otimes\mathcal{I}(\psi))_{z_{\beta}}$ for any $\beta\in I_1$ and
 	\begin{displaymath}
 		\int_{M}|F|^2e^{-\varphi}c(-\psi)\le(\int_{0}^{+\infty}c(s)e^{-s}ds)\sum_{\beta\in I_1}\sum_{\alpha\in E_{\beta}}\frac{|d_{\beta,\alpha}|^2(2\pi)^ne^{-\varphi(z_{\beta})}}{\Pi_{1\le j\le n}(\alpha_j+1)c_{j,\beta_j}^{2\alpha_{j}+2}}.
 	\end{displaymath}
 	
 	Moreover, assume that $f=w_{\beta^*}^{\alpha_{\beta_*}}dw_{1,1}\wedge dw_{2,1}\wedge...\wedge dw_{n,1}$ on $V_{\beta^*}$, where $\beta^*=(1,1,...,1)\in I_1$, then equality $\inf\{\int_{M}|\tilde{F}|^2e^{-\varphi}c(-\psi):\tilde{F}\in H^0(M,\mathcal{O}(K_M))\,\&\,(\tilde{F}-f,z_\beta)\in(\mathcal{O}(K_{M})\otimes\mathcal{I}(\max_{1\le j\le n}\{2\sum_{1\le k\le m_j}p_{j,k}\pi_j^{*}(G_{\Omega_j}(\cdot,z_{j,k}))\}))_{z_\beta}$ for any $\beta\in I_1\}=(\int_{0}^{+\infty}c(s)e^{-s}ds)\sum_{\beta\in I_1}\sum_{\alpha\in E_{\beta}}\frac{|d_{\beta,\alpha}|^2(2\pi)^ne^{-\varphi(z_{\beta})}}{\Pi_{1\le j\le n}(\alpha_j+1)c_{j,\beta_j}^{2\alpha_{j}+2}}$ holds if and only if the following statements hold:

	$(1)$ $\varphi_j=2\log|g_j|+2u_j$ for any $j\in\{1,2,...,n\}$, where $u_j$ is a harmonic function on $\Omega_j$ and $g_j$ is a holomorphic function on $\Omega_j$ satisfying $g_j(z_{j,k})\not=0$ for any $k\in\{1,2,...,m_j\}$;
	
	$(2)$  there exists a nonnegative integer $\gamma_{j,k}$ for any $j\in\{1,2,...,n\}$ and $k\in\{1,2,...,m_j\}$, which satisfies that $\Pi_{1\le k\leq m_j}\chi_{j,z_{j,k}}^{\gamma_{j,k}+1}=\chi_{j,-u_j}$ and $\sum_{1\le j\le n}\frac{\gamma_{j,\beta_j}+1}{p_{j,\beta_j}}=1$ for any $\beta\in I_1$;
	
	$(3)$ $f=(c_{\beta}\Pi_{1\le j\le n}w_{j,\beta_j}^{\gamma_{j,\beta_j}}+g_{\beta})dw_{1,\beta_1}\wedge dw_{2,\beta_2}\wedge...\wedge dw_{n,\beta_n}$ on $V_{\beta}$ for any $\beta\in I_1$, where $c_{\beta}$ is a constant and $g_{\beta}$ is a holomorphic function on $V_{\beta}$ such that $(g_{\beta},z_{\beta})\in\mathcal{I}(\psi)_{z_{\beta}}$;
	
	$(4)$ $\lim_{z\rightarrow z_{\beta}}\frac{c_{\beta}\Pi_{1\le j\le n}w_{j,\beta_j}^{\gamma_{j,\beta_j}}dw_{1,\beta_1}\wedge dw_{2,\beta_2}\wedge...\wedge dw_{n,\beta_n}}{\wedge_{1\le j\le n}\pi_{j}^*(g_j(P_{j})_*(f_{u_j}(\Pi_{1\le k\le m_j}f_{z_{j,k}}^{\gamma_{j,k}+1})(\sum_{1\le k\le m_j}p_{j,k}\frac{df_{z_{j,k}}}{f_{z_{j,k}}})))}=c_0$ for any $\beta\in I_1$, where $c_0\in\mathbb{C}\backslash\{0\}$ is a constant independent of $\beta$, $f_{u_j}$ is a holomorphic function $\Delta$ such that $|f_{u_j}|=P_j^*(e^{u_j})$ and $f_{z_{j,k}}$ is a holomorphic function on $\Delta$ such that $|f_{z_{j,k}}|=P_j^*(e^{G_{\Omega_j}(\cdot,z_{j,k})})$ for any $j\in\{1,2,...,n\}$ and $k\in\{1,2,...,m_j\}$.
 	 \end{Theorem}

\subsection{Some other required results}
\label{sec:other}

In this section, we recall and give some lemmas, which will be used in the proofs of the main theorems.

Let $U\subset \mathbb{C}^n$  be an open set. Let us recall the definition of admissible weight given in \cite{Winiarski} and \cite{Winiarski2}.
\begin{Definition}[see \cite{Winiarski,Winiarski2}]
	A nonnegative measurable function $\rho$ on $U$ is called an admissible weight, if for any $z_0\in U$ the following condition is satisfied: there exists a neighborhood $V_{z_0}$ in $U$ and a constant $C_{z_0}>0$ such that
	$$|f(z)|^2\leq C_{z_0}\int_{U}|f|^2\rho$$ 
	holds for any $z\in V_{z_0}$ and any holomorphic function $f$ on  $U$. 
\end{Definition}

Let $\rho$  be an  admissible weight on $U$. The weighted Bergman space $A^2(U,\rho)$ is defined as follows:
$$A^2(U,\rho):=\left\{f\in\mathcal{O}(U):\int_U|f|^2\rho<+\infty\right\}.$$
Denote that $$\ll f,g\gg_{U,\rho}:=\int_Uf\overline g\rho$$
and
 $||f||_{U,\rho}:=(\int_U|f|^2\rho)^{\frac{1}{2}}$ for any $f,g\in A^2(U,\rho)$.

\begin{Lemma}[see \cite{Winiarski,Winiarski2}]
$A^2(U,\rho)$ is a separable Hilbert space equipped with the inner product $\ll\cdot,\cdot\gg_{U,\rho}$.
\end{Lemma}

We recall a sufficient condition for a weight to be an admissible weight.

\begin{Lemma}[see \cite{GY-saitohprodct}]\label{l:admissible}
	 Let $\rho$ be a nonnegative Lebesgue measurable function on $U$, and let $S$ be an analytic subset of $U$. Assume  that for any $K\Subset U\backslash S$, there is $a>0$ such that $\int_{K}\rho^{-a}<+\infty$. Then $\rho$ is an admissible weight on $U$.
\end{Lemma}

Let $U\subset \mathbb{C}^n$ and $W\subset \mathbb{C}^m$ be two open sets.
 Let $\rho_1$ and  $\rho_2$ be two nonnegative  Lebesgue measurable functions on $U$ and $W$ respectively.  Assume that for any relatively compact set $U_1\Subset U$ ($W_2\Subset W$), there exists a real number $a_1>0$ ($a_2>0$) such that $\rho_1^{-a_1}$ ($\rho_2^{-a_2}$) is integrable on $U_1$ ($W_2$).
Let $M:=U\times W$ and $\rho=\rho_1\times\rho_2$. By Lemma \ref{l:admissible}, we know that $\rho_1$, $\rho_2$ and $\rho$ are admissible weights on $U$,  on $W$ and on $M$ respectively.

The following lemma gives a product property of Bergman spaces.

\begin{Lemma}
[see \cite{GMY-boundary4}]
\label{basis of product}Let $\{f_i(z)\}_{i\in \mathbb{Z}_{\ge 0}}$ and $\{g_j(w)\}_{j\in \mathbb{Z}_{\ge 0}}$ be the complete orthonormal basis of $A^2(U,\rho_1)$ and $A^2(W,\rho_2)$ respectively. Then $\{f_i(z) g_j(w)\}_{i,j\in\mathbb{Z}_{\ge 0}}$ is a complete orthonormal basis of $A^2(M,\rho)$.
\end{Lemma}

Let $D_j$, $M$, $M_j$, $Z_j$, $Z_0$, $I_1$ be as in Section \ref{sec:hardy}. Denote that 
$$S:=\prod_{1\le j\le n}\partial D_j.$$ Let us recall the definition and some properties of the Hardy space over $S$.

Let $\lambda$ be a Lebesgue measurable function on $S$ such that $\inf_{S}\lambda>0$.	Let $f\in L^2(S,\lambda d\sigma)$, where $d\sigma:=\frac{1}{(2\pi)^n}|dw_1|\ldots|dw_n|$. We call $f\in H^2_{\lambda}(M,S)$ if there exists $\{f_m\}_{m\in\mathbb{Z}_{\ge0}}\subset\mathcal{O}(M)\cap C(\overline M)\cap L^2(S,\lambda d\sigma)$ such that $\lim_{m\rightarrow+\infty}\|f_m-f\|_{S,\lambda}^2=0$, where $\|g\|_{S,\lambda}:=\left(\int_{S}|g|^2\lambda d\sigma\right)^{\frac{1}{2}}$ for any $g\in L^2(S,\lambda d\sigma)$.

 Denote that 
$$\ll f,g\gg_{S,\lambda}=\frac{1}{(2\pi)^n}\int_S f\overline g\lambda |dw_1|\ldots|dw_n|$$
for any $f,g\in L^2(S,\lambda d\sigma),$
then $H_{\lambda}^2(M,S)$ is a Hilbert space equipped with the inner product $\ll \cdot,\cdot\gg_{S,\lambda}$ (see \cite{GY-saitohprodct}). There exists a  linear injective map $P_S:H^2_{\lambda}(M,S)\rightarrow\mathcal{O}(M)$ satisfying that $P_S(f)=f$ for any $f\in\mathcal{O}(M)\cap C(\overline M)\cap L^2(S,\lambda d\sigma)$ (see \cite{GY-saitohprodct}). For simplicity, denote $P_S(f)$ by $f^*$.

We recall three lemmas about $H^2_{\lambda}(M,S)$, which will be used in the proof of Lemma \ref{l:pro1}.
\begin{Lemma}[\cite{GY-saitohprodct}]
	\label{l:a3}For any compact subset $K$ of $M$, there exists a positive constant $C_K$ such that 
	$$|f^*(z)|\le C_K\|f\|_{S,\lambda}$$
	holds for any $z\in K$ and $f\in H_{\lambda}^2(M,S)$.
\end{Lemma}

\begin{Lemma}[\cite{GY-saitohprodct}]
	\label{l:a K<+infty}Assume that $M_{S}(Z_0,J,\lambda)<+\infty$. Then  there is a unique holomorphic function $f\in H^2_{\lambda}(M,S)$ such that $f^*(z_\beta)=h_0(z_\beta)$ for any $\beta\in I_1$, and $M_{S}(Z_0,J,\lambda)=\|f\|_{S,\lambda}^2$.
	\end{Lemma}

	Let $M_a=\prod_{1\le j\le n_a}D_j$ be a bounded domain in $\mathbb{C}^{n_a}$, where $D_j$ is planar regular region with finite boundary components which are analytic Jordan curves for any $1\le j\le n_a$. Denote that $S_a:=\prod_{1\le j\le n_a}\partial D_j$. Let $M_b=\prod_{1\le j\le n_b}\tilde D_j$ be a bounded domain in $\mathbb{C}^{n_b}$, where $\tilde D_j$ is planar regular region with finite boundary components which are analytic Jordan curves for any $1\le j\le n_b$. Denote that $S_b:=\prod_{1\le j\le n_b}\partial \tilde D_j$. Denote that $M:=M_a\times M_b\subset \mathbb{C}^{n_a+n_b}$ ($n=n_a+n_b$) and $S:=S_a\times S_b$.

\begin{Lemma}[\cite{GY-saitohprodct}]
	\label{p:7}Let $\lambda_a$ be a Lebesgue measurable function on $S_a$ such that $\inf_{S_a}\lambda_a>0$, and let $\lambda_a$ be a Lebesgue measurable function on $S_b$ such that $\inf_{S_b}\lambda_b>0$. Denote that $\lambda:=\lambda_a\lambda_b$ on $S$. Assume that $H^2_{\lambda}(M,S)\not=\{0\}$. Then we have $H^2_{\lambda_a}(M_a,S_a)\not=\{0\}$ and $H^2_{\lambda_b}(M_b,S_b)\not=\{0\}$. Furthermore,
	$\{e_m(z)\tilde e_l(w)\}_{m,l\in\mathbb{Z}_{>0}}$ is a complete orthonormal basis for $H^2_{\lambda}(M,S)$, where $\{e_m\}_{m\in\mathbb{Z}_{>0}}$ is a complete orthonormal basis for $H^2_{\lambda_a}(M_a,S_a)$, and $\{\tilde e_m\}_{m\in\mathbb{Z}_{>0}}$ is a complete orthonormal basis for $H^2_{\lambda_b}(M_b,S_b)$.
\end{Lemma}

In the following, we give three product properties, which will be used in the proof of Theorem \ref{thm:3.1}. 

Let $Z_j=\{z_{j,1},z_{j,2},...,z_{j,m_j}\}\subset D_j$ for any  $j\in\{1,2,...,n\}$, where $m_j$ is a positive integer. Denote that 
$$Z_0:=\prod_{1\le j\le n}Z_j\subset M.$$
 Denote that $I_1:=\{(\beta_1,\beta_2,...,\beta_n):1\le \beta_j\le m_j$ for any $j\in\{1,2,...,n\}\}$, $V_{\beta}:=\prod_{1\le j\le n}V_{z_{j,\beta_j}}$ and $z_{\beta}:=(z_{1,\beta_1},z_{2,\beta_2},\ldots,z_{n,\beta_n})\in M$ for any $\beta=(\beta_1,\beta_2,...,\beta_n)\in I_1$ .
Let $h_j$ be a holomorphic function on a neighborhood of $Z_j$ for any $1\le j\le n$ satisfying that there exists $k\in\{1,\ldots,m_j\}$ such that $h_j(z_{j,k})\not=0$. Denote that $h_0=\prod_{1\le j\le n}h_j$. 

Let $\rho_1$ and $\rho_2$ be two Lebesgue measurable functions on $\partial D_1$ and $S_1:=\prod_{2\le j\le n}\partial D_j$ respectively, which satisfy that $\inf_{\partial D_1}\rho_1>0$ and $\inf_{\ S_1}\rho_2>0$. Let $\lambda_1$ and $\lambda_2$ be two nonnegative Lebesgue measurable functions on $D_1$ and $M_1$ respectively, which satisfy that for any relatively compact subset $R_1\Subset D_1$ ($R_2\Subset M_1$), there is $a>0$ such that $\lambda_1^{-a}$ ($\lambda_2^{-a}$) is integrable on $R_1$ ($R_2$). By Lemma \ref{l:admissible}, we know that $\lambda_1$  and $\lambda_2$ are admissible weights on $D_1$ and on $M_1$ respectively. 
 
 Let us consider the following  minimal integrals. Let $J_{\beta}$ be the maximal ideal of $\mathcal{O}_{z_\beta}$ for any $\beta\in I_1$.
Denote that 
\begin{displaymath}
	\begin{split}
		M_{H,1}(Z_0,J,\rho_1 \lambda_2):=\inf\bigg\{&\|f\|^2_{\partial D_1\times M_1,\rho_1 \lambda_2}:f\in H_{\rho}^2(M,\partial D_1\times M_1) 
		\\ &\text{ s.t. } f^*(z_{\beta})=h_0(z_{\beta})\text{ for any $\beta\in I_1$}\bigg\},\end{split}
\end{displaymath}
\begin{displaymath}
	\begin{split}
		M_{\partial D_1}:=\inf\Bigg\{\frac{1}{2\pi}\int_{\partial D_1}&|f|^2\rho_1|dz_1|: f\in H^2(D_1)		\\ & \text{ s.t. } f(z_{j,k})=h_1(z_{j,k})\text{ for any $1\le k\le m_1$}\Bigg\}\end{split}
\end{displaymath}
and 
\begin{displaymath}
	\begin{split}
		M_{M_1}:=\inf\bigg\{&\int_{M_1}|f|^2\lambda_2 :f\in \mathcal{O}(M_1) 
		\\ &\text{ s.t. } f(z_{\gamma})=\prod_{2\le l\le n}h_l(z_{l,\gamma_l}) \text{ for any $\gamma\in I_{1,1}$}\bigg\},\end{split}
\end{displaymath}
where $I_{1,1}:=\{\gamma=(\gamma_{2},\ldots,\gamma_n)\in\mathbb{Z}^{n-1} :1\le\gamma_l\le m_l$ for any $2\le l\le n\}$ and  $z_{\gamma}:=(z_{2,\gamma_{2}},\ldots,z_{n,\gamma_{n}})\in M_j$ for any $\gamma\in I_{1,1}$.

\begin{Lemma}
	\label{l:pro2} $M_{H,1}(Z_0,J,\rho_1 \lambda_2)=M_{\partial D_1}\times M_{M_1}$.
\end{Lemma}
\begin{proof}
	By definitions of $M_{H,1}(Z_0,J,\rho_1 \lambda_2)$, $M_{\partial D_1}$ and $M_{M_1}$, we have 
	$$M_{H,1}(Z_0,J,\rho_1 \lambda_2)\le M_{\partial D_1}\times M_{M_1}.$$ Thus, it suffices to prove $M_{H,1}(Z_0,J,\rho_1 \lambda_2)\ge M_{\partial D_1}\times M_{M_1}$.
	
	Without loss of generality, assume that $M_{H,1}(Z_0,J,\rho_1 \lambda_2)<+\infty$. There exists $f_0\in H^2_{\rho_1 \lambda_2}(M,\partial D_1,\times M_1)$ satisfying $f_0^*(z_{\beta})=h_0(z_{\beta})$ for any $\beta\in I_1$ and
	\begin{equation}
		\label{eq:0320a}
		M_{H,1}(Z_0,J,\rho_1 \lambda_2)=\|f_0\|^2_{\partial D_1\times M_1,\rho_1 \lambda_2}.
	\end{equation}
	 As $H^2_{\rho_1 \lambda_2}(M,\partial D_1\times M_1)\not=\emptyset$, by Lemma \ref{l:prod-d1xm1}, $H^2_{\rho_1}(D_1,\partial D_1)\not=\{0\}$ and $A^2(M_1,\lambda_2)\not=\{0\}$. 

Let  $\{e_l\}_{l\in\mathbb{Z}_{>0}}$ be a complete orthonormal basis for $H^2_{\rho_1}(D_1,\partial D_1)$, which satisfies that $e_l(z_{1,l})\not=0$ for $1\le l\le m_1$ and $e_l(z_{1,k})=0$ for $0\le k< l$. Denote that 
$$K_1:=\{e_l\}_{l>m_1}.$$
We call $\gamma<\tilde\gamma$ for $\gamma,\tilde\gamma\in I_{1,1}$ if there exists $s\in\{2,\ldots,n\}$ such that $\gamma_l=\tilde\gamma_l$ when $l<s$ and $\gamma_s>\tilde\gamma_s$.
 Let $\{\tilde e_m\}_{m\in\mathbb{Z}_{>0}}$ be a complete orthonormal basis for $A^2(M_1,\lambda_2)$, which satisfies that there exists $N_1\in\mathbb{Z}_{>0}$ such that $\sum_{\gamma\in I_{1,1}}|\tilde e_m(z_{\gamma})|=0$  when $m>N_1$ and $\sum_{\gamma\in I_{1,1}}|\tilde e_m(z_{\gamma})|\not=0$  when $m\le N_1$, and $s_m$ is strictly increasing with respect to $m$ when $m\le N_1$, where $s_m:=\inf\{\gamma\in I_{1,1}:\tilde e_m(z_{\gamma})\not=0\}$. Denote that 
 $$K_2:=\{\tilde e_m\}_{m>N_1}.$$
	Lemma \ref{l:prod-d1xm1} shows that $\{e_m(z)\tilde e_l(w)\}_{m,l\in\mathbb{Z}_{>0}}$ is a complete orthonormal basis for $H^2_{\rho_1 \lambda_2}(M,\partial D_1\times M_1)$.
	Then we have 
	$$f_0=\sum_{l,m\in\mathbb{Z}_{>0}}a_{l,m}e_l\tilde e_m.$$
By Lemma \ref{l:b1-p}, we know that 
$$f_0^*=\sum_{l,m\in\mathbb{Z}_{>0}}a_{l,m}e^*_l\tilde e_m\,\,\,\,\text{(compactly uniform convergence)}.$$
Since there exists $k\in\{1,\ldots,m_1\}$ such that $h_1(z_{1,k})\not=0$, without loss of generality, assume that $h_z(z_{1,1})\not=0.$ As $f_0^*(z_{\beta})=h_0(z_\beta)=\prod_{1\le j\le n}h_j(z_{j,\beta_j})$ for any $\beta\in I_1$, we obtain that 
$$\frac{\sum_{l,m\in\mathbb{Z}_{>0}}a_{l,m}e^*_l(z_{1,1})\tilde e_m(z_{\gamma})}{h_1(z_{1,1})}=\prod_{2\le j\le n}h_j(z_{j,\gamma_j})$$
for any $\gamma\in I_{1,1}$. Note that $\sum_{l,m\in\mathbb{Z}_{>0}}a_{l,m}e^*_l(z_{1,1})\tilde e_m\in\mathcal{O}(M_1)$ and
\begin{displaymath}
	\begin{split}
	&\int_{M_1}|\sum_{l,m\in\mathbb{Z}_{>0}}a_{l,m}e^*_l(z_{1,1})\tilde e_m|^2\lambda_2\\
	=& \sum_{l,m\in\mathbb{Z}_{>0}}|a_{l,m}e^*_l(z_{1,1})|^2\\
	=&|e^*(z_{1,1})|^2\|f_0\|^2_{\partial D_1\times M_1,\rho_1 \lambda_2}\\
	<&+\infty.
		\end{split}
\end{displaymath}
Thus, we have $M_{M_1}<+\infty$. As $H^2_{\rho_1}(D_1,\partial D_1)\not=\emptyset$ and $D_1$ is a planar regular region bounded by finite analytic Jordan curves, we know that $M_{\partial D_1}<+\infty$.

 Let $f_1\in H^2_{\rho_1}(D_1,\partial D_1)$ satisfy $f_1^*(z_{1,k})=h_1(z_{1,k})$ for any $1\le k\le m_1$ and 
$$M_{\partial D_1}=\frac{1}{2\pi}\int_{\partial D_1}|f_1|^2\rho_1|dz_1|.$$
Let $f_2\in\mathcal{O}(M_1)$ satisfy that $f_2(z_{\gamma})=\prod_{2\le j\le n}h_j(z_{j,\gamma_j})$ for any $\gamma\in I_{1,1}$ and 
$$M_{M_1}=\int_{M_1}|f_2|^2\lambda_2.$$
Then we know that 
\begin{equation}
	\label{eq:0320b}\int_{\partial D_1}f_1\overline{f}\rho_1|dz_1|=0
\end{equation}
for any $f\in K_1$, and 
\begin{equation}
	\label{eq:0320c}
	\int_{M_1}f_2\overline{g}\lambda_2=0
\end{equation}
for any $g\in K_2$.
Denote that
\begin{equation}
	\label{eq:0320d}F_0:=f_0-f_1f_2,
\end{equation}
then we have $F_0\in H^2_{\rho_1 \lambda_2}(M,\partial D_1\times M_1)$ and $F^*_0(z_{\beta})=0$ for any $\beta\in I_1$. As $\{e_m(z)\tilde e_l(w)\}_{m,l\in\mathbb{Z}_{>0}}$ is a complete orthonormal basis for $H^2_{\rho_1 \lambda_2}(M,\partial D_1\times M_1)$, there exists $\{b_{l,m}\}_{l,m\in\mathbb{Z}_{>0}}\subset\mathbb{C}$ such that 
$$F_0=\sum_{l,m\in\mathbb{Z}_{>0}}b_{l,m}e_l\tilde e_m=F_1+F_2,$$
 where $F_1:=\sum_{1\le l\le m_1}\sum_{1\le m\le N_1}b_{l,m}e_l\tilde e_m$ and $F_2:=\sum_{e_l\in K_1\text{ or }\tilde e_m\in K_2}b_{l,m}e_l\tilde e_m$.
 Note that $F^*_2(z_{\beta})=0$ for any $\beta\in I_1$, then $F^*_1(z_{\beta})=0$ for any $\beta\in I_1$. By the construction   of $\{e_l\}_{1\le l\le m_1}$ and $\{\tilde e_m\}_{1\le m\le N_1}$, we know $b_{l,m}=0$ for $1\le l\le m_1$ and $1\le m\le N_1$, i.e., 
 $$F_1\equiv0.$$
Note that $F_0=F_2=\sum_{e_l\in K_1\text{ or }\tilde e_m\in K_2}b_{l,m}e_l\tilde e_m,$ then it follows from equality \eqref{eq:0320a}, \eqref{eq:0320b}, \eqref{eq:0320c} and \eqref{eq:0320d} that 
 \begin{equation}\nonumber
 \begin{split}
 	 	M_{H,1}(Z_0,J,\rho_1 \lambda_2)&=\|f_0\|^2_{\partial D_1\times M_1,\rho_1 \lambda_2}\\
 	 	&=\|F_2\|^2_{\partial D_1\times M_1,\rho_1 \lambda_2}+\|f_1f_2\|^2_{\partial D_1\times M_1,\rho_1 \lambda_2}\\
 	 	&\ge M_{\partial D_1}\times M_{M_1}.
 \end{split}
 \end{equation}

Thus, Lemma \ref{l:pro2} holds.
\end{proof}

Let us consider the following  minimal integrals. Denote that 
\begin{displaymath}
	\begin{split}
		M_{S}(Z_0,J,\rho_1 \rho_2):=\inf\bigg\{&\|f\|^2_{S,\rho_1 \rho_2}:f\in H_{\rho_1 \rho_2}^2(M,S) 
		\\ &\text{ s.t. } f^*(z_{\beta})=h_0(z_{\beta})\text{ for any $\beta\in I_1$}\bigg\},\end{split}
\end{displaymath}
and 
\begin{displaymath}
	\begin{split}
		M_{S_1}:=\inf\bigg\{&\|f\|^2_{S_1,\rho_2}:f\in H_{\rho_2}^2(M_1,S_1) 
		\\ &\text{ s.t. } f^*(z_{\gamma})=\prod_{2\le j\le n}h_j(z_{j,\gamma_j})\text{ for any $\gamma\in I_{1,1}$}\bigg\}.\end{split}
\end{displaymath}

We give a product property as follows:

\begin{Lemma}
	\label{l:pro1}$M_{S}(Z_0,J,\rho_1 \rho_2)=M_{\partial D_1}\times M_{S_1}$.
\end{Lemma}

\begin{proof} The proof is similar to the proof of Lemma \ref{l:pro2}. 	By definitions of $M_{S}(Z_0,J,\rho_1 \rho_2)$, $M_{\partial D_1}$ and $M_{S_1}$, we have 
$$M_{S}(Z_0,J,\rho_1 \rho_2)\le M_{\partial D_1}\times M_{S_1}.$$ Thus, it suffices to prove $M_{S}(Z_0,J,\rho_1 \rho_2)\ge M_{\partial D_1}\times M_{S_1}$.
	
	Without loss of generality, assume that $M_{S}(Z_0,J,\rho_1 \rho_2)<+\infty$. By Lemma \ref{l:a K<+infty}, there exists $f_0\in H^2_{\rho_1 \rho_2}(M,S)$ satisfying $f_0^*(z_{\beta})=h_0(z_{\beta})$ for any $\beta\in I_1$ and
	\begin{equation}
		\label{eq:0321a}
		M_{S}(Z_0,J,\rho_1 \rho_2)=\|f_0\|^2_{S,\rho_1 \rho_2}.
	\end{equation}
	 As $H^2_{\rho_1 \rho_2}(M,S)\not=\emptyset$, by Lemma \ref{p:7}, $H^2_{\rho_1}(D_1,\partial D_1)\not=\{0\}$ and $H^2_{\rho_2}(M_1,S_1)\not=\{0\}$. 

Let  $\{e_l\}_{l\in\mathbb{Z}_{>0}}$ be a complete orthonormal basis for $H^2_{\rho_1}(D_1,\partial D_1)$, which satisfies that $e_l(z_{1,l})\not=0$ for $1\le l\le m_1$ and $e_l(z_{1,k})=0$ for $0\le k< l$. Denote that 
$$K_1:=\{e_l\}_{l>m_1}.$$
We call $\gamma<\tilde\gamma$ for $\gamma,\tilde\gamma\in I_{1,1}$ if there exists $s\in\{2,\ldots,n\}$ such that $\gamma_l=\tilde\gamma_l$ when $l<s$ and $\gamma_s>\tilde\gamma_s$.
 Let $\{\tilde e_m\}_{m\in\mathbb{Z}_{>0}}$ be a complete orthonormal basis for $H^2_{\rho_2}(M_1,S_1)$, which satisfies that there exists $N_1\in\mathbb{Z}_{>0}$ such that $\sum_{\gamma\in I_{1,1}}|\tilde e_m(z_{\gamma})|=0$  when $m>N_1$ and $\sum_{\gamma\in I_{1,1}}|\tilde e_m(z_{\gamma})|\not=0$  when $m\le N_1$, and $s_m$ is strictly increasing with respect $m$ when $m\le N_1$, where $s_m:=\inf\{\gamma\in I_{1,1}:\tilde e_m(z_{\gamma})\not=0\}$. Denote that 
 $$K_2:=\{\tilde e_m\}_{m>N_1}.$$
	Lemma \ref{p:7} shows that $\{e_m(z)\tilde e_l(w)\}_{m,l\in\mathbb{Z}_{>0}}$ is a complete orthonormal basis for $H^2_{\rho_1 \rho_2}(M,S)$.
	Then we have 
	$$f_0=\sum_{l,m\in\mathbb{Z}_{>0}}a_{l,m}e_l\tilde e_m.$$
By Lemma \ref{l:a3}, we know that 
$$f_0^*=\sum_{l,m\in\mathbb{Z}_{>0}}a_{l,m}e^*_l\tilde e^*_m\,\,\,\,\text{(compactly uniform convergence)}.$$
Since there exists $k\in\{1,\ldots,m_1\}$ such that $h_1(z_{1,k})\not=0$, without loss of generality, assume that $h_z(z_{1,1})\not=0.$ Note that $\sum_{l,m\in\mathbb{Z}_{>0}}|a_{l,m}e^*_l(z_{1,1})|^2<+\infty$, then we have 
$$\sum_{l,m\in\mathbb{Z}_{>0}}a_{l,m}e^*_l(z_{1,1})\tilde e_m\in H^2_{\rho_2}(M_1,S_1).$$ As $f^*(z_{\beta})=h_0(z_\beta)=\prod_{1\le j\le n}h_j(z_{j,\beta_j})$ for any $\beta\in I_1$, we obtain that 
$$\frac{\sum_{l,m\in\mathbb{Z}_{>0}}a_{l,m}e^*_l(z_{1,1})\tilde e^*_m(z_{\gamma})}{h_1(z_{1,1})}=\prod_{2\le j\le n}h_j(z_{j,\gamma_j})$$
for any $\gamma\in I_{1,1}$. Thus, we have $M_{S_1}<+\infty$. Similarly, we have $M_{\partial D_1}<+\infty$. 

 Let $f_1\in H^2_{\rho_1}(D_1,\partial D_1)$ satisfy $f_1^*(z_{1,k})=h_1(z_{1,k})$ for any $1\le k\le m_1$ and 
$$M_{\partial D_1}=\frac{1}{2\pi}\int_{\partial D_1}|f_1|^2\rho_1|dz_1|.$$
Let $f_2\in H^2_{\rho_2}(M_1,S_1)$ satisfy that $f^*_2(z_{\gamma})=\prod_{2\le j\le n}h_j(z_{j,\gamma_j})$ for any $\gamma\in I_{1,1}$ and 
$$M_{M_1}=\|f_2\|^2_{S_1,\rho_2}.$$
Then we know that 
\begin{equation}
	\label{eq:0321b}\int_{\partial D_1}f_1\overline{f}\rho_1|dz_1|=0
\end{equation}
for any $f\in K_1$, and 
\begin{equation}
	\label{eq:0321c}
	\ll f_2,g\gg_{S_1,\rho_2}=0
\end{equation}
for any $g\in K_2$.
Denote that
\begin{equation}
	\label{eq:0321d}F_0:=f_0-f_1f_2,
\end{equation}
then we have $F_0\in H^2_{\rho_1 \rho_2}(M,S)$ and $F^*_0(z_{\beta})=0$ for any $\beta\in I_1$. As $\{e_m(z)\tilde e_l(w)\}_{m,l\in\mathbb{Z}_{>0}}$ is a complete orthonormal basis for $H^2_{\rho_1 \rho_2}(M,S)$, there exists $\{b_{l,m}\}_{l,m\in\mathbb{Z}_{>0}}\subset\mathbb{C}$ such that 
$$F_0=\sum_{l,m\in\mathbb{Z}_{>0}}b_{l,m}e_l\tilde e_m=F_1+F_2,$$
 where $F_1:=\sum_{1\le l\le m_1}\sum_{1\le m\le N_1}b_{l,m}e_l\tilde e_m$ and $F_2:=\sum_{e_l\in K_1\text{ or }\tilde e_m\in K_2}b_{l,m}e_l\tilde e_m$.
 Note that $F^*_2(z_{\beta})=0$ for any $\beta\in I_1$, then $F^*_1(z_{\beta})=0$ for any $\beta\in I_1$. By the construction   of $\{e_l\}_{1\le l\le m_1}$ and $\{\tilde e_m\}_{1\le m\le N_1}$, we know $b_{l,m}=0$ for $1\le l\le m_1$ and $1\le m\le N_1$, i.e., 
 $$F_1\equiv0.$$
Note that $F_0=F_2=\sum_{e_l\in K_1\text{ or }\tilde e_m\in K_2}b_{l,m}e_l\tilde e_m,$ then it follows from equality \eqref{eq:0321a}, \eqref{eq:0321b}, \eqref{eq:0321c} and \eqref{eq:0321d} that 
 \begin{equation}\nonumber
 	M_{S}(Z_0,J,\rho_1 \rho_2)=\|f_0\|^2_{S,\rho_1 \rho_2}=\|F_2\|^2_{S,\rho_1 \rho_2}+\|f_1f_2\|^2_{S,\rho_1 \rho_2}\ge M_{\partial D_1}\times M_{S_1}.
 \end{equation}

Thus, Lemma \ref{l:pro1} holds.
\end{proof}

Let us consider the following  minimal integrals. Denote that 
\begin{displaymath}
	\begin{split}
		M_{D_1}:=\inf\bigg\{&\int_{D_1}|f|^2\lambda_1:f\in \mathcal{O}(D_1)
		\\ &\text{ s.t. } f(z_{1,k})=h_1(z_{1,k})\text{ for any $1\le k\le m_1$}\bigg\},\end{split}
\end{displaymath}
and 
\begin{displaymath}
	\begin{split}
		M_{M}:=\inf\bigg\{&\int_M|f|^2\lambda_1\lambda_2 :f\in \mathcal{O}(M)
		\\ &\text{ s.t. } f(z_{\beta})=h_0(z_\beta)\text{ for any $\beta\in I_1$}\bigg\}.\end{split}
\end{displaymath}

\begin{Lemma}
	\label{l:pro3}$M_M=M_{D_1}\times M_{M_1}.$
\end{Lemma}
\begin{proof}The proof is similar to the proof of Lemma \ref{l:pro2}. 
	By definitions of $M_M$, $M_{ D_1}$ and $M_{M_1}$, we have $M_M\le M_{D_1}\times M_{M_1}$. Thus, it suffices to prove $M_M\ge M_{D_1}\times M_{M_1}$.
	
	Without loss of generality, assume that $M_M<+\infty$. There exists $f_0\in \mathcal{O}(M)$ satisfying $f_0(z_{\beta})=h_0(z_{\beta})$ for any $\beta\in I_1$ and
	\begin{equation}
		\label{eq:0310a}
		M_M=\int_M|f_0|^2\lambda_1\lambda_2.
	\end{equation}
	Let  $\{e_l\}_{l\in\mathbb{Z}_{>0}}$ be a complete orthonormal basis for $A^2(D_1,\lambda_1)$, which there exists $N_1\in\mathbb{Z}_{>0}$ such that $\sum_{1\le k\le m_1}| e_l(z_{1,k})|=0$  when $l>N_1$ and $\sum_{1\le k\le m_1}| e_l(z_{1,k})|\not=0$  when $l\le N_1$, and $s_l$ is strictly increasing with respect $l$ when $l\le N_1$, where $s_l:=\inf\{k\in \{1,\ldots,m_1\}:e_l(z_{1,k})\not=0\}$. Denote that 
$$K_1:=\{e_l\}_{l>N_1}.$$
 Let $\{\tilde e_m\}_{m\in\mathbb{Z}_{>0}}$ be a complete orthonormal basis for $A^2(M_1,\lambda_2)$, which satisfies that there exists $N_2\in\mathbb{Z}_{>0}$ such that $\sum_{\gamma\in I_{1,1}}|\tilde e_m(z_{\gamma})|=0$  when $m>N_2$ and $\sum_{\gamma\in I_{1,1}}|\tilde e_m(z_{\gamma})|\not=0$  when $m\le N_2$, and $\tilde s_m$ is strictly increasing with respect $m$ when $m\le N_2$, where $\tilde s_m:=\inf\{\gamma\in I_{1,1}:\tilde e_m(z_{\gamma})\not=0\}$. Denote that 
 $$K_2:=\{\tilde e_m\}_{m>N_2}.$$
	Lemma \ref{basis of product} shows that $\{e_m(z)\tilde e_l(w)\}_{m,l\in\mathbb{Z}_{>0}}$ is a complete orthonormal basis for $A^2(M,\lambda_1\lambda_2)$.
	Then we have 
	$$f_0=\sum_{l,m\in\mathbb{Z}_{>0}}a_{l,m}e_l\tilde e_m.$$
Since there exists $k\in\{1,\ldots,m_1\}$ such that $h_1(z_{1,k})\not=0$, without loss of generality, assume that $h_z(z_{1,1})\not=0.$ As $f_0(z_{\beta})=h_0(z_\beta)=\prod_{1\le j\le n}h_j(z_{j,\beta_j})$ for any $\beta\in I_1$, we obtain that 
$$\frac{\sum_{l,m\in\mathbb{Z}_{>0}}a_{l,m}e_l(z_{1,1})\tilde e_m(z_{\gamma})}{h_1(z_{1,1})}=\prod_{2\le j\le n}h_j(z_{j,\gamma_j})$$
for any $\gamma\in I_{1,1}$. Note that $\sum_{l,m\in\mathbb{Z}_{>0}}a_{l,m}e_l(z_{1,1})\tilde e_m\in\mathcal{O}(M_1)$ and
\begin{displaymath}
	\begin{split}
	&\int_{M_1}|\sum_{l,m\in\mathbb{Z}_{>0}}a_{l,m}e_l(z_{1,1})\tilde e_m|^2\lambda_2\\
	=& \sum_{l,m\in\mathbb{Z}_{>0}}|a_{l,m}e_l(z_{1,1})|^2\\
	=&|e(z_{1,1})|^2\int_{M}|f_0|^2\lambda_1\lambda_2\\
	<&+\infty.
		\end{split}
\end{displaymath}
Thus, we have $M_{M_1}<+\infty$. Similarly, we have $M_{D_1}<+\infty$.

 Let $f_1\in \mathcal{O}(D_1)$ satisfy $f_1(z_{1,k})=h_1(z_{1,k})$ for any $1\le k\le m_1$ and 
$$M_{ D_1}=\int_{ D_1}|f_1|^2\lambda_1.$$
Let $f_2\in\mathcal{O}(M_1)$ satisfy that $f_2(z_{\gamma})=\prod_{2\le j\le n}h_j(z_{j,\gamma_j})$ for any $\gamma\in I_{1,1}$ and 
$$M_{M_1}=\int_{M_1}|f_2|^2\lambda_2.$$
Then we know that 
\begin{equation}
	\label{eq:0310b}\int_{D_1}f_1\overline{f}\lambda_1=0
\end{equation}
for any $f\in K_1$, and 
\begin{equation}
	\label{eq:0310c}
	\int_{M_1}f_2\overline{g}\lambda_2=0
\end{equation}
for any $g\in K_2$.
Denote that
\begin{equation}
	\label{eq:0310d}F_0:=f_0-f_1f_2,
\end{equation}
then we have $F_0\in A^2(M,\lambda_1\lambda_2)$ and $F_0(z_{\beta})=0$ for any $\beta\in I_1$. As $\{e_m(z)\tilde e_l(w)\}_{m,l\in\mathbb{Z}_{>0}}$ is a complete orthonormal basis for $A^2(M,\lambda_1\lambda_2)$, there exists $\{b_{l,m}\}_{l,m\in\mathbb{Z}_{>0}}\subset\mathbb{C}$ such that 
$$F_0=\sum_{l,m\in\mathbb{Z}_{>0}}b_{l,m}e_l\tilde e_m=F_1+F_2,$$
 where $F_1:=\sum_{1\le l\le N_1}\sum_{1\le m\le N_2}b_{l,m}e_l\tilde e_m$ and $F_2:=\sum_{e_l\in K_1\text{ or }\tilde e_m\in K_2}b_{l,m}e_l\tilde e_m$.
 Note that $F_2(z_{\beta})=0$ for any $\beta\in I_1$, then $F_1(z_{\beta})=0$ for any $\beta\in I_1$. By the construction   of $\{e_l\}_{1\le l\le N_1}$ and $\{\tilde e_m\}_{1\le m\le N_2}$, we know $b_{l,m}=0$ for $1\le l\le N_1$ and $1\le m\le N_2$, i.e., 
 $$F_1\equiv0.$$
Note that $F_0=F_2=\sum_{e_l\in K_1\text{ or }\tilde e_m\in K_2}b_{l,m}e_l\tilde e_m,$ then it follows from equality \eqref{eq:0310a}, \eqref{eq:0310b}, \eqref{eq:0310c} and \eqref{eq:0310d} that 
 \begin{equation}\nonumber
 	M_M=\int_M|f_0|^2\lambda_1\lambda_2=\int_M|F_2|^2\lambda_1\lambda_2+\int_M|f_1f_2|^2\lambda_1\lambda_2\ge M_{ D_1}\times M_{M_1}.
 	\end{equation}
 	
 	Thus, Lemma \ref{l:pro3} holds.
\end{proof}

\section{Proofs of Theorem \ref{main theorem1}, Remark \ref{rem1.1}, Corollary \ref{corollary1} and Corollary \ref{c:1.2}}\label{sec:proof1}

In this section, we prove Theorem \ref{main theorem1}, Remark \ref{rem1.1}, Corollary \ref{corollary1} and Corollary \ref{c:1.2}.

\subsection{Proof of Theorem \ref{main theorem1}}

We prove Theorem \ref{main theorem1} in three steps.

\

\emph{Step 1: proof of inequality \eqref{eq:221201b}} 

\

Denote \begin{equation*}
\begin{split}
\inf\bigg\{&\int_{\{2\psi<-t\}}|f|^2\tilde\rho:f\in \mathcal{O}(\{2\psi<-t\}) 
		\\ &\text{ s.t. } f^{(l)}(z_j)=a_{j,l}\text{ for any $0\le l\le k_j$ and any $1\le j\le m$}\bigg\}
\end{split}
\end{equation*}
by $G(t)$ for $t\ge0$.  Note that $\tilde\rho=e^{-\varphi}c(-2\psi)$ and
$G(0)=M(Z_0,\mathfrak{a},\tilde\rho).$ As $v(dd^c(\varphi+2\psi),z_j)\ge 2(k_j+1)$ for any $1\le j\le m$, it follows from 
Theorem \ref{thm:general_concave} that $G(h^{-1}(r))$ is concave, where $h(t)=\int_t^{+\infty}c(s)e^{-s}ds$. 

By Lemma \ref{l:unique}, there exists a holomorphic function $F_0$ on $D$ such that $f^{(l)}(z_j)=a_{j,l}$ for any $0\le l\le k_j$ and any $1\le j\le m$, and 
$$G(0)=\int_{D}|F_0|^2\tilde\rho.$$
By definition of $G(t)$, we have 
 $$G(-\log r)\le \int_{\{2\psi<\log r\}}|F_0|^2\tilde\rho$$
 for any $r\in(0,1]$,
then combining the concavity of $G(h^{-1}(r))$, we obtain that 
\begin{equation}
	\label{eq:221201c}\frac{\int_{\{z\in D:2\psi(z)\ge\log r\}}|F_0(z)|^2\tilde\rho}{\int_0^{-\log r}c(t)e^{-t}dt}\le \frac{G(0)-G(-\log r)}{\int_0^{-\log r}c(t)e^{-t}dt}\le \frac{G(0)}{\int_0^{+\infty}c(t)e^{-t}dt}<+\infty.
\end{equation}
Since $\lim_{t\rightarrow0+0}c(t)=c(0)=1$ and $\lim_{w\rightarrow z}\varphi(w)=\varphi(z)$ for any $z\in\partial D$, it follows from Lemma  \ref{l:01a} and inequality \eqref{eq:221201c} that $F_0\in H^2(D)$ and 
\begin{equation}
	\label{eq:221201d}\begin{split}
			M_H(Z_0,\mathfrak{a},\rho)&\le \frac{1}{2\pi}\int_{\partial D}|F_0|^2\rho|dz|\\
			&\le\frac{1}{2\pi}\liminf_{r\rightarrow1-0}\frac{\int_{\{z\in D:\psi\ge \log r\}}|F_0|^2\tilde\rho}{1-r}\\
			&=\frac{1}{2\pi}\liminf_{r\rightarrow1-0}\frac{\int_{\{z\in D:2\psi\ge \log r\}}|F_0|^2\tilde\rho}{\int_{0}^{-\log r}c(t)e^{-t}dt}\times\frac{\int_{0}^{-\log r}c(t)e^{-t}dt}{1-r^{\frac{1}{2}}}
			\\&\le \frac{1}{\pi\int_{0}^{+\infty}c(t)e^{-t}dt}M(Z_0,\mathfrak{a},\tilde\rho)
							\end{split}
\end{equation}

Thus, inequality \eqref{eq:221201b} holds.

\

\emph{Step 2: necessity of the characterization}

\

 Assume that the equality
\begin{equation}
	\label{eq:221201e}M_H(Z_0,\mathfrak{a},\rho)=\frac{M(Z_0,\mathfrak{a},\tilde\rho)}{\pi\int_{0}^{+\infty}c(t)e^{-t}dt}
\end{equation}
holds. 
Combining inequality \eqref{eq:221201c} and inequality \eqref{eq:221201d}, we get that
 \begin{equation*}
	\liminf_{r\rightarrow1-0}\frac{\int_{\{z\in D:2\psi(z)\ge\log r\}}|F_0(z)|^2\tilde\rho}{\int_0^{-\log r}c(t)e^{-t}dt}=\liminf_{r\rightarrow1-0}\frac{G(0)-G(-\log r)}{\int_0^{-\log r}c(t)e^{-t}dt}=\frac{G(0)}{\int_0^{+\infty}c(t)e^{-t}dt}.
\end{equation*}
Since $G(h^{-1}(r))$ is concave, we know 
 that $G(h^{-1}(r))$ is linear with respect to $r\in(0,\int_0^{+\infty}c(t)e^{-t}dt)$. By Theorem \ref{thm:m-points}, we get that

	$(1)$ $\varphi+2\psi=2\log|g_1|+ 2\sum_{1\le j\le m}G_{D}(\cdot,z_j)+2u_1$, where $g_1$ is a holomorphic function on $D$ such that $ord_{z_j}(g_1)=\min\{l:a_{j,l}\not=0\}$, and $u_1$ is a harmonic function on $D$;
	 
	$(2)$ $\psi=\sum_{1\le j\le m}p_jG_{D}(\cdot,z_j)$;
	
	$(3)$ $\chi_{-u_1}=\prod_{1\le j\le m}\chi_{z_j}$;
	
	$(4)$  For any $1\le j\le m$, 
	\begin{equation}
		\label{eq:1202a}
		\lim_{z\rightarrow z_j}\frac{g_1p_*\left(f_{u_1}\left(\prod_{1\le j\le m}f_{z_j}\right)\left(\sum_{1\le j\le m}p_j\frac{df_{z_j}}{f_{z_j}}\right)\right)}{\sum_{0\le l\le k_j}a_{j,l}(z-z_j)^{l}dz}=c_0
		\end{equation}
	holds, where $c_0\not=0$ is a constant independent of $j$.

As $v(dd^c(\varphi+2\psi),z_j)\geq2(k_j+1)$, it follows from statements $(1)$ and $(4)$ above that, for any $1\le j\le m$, $ord_{z_j}(g_1)=k_j$ and $a_j=0$ for any $l<k_j$. By Weierstrass theorem (see \cite{OF81}), there exists a holomorphic function $g_0$ on $D$ such that $dg_{0}\not=0$ on $\Omega\backslash Z_0$ and $ord_{z_j}(g)=k_j$ for any $j$. Denote that 
$$g_2:=\frac{g_1}{g_0}\quad\text{and}\quad u_2:=u_1+\log|g_0|-\sum_{1\le j\le m}k_jG_D(\cdot,z_j)$$
on $D$. Note that $u_2$ is harmonic on $D$, and $g_2$ is harmonic on $D$ satisfying $dg_2(z_j)\not=0$ for any $1\le j\le m$.
 Combining statements $(1)$ and $(3)$ above,  we have 
 $$\varphi+2\psi=2\log|g_2|+ 2\sum_{1\le j\le m}(k_j+1)G_{D}(\cdot,z_j)+2u_2$$ 
 and 
 $$\chi_{-u_2}=\prod_{1\le j\le m}\chi_{z_j}^{k_j+1}.$$
 
 In the following, we prove that $g_2\not=0$ on $D$. Then $u:=\log|g_2|+u_2$ is harmonic on $D$ and $\chi_{-u}=\prod_{1\le j\le m}\chi_{z_j}^{k_j+1}$. Combining  equality \eqref{eq:1202a}, $\log|g_1|+u_1=\log|g_0|+\log|g_2|+u_1=u+\sum_{1\le j\le m}k_jG_D(\cdot,z_j)$ and $a_{j,l}=0$ for any $l<k_j$, we have 
 $$\lim_{z\rightarrow z_j}\frac{p_*\left(f_{u}\left(\prod_{1\le j\le m}f^{k_j+1}_{z_j}\right)\left(\sum_{1\le j\le m}p_j\frac{df_{z_j}}{f_{z_j}}\right)\right)}{a_{j,k_j}(z-z_j)^{k_j}dz}=c_0,
$$ thus the necessity of the characterization holds.
 
Denote that $h:=\varphi+2\psi-2\sum_{1\le j\le m}(k_j+1)G_{D}(\cdot,z_j)$ on $\overline D$. Then $h$ is subharmonic on $D$ and $h$ is continuous at $z$ for any $z\in\partial D$. It suffices to prove that $h$ is harmonic on $D$. 
 By solving Dirichlet problem, there is a continuous function $\tilde h$ on $\overline D$, which satisfies that $\tilde h=h$ on $\partial D$ and $\tilde h$ is harmonic on $D$. As $h$ is subharmonic on $D$, we have 
 $$h\le \tilde h$$
  on $\overline D$. Denote that
  $$\tilde\varphi:=\varphi+\tilde h-h.$$ Then we have $\tilde\varphi|_{\partial D}=\varphi|_{\partial D}$ and $\tilde \varphi+2\psi=2\sum_{1\le j\le m}(k_j+1)G_{D}(\cdot,z_j)+\tilde h$. Denote that 
  $$\tilde\rho_1:=e^{-\tilde\varphi}c(-2\psi)$$
  on $\overline D$.
  Note that $\tilde\rho_1\le \tilde\rho$.
  By definition, we have 
 \begin{displaymath}
 	 M(Z_0,\mathfrak{a},\tilde\rho)\ge M(Z_0,\mathfrak{a},\tilde\rho_1).
 \end{displaymath}
Combining equality \eqref{eq:221201e} and inequality \eqref{eq:221201b}, we have 
$$M_H(Z_0,\mathfrak{a},\rho)\le\frac{M(Z_0,\mathfrak{a},\tilde\rho_1)}{\pi\int_{0}^{+\infty}c(t)e^{-t}dt}\le\frac{M(Z_0,\mathfrak{a},\tilde\rho)}{\pi\int_{0}^{+\infty}c(t)e^{-t}dt}=M_H(Z_0,\mathfrak{a},\rho),$$
which shows that 
$$ M(Z_0,\mathfrak{a},\tilde\rho)= M(Z_0,\mathfrak{a},\tilde\rho_1).$$ As $M(Z_0,\mathfrak{a},\tilde\rho)<+\infty$ and $\sum_{1\le j\le m}\sum_{0\le l\le k_j}|a_{j,l}|\not=0$,   we have $\tilde\rho_1=\tilde\rho$, which implies that $2\log|g_2|$ is harmonic on $D$, i.e. $g_2\not=0$ on $D$. 

Thus, the necessity of characterization in Theorem \ref{main theorem1} has been proved.

\

\emph{Step 3:  sufficiency of the characterization}

\

 Assume that the four statements $(1)-(4)$ in Theorem \ref{main theorem1} hold. By Weierstrass theorem (see \cite{OF81}), there exists a holomorphic function $g_0$ on $D$ such that $dg_{0}\not=0$ on $\Omega\backslash Z_0$ and $ord_{z_j}(g)=k_j$ for any $j$. Denote that $\tilde u=u+\sum_{1\le j\le m}k_jG_D(\cdot,z_j)-\log|g_0|$ is a harmonic on $\Omega$. Thus, we have $\varphi+2\psi=2\log|g_0|+\sum_{1\le j\le m}2G_D(\cdot,z_j)+2\tilde u$, $\chi_{-\tilde u}=\chi_{-u}\prod_{1\le j\le m}\chi_{z_j}^{-k_j}=\prod_{1\le j\le m}\chi_{z_j}$ and 
 $$\lim_{z\rightarrow z_j}\frac{g_0p_*\left(f_{\tilde u}\left(\prod_{1\le j\le m}f_{z_j}\right)\left(\sum_{1\le j\le m}p_j\frac{df_{z_j}}{f_{z_j}}\right)\right)}{a_{j,k_j}(z-z_j)^{k_j}dz}=c_0
$$
 for any $j$. Then, by  Theorem \ref{thm:m-points}, we know that $G(h^{-1}(r))$ is linear with respect to $r\in(0,\int_0^{+\infty}c(t)e^{-t}dt)$, where the definition of $G(t)$ comes from Step 1.

Using Corollary \ref{c:linear} and Remark \ref{rem:finite uniform section}, we obtain that 
\begin{equation}
	\label{eq:221205a}
	G(t)=\int_{\{2\psi<-t\}}|F_0|^2\tilde\rho
	\end{equation}
 for any $t\geq0$ and 
 \begin{equation}
 	\label{eq:0317a}
 	\begin{split}
 	 	F_0=&\frac{g_0p_*\left(f_{\tilde u}\left(\prod_{1\le j\le m}f_{z_j}\right)\left(\sum_{1\le j\le m}p_j\frac{df_{z_j}}{f_{z_j}}\right)\right)}{c_0dz}\\
 		=&\frac{p_*\left(f_{ u}\left(\prod_{1\le j\le m}f^{k_j+1}_{z_j}\right)\left(\sum_{1\le j\le m}p_j\frac{df_{z_j}}{f_{z_j}}\right)\right)}{c_0dz},	
 	\end{split}
\end{equation}
 where  $p$ is the universal covering from unit disc $\Delta$ to $ D$, $f_{\tilde u}$ is a holomorphic function on $\Delta$ such that $|f_{\tilde u}|=p^*(e^{\tilde u})$, $f_{z_0}$ is a holomorphic function on $\Delta$ such that $|f_{z_j}|=p^*(e^{G_D(\cdot,z_j)})$ for any $j$, and $f_u=\frac{f_{\tilde u}g_0}{\prod_{1\le j\le m}f_{z_j}^{k_j}}$ satisfies that $|f_u|=p^*\left(e^{u}\right)$. Note that $u=\frac{\varphi}{2}+\psi-\sum_{1\le j\le m}(k_j+1)G_D(\cdot,z_j)$ can be  extended to a continuous function on  $\overline D$, then we know 
$$|F_0|\in C(\overline D).$$

 Let $f\in H^2(D)$ satisfying $f^{(l)}(z_j)=l!a_{j,l}$ for any $1\le l\le k_j$ and any $1\le j\le m$. Note that $(f-F_0,z_j)\in\mathcal{I}(\varphi+2\psi)_{z_j}$ for any $j$, $c(t)e^{-t}$ is decreasing and $\{\psi<-t\}\Subset D$ for any $t>0$, then it follows from $\int_{D}|F_0|^2\tilde\rho<+\infty$ that 
 \begin{displaymath}
 	\begin{split}
 		\int_{\{2\psi<-t\}}|f|^2\tilde\rho&\le2\int_{\{2\psi<-t\}}|f-F_0|^2e^{-\varphi}c(-2\psi)+2\int_{D}|F_0|^2\tilde\rho\\
 		&\le2C\int_{\{2\psi<-t\}}|f-F_0|^2e^{-\varphi-2\psi}+2\int_{D}|F_0|^2\tilde\rho\\
 		&<+\infty
 	\end{split}
 \end{displaymath}
 for any $t>0$.  Following from Lemma \ref{l:unique}, we have 
\begin{displaymath}
\begin{split}
	\int_{\{2\psi<-t\}}|f|^2\tilde\rho=&\int_{\{2\psi<-t\}}|F_0|^2\tilde\rho+\int_{\{2\psi<-t\}}|f-F_0|^2\tilde\rho,
\end{split}
\end{displaymath}
which implies that 
\begin{equation}
	\nonumber \int_{\{2\psi<-t\}}F_0\overline{F_0-f}\tilde\rho=0
\end{equation}
for any $t>0$. 
 It follows from Lemma \ref{l:3} and Lemma \ref{l:4} that there exists $r_1>0$ such that 
\begin{equation}
	\nonumber \int_{\{z\in D:\psi(z)=r\}}F_0\overline{F_0-f} e^{-\varphi}\left(\frac{\partial \psi}{\partial v_z} \right)^{-1}|dz|=0
\end{equation}
holds for any $r\in(-r_1,0)$, which implies that 
\begin{equation}
	\label{eq:1219a}\int_{\{z\in D:\psi(z)=r\}}|f|^2e^{-\varphi}\left(\frac{\partial \psi}{\partial v_z} \right)^{-1}|dz|\ge \int_{\{z\in D:\psi(z)=r\}}|F_0|^2e^{-\varphi}\left(\frac{\partial \psi}{\partial v_z} \right)^{-1}|dz|.
\end{equation}
As $|F_0|\in C(\overline D)$, it follows from
 the dominated convergence theorem,  Lemma \ref{l:0-4v2} and equality \eqref{eq:1219a} that
 \begin{displaymath}
 	\int_{\partial D}|f|^2e^{-\varphi}\left(\frac{\partial \psi}{\partial v_z} \right)^{-1}|dz|\ge \int_{\partial D}|F_0|^2e^{-\varphi}\left(\frac{\partial \psi}{\partial v_z} \right)^{-1}|dz|, \end{displaymath}
 then we have 
 \begin{equation}
 	\label{eq:1219b}
 	M_H(Z_0,\mathfrak{a},\rho)=\frac{1}{2\pi}\int_{\partial D}|F_0|^2e^{-\varphi}\left(\frac{\partial \psi}{\partial v_z} \right)^{-1}|dz|. \end{equation}
 
  Note that $\lim_{t\rightarrow0+0}c(t)=c(0)=1$. It follows from equality \eqref{eq:221205a}, the dominated convergence theorem and Lemma \ref{l:3} that 
 \begin{displaymath}
	\begin{split}
		\frac{M(Z_0,\mathfrak{a},\tilde\rho)}{\int_0^{+\infty}c(t)e^{-t}dt}&=\frac{G(0)}{\int_0^{+\infty}c(t)e^{-t}dt}\\
		&=\lim_{r\rightarrow 1-0}\frac{\int_{\{z\in D:2\psi(z)\ge\log r\}}|F_0|^2\tilde\rho}{\int_0^{-\log r}c(t)e^{-t}dt}\\
		&=\frac{1}{2}\int_{\partial D}|F_0|^2e^{-\varphi}\left(\frac{\partial\psi}{\partial v_z}\right)^{-1}|dz|.
	\end{split}
\end{displaymath}
 Combining equality \eqref{eq:1219b}, we have $$M_H(Z_0,\mathfrak{a},\rho)=\frac{M(Z_0,\mathfrak{a},\tilde\rho)}{\pi\int_0^{+\infty}c(t)e^{-t}dt}.$$
 Thus, Theorem \ref{main theorem1} has been proved.

\subsection{Proof of Remark \ref{rem1.1}}

Remark \ref{rem1.1} holds by equality \eqref{eq:221205a}, \eqref{eq:0317a} and \eqref{eq:1219b} in the proof of Theorem \ref{main theorem1}.

\subsection{Proof of Corollary \ref{corollary1}}
In this section, we prove Corollary \ref{corollary1}.

Denote that $M:=\inf\{\int_{D}|f|^2\lambda:f\in \mathcal{O}(D)$ such that $f^{(l)}(z_j)=0$ for $0\le l<k_j$ and $f^{(k_j)}(z_j)=k_j!a_j$ for any $1\le j\le m\}.$ Following from Theorem \ref{main theorem1} and Theorem \ref{c:L2-1d-char} (Taking $c\equiv1$), we have 
\begin{equation}
	\label{eq:1229b}
	M_H\le \frac{M}{\pi}\le \sum_{1\le j\le m}\frac{2|a_j|^2t_j}{(k_j+1)c_{\beta}(z_j)^{2(k_j+1)}}\lambda(z_j).
\end{equation} 
By Lemma \ref{l:exists of $M_H$}, there exists $f\in H^2(D)$ such that $f^{(l)}(z_j)=0$ for $0\le l<k_j$ and $f^{(k_j)}(z_j)=k_j!a_j$ for any $1\le j\le m$, and
	$$\frac{1}{2\pi}\int_{\partial D}|f|^2\lambda\left(\frac{\partial\psi}{\partial v_z}\right)^{-1}|dz|\le\sum_{1\le j\le m}\frac{2|a_j|^2t_j}{(k_j+1)c_{\beta}(z_j)^{2(k_j+1)}}\lambda(z_j).$$
	
	In the following part, we prove the characterization of the holding of equality $M_H=\sum_{1\le j\le m}\frac{2|a_j|^2t_j}{(k_j+1)c_{\beta}(z_j)^{2(k_j+1)}}\lambda(z_j).$

Firstly, we prove the necessity. 
Assume that $M_H=\sum_{1\le j\le m}\frac{2|a_j|^2t_j}{(k_j+1)c_{\beta}(z_j)^{2(k_j+1)}}\lambda(z_j),$ then by inequality \eqref{eq:1229b}, we have 
$$M_H=\frac{M}{\pi}.$$
Using Theorem \ref{main theorem1}, we know the two statements in Corollary \ref{corollary1} hold.

Secondly, we prove the sufficiency. Assume that the two statements in Corollary \ref{corollary1} hold. Theorem \ref{main theorem1} shows that $M_H=\frac{M}{\pi}$, and Theorem \ref{c:L2-1d-char} shows that $\frac{M}{\pi}=\sum_{1\le j\le m}\frac{2|a_j|^2t_j}{(k_j+1)c_{\beta}(z_j)^{2(k_j+1)}}\lambda(z_j)$. Then we have $M_H=\sum_{1\le j\le m}\frac{2|a_j|^2t_j}{(k_j+1)c_{\beta}(z_j)^{2(k_j+1)}}\lambda(z_j).$

Thus, Corollary \ref{corollary1} holds.

\subsection{Proof of Corollary \ref{c:1.2}}
We prove Corollary \ref{c:1.2} by inductive method. 

If $k=0$, it follows from Corollary \ref{corollary1} that Corollary \ref{c:1.2} holds.

Assume that $k\ge1$ and there is a constant $C_1$, such that for any $\tilde a_{j,l}\in\mathbb{C}$, where  $1\le j\le m$ and $0\le l\le k-1$, there exists $f\in H^2(D)$ such that $f^{(l)}(z_j)=\tilde a_{j,l}$ for any $1\le j\le m$ and $0\le l\le k-1$, and
	$\frac{1}{2\pi}\int_{\partial D}|f|^2|dz|\le C_1\sum_{1\le j\le m}\sum_{0\le l\le k-1}|\tilde a_{j,l}|^2.$
Then there exists $f_1\in H^2(D)$ such that $f_1^{(l)}(z_j)= a_{j,l}$ for any $1\le j\le m$ and $0\le l\le k-1$, and
\begin{equation}
	\label{eq:0206a}
	\frac{1}{2\pi}\int_{\partial D}|f_1|^2|dz|\le C_1\sum_{1\le j\le m}\sum_{0\le l\le k-1}| a_{j,l}|^2.
\end{equation}

Following from Lemma \ref{l:bounded} and inequality \eqref{eq:0206a}, we have 
\begin{equation}
	\label{eq:0207a}\sum_{1\le j\le m}|f_1^{(k)}(z_j)|^2\le \frac{C_2}{2\pi}\int_{\partial D}|f_1|^2|dz|\le C_1C_2\sum_{1\le j\le m}\sum_{0\le l\le k-1}| a_{j,l}|^2.
\end{equation}
According to Corollary \ref{corollary1}, there is $f_2\in H^2(D)$ such that for any $j$, $f_2^{(l)}(z_j)=0$ for $0\le l\le k-1$ and $f_2^{(k)}(z_j)=a_{j,k}-f_1^{(k)}(z_j)$, and 
\begin{equation}
	\label{eq:0207b}
	\frac{1}{2\pi}\int_{\partial D}|f_2|^2|dz|\le C_3\sum_{1\le j\le m}|a_{j,k}-f_1^{(k)}(z_j)|^2,
\end{equation}
where $C_3$ is a constant independent of $a_{j,k}$. Denote that 
$$f:=f_1+f_2,$$ 
then we have $f^{(l)}(z_j)=a_{j,l}$ for $1\le j\le m$ and $0\le l\le k$. Combining inequality \eqref{eq:0206a}, \eqref{eq:0207a} and \eqref{eq:0207b}, we have
\begin{displaymath}
	\begin{split}
		&\frac{1}{2\pi}\int_{\partial D}|f|^2|dz|\\
		\le&\frac{1}{\pi}\int_{\partial D}|f_1|^2|dz|+\frac{1}{\pi}\int_{\partial D}|f_2|^2|dz|\\
		\le&C_1\sum_{1\le j\le m}\sum_{0\le l\le k-1}| a_{j,l}|^2+C_3\sum_{1\le j\le m}|a_{j,k}-f_1^{(k)}(z_j)|^2\\
		\le&C_1\sum_{1\le j\le m}\sum_{0\le l\le k-1}| a_{j,l}|^2+2C_3\sum_{1\le j\le m}|a_{j,k}|^2+2C_3C_1C_2\sum_{1\le j\le m}\sum_{0\le l\le k-1}| a_{j,l}|^2.
	\end{split}
\end{displaymath}
Take $C=\max\{C_1+2C_1C_2C_3,2C_3\}$, thus Corollary \ref{c:1.2} holds by induction.

\section{Proofs of Theorem \ref{thm:2.1}, Theorem \ref{thm:2.2}, Remark \ref{rem2.1}, Corollary \ref{c:2.1} and Corollary \ref{c:2.2}}
\label{sec:proof2}

In this section, we prove Theorem \ref{thm:2.1}, Theorem \ref{thm:2.2}, Remark \ref{rem2.1}, Corollary \ref{c:2.1} and Corollary \ref{c:2.2}.

\subsection{Proof of Theorem \ref{thm:2.1}}
We prove Theorem \ref{thm:2.1} in three steps.

\

\emph{Step 1: proof of inequality \eqref{eq:0212a}} 

\

Denote that $\hat\rho:=\prod_{1\le j\le n}e^{-\varphi_j}$, then we have $-\log\hat\rho$ is plurisubharmonic on $M$ and $\hat\rho(w_j,\hat w_j)\leq \liminf_{w\rightarrow w_j}\hat\rho(w,\hat w_j)$ for any $(w_j,\hat w_j)\in \partial D_j\times M_j\subset \partial M$ and any $1\le j\le n$. 

By Lemma \ref{l:unique}, there exists a holomorphic function $F_0$ on $M$ such that $(F_0-f_0,z_{\beta})\in J_{\beta}$ for any $\beta\in I_1$, and 
$$G(0)=\int_{M}|F_0|^2\tilde\rho.$$
By definition of $G(t)$, we have 
 $$G(-\log r)\le \int_{\{2\psi<\log r\}}|F_0|^2\tilde\rho$$
 for any $r\in(0,1]$,
then combining the concavity of $G(h^{-1}(r))$, we obtain that 
\begin{equation}
	\label{eq:0212b}\frac{\int_{\{z\in M:2\psi(z)\ge\log r\}}|F_0(z)|^2\tilde\rho}{\int_0^{-\log r}c(t)e^{-t}dt}\le \frac{G(0)-G(-\log r)}{\int_0^{-\log r}c(t)e^{-t}dt}\le \frac{G(0)}{\int_0^{+\infty}c(t)e^{-t}dt}<+\infty.
\end{equation}
Since $\lim_{t\rightarrow0+0}c(t)=c(0)=1$ and  $\hat\rho(w_j,\hat w_j)\leq \liminf_{w\rightarrow w_j}\hat\rho(w,\hat w_j)$ for any $(w_j,\hat w_j)\in \partial D_j\times M_j\subset \partial M$ and any $1\le j\le n$, it follows from Proposition  \ref{p:b6} and inequality \eqref{eq:0212b} that then there is $\tilde F_0\in H^2_{\rho}(M,\partial M)$ such that $\tilde F_0^*=F_0$  and 
\begin{equation}
	\label{eq:0212c}\begin{split}
			M_H(Z_0,J,\rho)&\le \|\tilde F_0\|_{\partial M,\rho}^2\\
			&\le\frac{1}{\pi}\liminf_{r\rightarrow1-0}\frac{\int_{\{z\in D:2\psi\ge \log r\}}|F_0|^2\tilde\rho}{1-r}\\
			&=\frac{1}{\pi}\liminf_{r\rightarrow1-0}\frac{\int_{\{z\in D:2\psi\ge \log r\}}|F_0|^2\tilde\rho}{\int_{0}^{-\log r}c(t)e^{-t}dt}\times\frac{\int_{0}^{-\log r}c(t)e^{-t}dt}{1-r}
			\\&\le \frac{M(Z_0,J,\tilde\rho)}{\pi\int_{0}^{+\infty}c(t)e^{-t}dt}
							\end{split}
\end{equation}

This, inequality \eqref{eq:0212a} holds.

\

\emph{Step 2: necessity of the characterization}

\

 Assume that the equality
\begin{equation}
	\nonumber M_H(Z_0,J,\rho)=\frac{M(Z_0,J,\tilde\rho)}{\pi\int_{0}^{+\infty}c(t)e^{-t}dt}
\end{equation}
holds. 
Combining inequality \eqref{eq:0212b} and inequality \eqref{eq:0212c}, we get that
 \begin{equation*}
	\liminf_{r\rightarrow1-0}\frac{\int_{\{z\in M:2\psi(z)\ge\log r\}}|F_0(z)|^2\tilde\rho}{\int_0^{-\log r}c(t)e^{-t}dt}=\liminf_{r\rightarrow1-0}\frac{G(0)-G(-\log r)}{\int_0^{-\log r}c(t)e^{-t}dt}=\frac{G(0)}{\int_0^{+\infty}c(t)e^{-t}dt}
\end{equation*}
and $$M_H(Z_0,J,\rho)= \|\tilde F_0\|_{\partial M,\rho}^2.$$
Since $G(h^{-1}(r))$ is concave, we know 
 that $G(h^{-1}(r))$ is linear with respect to $r\in[0,\int_0^{+\infty}c(t)e^{-t}dt]$.

\

\emph{Step 3:  sufficiency of the characterization}

\

Assume that $G(h^{-1}(r))$ is linear with respect to $r\in[0,\int_0^{+\infty}c(t)e^{-t}dt]$ and
\begin{equation}
	\label{eq:0216a}
	M_H(Z_0,J,\rho)= \|\tilde F_0\|_{\partial M,\rho}^2.
\end{equation}
Using Corollary \ref{c:linear}, we obtain that 
\begin{equation}
	\nonumber
	G(t)=\int_{\{2\psi<-t\}}|F_0|^2\tilde\rho
	\end{equation}
 for any $t\geq0$.
 Thus, inequality \eqref{eq:0212b} becomes an equality, which shows that 
 \begin{equation}
 	\label{eq:0213f}
 	\frac{1}{\pi}\liminf_{r\rightarrow1-0}\frac{\int_{\{z\in D:2\psi\ge \log r\}}|F_0|^2\tilde\rho}{1-r}= \frac{M(Z_0,J,\tilde\rho)}{\pi\int_{0}^{+\infty}c(t)e^{-t}dt}.
 \end{equation}
 
 Note that 
 $$\{z\in M:2\psi(z)=s\}=\cup_{1\le j\le m}\{w_j\in D_j:2\psi_j(w_j)=s\}\times\{\hat w_j\in M_j:2\hat\psi_j(\hat w_j)\le s\},$$
 where $s\in (-\infty,0)$, $\hat\psi_j:=\max_{1\le j'\le m,j'\not=j}\{\sum_{1\le k\le m_j}p_{j',k}G_{D_{j'}}(\cdot,z_{j',k})\}$ on $M_j$ and $\psi_j:=\sum_{1\le k\le m_j}p_{j,k}G_{D_j}(\cdot,z_{j,k})$ on $D_j$. Denote that $$M_{j,s}:=\{\hat w_j\in M_j:2\hat\psi_j(\hat w_j)\le s\}$$
  and $$D_{j,s}:=\{w_j\in D_j:2\psi_j(w_j)\le s\}$$
  for $1\le j\le n$. Following from Lemma \ref{l:4}, there exists $r_0\in(0,1)$ such that $\bigtriangledown\psi_j\not=0$ on $D_j\backslash D_{j,\log r_0}$ for any $1\le j\le n$.
 By Lemma \ref{l:3}, we have
\begin{equation}
\label{eq:0213g} \begin{split}&\int_{\{z\in M:2\psi(z)\ge \log r\}}|F_0|^2\tilde\rho\\
=&\sum_{1\le j\le n}\int_{\{z\in M:2\psi(z)\ge \log r\,\&\,\psi_j(z)>\hat\psi_j(z)\}}|F_0|^2\tilde\rho\\
	=&\sum_{1\le j\le n}
	\int_{\log r}^0\int_{M_{j,s}}\int_{\partial D_{j,s}}\frac{|F_0(w_j,\hat w_j)|^2\tilde\rho}{2|\bigtriangledown\psi_j|}|dw_j|d\mu_j(\hat w_j)ds\\
	=&\sum_{1\le j\le n}
	\int_{\log r}^0\int_{M_{j,s}}\int_{\partial D_{j,s}}\frac{|F_0(w_j,\hat w_j)|^2}{2|\bigtriangledown\psi_j|}\times c(-2\psi)\prod_{1\le l\le n}e^{-\varphi_l} |dw_j|d\mu_j(\hat w_j)ds\\
	=&\sum_{1\le j\le n}
	\int_{\log r}^0c(-s)\int_{M_{j,s}}\int_{\partial D_{j,s}}\frac{|F_0(w_j,\hat w_j)|^2}{2|\bigtriangledown\psi_j|}\times \prod_{1\le l\le n}e^{-\varphi_l} |dw_j|d\mu_j(\hat w_j)ds.
\end{split}
\end{equation}	
 for $r\in(r_0,1)$. By Lemma \ref{l:0-4v3} and $\tilde F_0^*=F_0$, 
 $$\lim_{s\rightarrow0}\sum_{1\le j\le n}\int_{M_{j,s}}\int_{\partial D_{j,s}}\frac{|F_0(w_j,\hat w_j)|^2}{|\bigtriangledown\psi_j|}\times e^{-\varphi} |dw_j|d\mu_j(\hat w_j)ds=2\pi\|\tilde F_0\|_{\partial M,\rho}^2.$$
 As $\lim_{s\rightarrow0+0}c(s)=c(0)=1$, equality \eqref{eq:0213g} implies that 
 \begin{equation}
 	\label{eq:0216b}
 	\frac{1}{\pi}\liminf_{r\rightarrow1-0}\frac{\int_{\{z\in D:2\psi\ge \log r\}}|F_0|^2\tilde\rho}{1-r}=\|\tilde F_0\|_{\partial M,\rho}^2.
 	 \end{equation}
Combining inequality \eqref{eq:0212c}, equality \eqref{eq:0216a}, \eqref{eq:0213f} and \eqref{eq:0216b}, we have $M_H(Z_0,J,\rho)=\frac{M(Z_0,J,\tilde\rho)}{\pi\int_{0}^{+\infty}c(t)e^{-t}dt}$.

 Thus, Theorem \ref{thm:2.1} has been proved.

\subsection{Proof of Theorem \ref{thm:2.2}}

As $\varphi_j$ is continuous at $z$ for any $z\in \partial D_j$, following from Weierstrass theorem (see \cite{OF81}), statement $(1)$ in Theorem \ref{thm:2.2} is equivalent to $\varphi_j=2\log|g_j|+2u_j$ for any $j\in\{1,2,...,n\}$, where $u_j$ is a harmonic function on $D_j$ and $g_j$ is a holomorphic function on $D_j$ satisfying $g_j(z_{j,k})\not=0$ for any $k\in\{1,2,...,m_j\}$. Thus, using Theorem \ref{thm:prod-finite-point} and Remark \ref{r:var=-infty}, the four statements holds if and only if $G(h^{-1}(r))$ is linear. 

We follow the notations in the proof of Theorem \ref{thm:2.1}. By Theorem \ref{thm:2.1}, we know that it suffices to prove that: if $G(h^{-1}(r))$ is linear, then 
\begin{equation}
	\label{eq:0216c}M_H(Z_0,J,\rho)=\|\tilde F_0\|^2_{\partial M,\rho},
\end{equation}
 where $F_0$ is a holomorphic function on $M$ (introduced in the proof of Theorem \ref{thm:2.1}) satisfying that $\tilde F_0^*=F_0$,
 \begin{equation}
 	\label{eq:0317b}M(Z_0,J,\tilde\rho)=\int_M|F_0|^2\tilde\rho
 \end{equation} 
and $(F_0-f_0,z_{\beta})\in\mathcal{I}(2\psi)_{z_\beta}$ for any $\beta\in I_1$.

In the following, assume that $G(h^{-1}(r))$ is linear on $[0,\int_0^{+\infty}c(t)e^{-t}dt]$. Using Corollary \ref{c:linear}, we obtain that 
\begin{equation}
	\label{eq:0216d}
	G(t)=\int_{\{2\psi<-t\}}|F_0|^2\tilde\rho
	\end{equation}
 for any $t\geq0$.

 Let $f$ be any element in $H^2_{\rho}(M,\partial M)$ satisfying that $(f^*-F_0,z_{\beta})\in \mathcal{I}(2\psi)_{z_{\beta}}$ for any $\beta\in I_1$. By Theorem \ref{thm:prod-finite-point}, we know that $\varphi_j=2\log|g_j|+2u_j$, where $u_j$ is a harmonic function on $D_j$ and $g_j$ is a holomorphic function on $D_j$ satisfying $g_{j}(z_{j,k})\not=0$ for $1\le k\le m_j$, thus $\varphi_j$ is bounded near $z_{j,k}$. Note that $c(t)e^{-t}$ is decreasing on $(0,+\infty)$, then $\int_{M}|F_0|^2e^{-\varphi}c(-2\psi)<+\infty$ implies that 
 $$|f^*|^2e^{-\varphi}c(-2\psi)\le C|f^*-F_0|^2e^{-2\psi}+2|F_0|^2e^{-\varphi}c(-2\psi)$$ is integrable near $z_\beta$. For any $z\in M\backslash\{z_\beta:\beta\in I_1\}$, as $c(-2\psi)$ is bounded near $z$, it follows from Lemma \ref{l:integral} that $|f^*|^2e^{-\varphi}c(-2\psi)$ is integrable near $z$. Thus, we obtain that 
 \begin{equation}\nonumber
 	\int_{\{2\psi<-t\}}|f^*|^2e^{-\varphi}c(-2\psi)<+\infty \end{equation} 
 holds for any $t>0$. 
 By equality \eqref{eq:0216d}, we have 
 $$\int_{\{2\psi<-t\}}(f^*-F_0)\overline{F_0}\tilde\rho=0,$$
 which implies
 that $$\sum_{1\le j\le n}
	\int_{-\infty}^tc(-s)\int_{M_{j,s}}\int_{\partial D_{j,s}}\frac{(f^*-F_0)\overline{F_0}}{|\bigtriangledown\psi_j|}\times \prod_{1\le l\le n}e^{-\varphi_l} |dw_j|d\mu_j(\hat w_j)ds=0$$
 for any $t>0$ according to Lemma \ref{l:3}, where the definitions of $M_{j,s}$ and $D_{j,s}$ can be seen in the proof of Theorem \ref{thm:2.1}. Thus, 
$$\sum_{1\le j\le n}\int_{M_{j,s}}\int_{\partial D_{j,s}}\frac{(f^*-F_0)\overline{F_0}}{|\bigtriangledown\psi_j|}\times \prod_{1\le l\le n}e^{-\varphi_l} |dw_j|d\mu_j(\hat w_j)=0,$$
 which shows that 
 \begin{equation}
 	\label{eq:0215b}
 	\sum_{1\le j\le n}\int_{M_{j,s}}\int_{\partial D_{j,s}}\frac{|f^*|^2e^{-\varphi}}{|\bigtriangledown\psi_j|}|dw_j|d\mu_j(\hat w_j)\geq\sum_{1\le j\le n}\int_{M_{j,s}}\int_{\partial D_{j,s}}\frac{|F_0|^2e^{-\varphi}}{|\bigtriangledown\psi_j|} |dw_j|d\mu_j(\hat w_j)
 \end{equation}
for any $s>0$.
Combining equality \eqref{eq:0215b} and Lemma \ref{l:0-4v3}, we know that 
$$\|f^*\|_{\partial M,\rho}\ge \|\tilde F_0\|_{\partial M,\rho},$$
which implies that equality \eqref{eq:0216c} holds.

Thus, Theorem \ref{thm:2.2} has been proved.

\subsection{Proof of Remark \ref{rem2.1}}

When the four statements in Theorem \ref{thm:2.2} hold, $G(h^{-1}(r))$ is linear. Then
Remark \ref{rem2.1} holds by Remark \ref{r:1.2}, equality \eqref{eq:0216c} and \eqref{eq:0317b}.

\subsection{Proof of Corollary \ref{c:2.1}}

Following from Theorem \ref{thm:2.1} and Theorem \ref{thm:prod-finite-jet} (Taking $c\equiv1$), we have 
\begin{equation}
	\label{eq:0217c}
	M_H(Z_0,\mathcal{I}(2\psi),\rho)\le \frac{M(Z_0,\mathcal{I}(2\psi),\rho)}{\pi}\le \sum_{\beta\in I_1}\sum_{\alpha\in E_{\beta}}\frac{|d_{\beta,\alpha}|^22^n\pi^{n-1}e^{-\varphi(z_{\beta})}}{\Pi_{1\le j\le n}(\alpha_j+1)c_{j,\beta_j}^{2\alpha_{j}+2}}.
\end{equation} 
By Lemma \ref{l:exists of $M_H$2}, there exists $f\in H^2_{\rho}(M,\partial M)$ such that $(f^*-f_0,z_{\beta})\in \mathcal{I}(2\psi)_{z_\beta}$ for any $\beta\in I_1$, and
	$$\|f\|_{\partial M,\rho}^2\le\sum_{\beta\in I_1}\sum_{\alpha\in E_{\beta}}\frac{|d_{\beta,\alpha}|^22^n\pi^{n-1}e^{-\varphi(z_{\beta})}}{\Pi_{1\le j\le n}(\alpha_j+1)c_{j,\beta_j}^{2\alpha_{j}+2}}.$$
	
	In the following part, we prove the characterization of the holding of equality 
	\begin{equation}
		\label{eq:0217b}M_H(Z_0,\mathcal{I}(2\psi),\rho)=\sum_{1\le j\le m}\frac{2|a_j|^2t_j}{(k_j+1)c_{\beta}(z_j)^{2(k_j+1)}}\rho(z_j).	\end{equation}

Firstly, we prove the necessity. 
Assume that equality \eqref{eq:0217b} holds, then by inequality \eqref{eq:0217c}, we have 
$$M_H(Z_0,\mathcal{I}(2\psi),\rho)=\frac{M(Z_0,\mathcal{I}(2\psi),\rho)}{\pi}.$$
Using Theorem \ref{thm:2.2}, we know the four statements in Corollary \ref{c:2.1} hold.

Secondly, we prove the sufficiency. Assume that the four statements in Corollary \ref{c:2.1} hold. Theorem \ref{thm:2.2} shows that $M_H(Z_0,\mathcal{I}(2\psi),\rho)=\frac{M(Z_0,\mathcal{I}(2\psi),\rho)}{\pi}$, and Theorem \ref{thm:prod-finite-jet} shows that $\frac{M(Z_0,\mathcal{I}(2\psi),\rho)}{\pi}=\sum_{\beta\in I_1}\sum_{\alpha\in E_{\beta}}\frac{|d_{\beta,\alpha}|^22^n\pi^{n-1}e^{-\varphi(z_{\beta})}}{\Pi_{1\le j\le n}(\alpha_j+1)c_{j,\beta_j}^{2\alpha_{j}+2}}$,  then equality \eqref{eq:0217b} holds.

Thus, Corollary \ref{c:2.1} holds.

\subsection{Proof of Corollary \ref{c:2.2}}

We prove Corollary \ref{c:2.2} by inductive method. 

If $k=0$, it follows from Corollary \ref{c:2.1} that Corollary \ref{c:2.2} holds.

Assume that $k\ge1$ and there is a constant $C_1$, such that for any $\tilde a_{\beta,\alpha}\in\mathbb{C}$, where  $\beta\in I_1$ and $\alpha\in L_{k-1}$, there exists $f\in H^2_{\rho}(M,\partial M)$ such that $\partial^{\alpha}f^*(z_{\beta})=a_{\beta,\alpha}$ for any $\beta\in I_1$ and $\alpha\in L_{k-1}$, and
	$\|f\|_{\partial M,\rho}^2\le C_1\sum_{\beta\in I_1,\alpha\in L_{k-1}}|\tilde a_{\beta,\alpha}|^2.$
Then there exists $f_1\in H^2_{\rho}(M,\partial M)$ such that $\partial^{\alpha}f_1^*(z_{\beta})=a_{\beta,\alpha}$ for any $\beta\in I_1$ and $\alpha\in L_{k-1}$, and
\begin{equation}
	\label{eq:0218a}
	\|f_1^*\|_{\partial M,\rho}^2\le C_1\sum_{\beta\in I_1,\alpha\in L_{k-1}}|a_{\beta,\alpha}|^2.
\end{equation}

Following from Lemma \ref{l:b1-p} and inequality \eqref{eq:0218a}, we have 
\begin{equation}
	\label{eq:0218b}\sum_{\beta\in I_1|\alpha|=k}|\partial^{\alpha} f_1^*(z_\beta)|^2\le C_2\|f_1^*\|_{\partial M,\rho}^2\le C_1C_2\sum_{\beta\in I_1,\alpha\in L_{k-1}}|a_{\beta,\alpha}|^2,
\end{equation}
where $|\alpha|=\sum_{1\le j\le n}\alpha_j$.
It follows from Corollary \ref{c:2.1} (taking $f_0=\sum_{|\alpha|=k}(a_{\beta,\alpha}-\partial^{\alpha} f_1^{*}(z_{\beta}))\prod_{1\le j\le n}(w_j-z_{j,\beta_j})^{\alpha_j}$ on $V_{\beta}$ and $\psi=2(n+k)\max\{\sum_{1\le k\le m_j}G_{D_j}(\cdot,z_{j,k})\}$), that there is $f_2\in H_{\rho}^2(M,\partial M)$ such that for any $\beta\in I_1$, $\partial^{\alpha}f_2^*(z_\beta)=0$ for $\alpha\in L_{k-1}$ and $\partial^{\alpha}f_2^*(z_\beta)=a_{\beta,\alpha}-\partial^{\alpha} f_1^{*}(z_{\beta})$, and 
\begin{equation}
	\label{eq:0218c}
	\frac{1}{2\pi}\int_{\partial D}|f_2|^2|dz|\le C_3\sum_{\beta\in I_1,|\alpha|=k}|a_{\beta,\alpha}-\partial^{\alpha} f_1^{*}(z_{\beta})|^2,
\end{equation}
where $C_3$ is a constant independent of $a_{\beta,\alpha}$. Denote that 
$$f:=f_1+f_2,$$ 
then we have $\partial^{\alpha}f^*(z_\beta)=a_{\beta,\alpha}$ for $\beta\in I_1$ and $\alpha\in L_k$. Combining inequality \eqref{eq:0218a}, \eqref{eq:0218b} and \eqref{eq:0218c}, we have
\begin{displaymath}
	\begin{split}
		&\|f\|_{\partial M,\rho}^2\\
		\le&2\|f_1\|_{\partial M,\rho}^2+2\|f_2\|_{\partial M,\rho}^2\\
		\le&C_1\sum_{\beta\in I_1,\alpha\in L_{k-1}}|a_{\beta,\alpha}|^2+C_3\sum_{\beta\in I_1,|\alpha|=k}|a_{\beta,\alpha}-\partial^{\alpha} f_1^{*}(z_{\beta})|^2\\
		\le&C_1\sum_{\beta\in I_1,\alpha\in L_{k-1}}|a_{\beta,\alpha}|^2+2C_3\sum_{\beta\in I_1,|\alpha|=k}|a_{\beta,\alpha}|^2+2C_3C_1C_2\sum_{\beta\in I_1,\alpha\in L_{k-1}}|a_{\beta,\alpha}|^2.
	\end{split}
\end{displaymath}
Take $C=\max\{C_1+2C_1C_2C_3,2C_3\}$, thus Corollary \ref{c:2.1} holds by induction.

\section{Proof of Theorem \ref{thm:3.1}}

We prove Theorem \ref{thm:3.1} in three steps: Firstly, we prove inequality \eqref{eq:0312a}; Secondly, we prove the necessity of the characterization; Finally, we prove the sufficiency of the characterization.

\emph{Step 1.}
By Lemma \ref{l:exists of $M_H$2}, there is a unique $F_0\in H^2_{\rho}(M,\partial M)$ satisfying that $F_0^*(z_{\beta})=h_0(z_{\beta})$ for any $\beta\in I_1$ and $M_H(Z_0,J,\rho)=\|F_0\|_{\partial M,\rho}^2$. For any $1\le j\le n$, denote that 
\begin{displaymath}
	\begin{split}
		M_{H,j}(Z_0,J,\rho):=\inf\bigg\{&\|f\|^2_{\partial D_j\times M_j,\rho}:f\in H_{\rho}^2(M,\partial D_j\times M_j) 
		\\ &\text{ s.t. } f^*(z_{\beta})=h_0(z_{\beta})\text{ for any $\beta\in I_1$}\bigg\}.\end{split}
\end{displaymath}
By definitions of $M_H(Z_0,J,\rho)$ and $M_{H,j}(Z_0,J,\rho)$, we have 
\begin{equation}
	\label{eq:0314a}
	M_H(Z_0,J,\rho)=\|F_0\|_{\partial M,\rho}^2\ge\sum_{1\le j\le n}M_{H,j}(Z_0,J,\rho).
\end{equation}
For any $1\le j\le n$, denote that
\begin{displaymath}
	\begin{split}
		M_{D_j}:=\inf\bigg\{&\int_{D_j}|f|^2e^{-\varphi_j}:f\in \mathcal{O}(D_j)
		\\ &\text{ s.t. } f(z_{j,k})=h_j(z_{j,k})\text{ for any $1\le k\le m_j$}\bigg\},\end{split}
\end{displaymath}
\begin{displaymath}
	\begin{split}
		M_{\partial D_j}:=\inf\Bigg\{\frac{1}{2\pi}\int_{\partial D_j}&|f|^2\bigg(\sum_{1\le k\le m_j}2\frac{\partial G_{D_j}(w_j,z_{j,k})}{\partial v_{w_j}}\bigg)^{-1}e^{-\varphi_j}|dw_j| :
		\\ &f\in H^2(D_j) \text{ s.t. } f(z_{j,k})=h_j(z_{j,k})\text{ for any $1\le k\le m_j$}\Bigg\}\end{split}
\end{displaymath}
and 
\begin{displaymath}
	\begin{split}
		M_{M_j}:=\inf\bigg\{&\int_{M_j}|f|^2\prod_{1\le l\le n,l\not=j}e^{-\varphi_l} :f\in \mathcal{O}(M_j) 
		\\ &\text{ s.t. } f(z_{\gamma})=\prod_{1\le l\le n,l\not=j}h_l(z_{l,\gamma_l}) \text{ for any $\gamma\in I_{1,j}$}\bigg\},\end{split}
\end{displaymath}
where $I_{1,j}:=\{\gamma=(\gamma_1,\ldots,\gamma_{j-1},\gamma_{j+1},\ldots,\gamma_n)\in\mathbb{Z}^{n-1} :1\le\gamma_l\le m_l$ for any $l\not=j\}$ and  $z_{\gamma}:=(z_{1,\gamma_1},\ldots,z_{j-1,\gamma_{j-1}},z_{j+1,\gamma_{j+1}},\ldots,z_{n,\gamma_{n}})\in M_j$ for any $\gamma\in I_{1,j}$.
It follows from Theorem \ref{main theorem1}, Lemma \ref{l:pro1}, Lemma \ref{l:pro2}, Lemma \ref{l:pro3} and inequality \eqref{eq:0314a}, that
\begin{equation}
	\label{eq:0314b}\begin{split}
		M_H(Z_0,J,\rho)\ge&\sum_{1\le j\le n}M_{H,j}(Z_0,J,\rho)\\
		=&\sum_{1\le j\le n}M_{\partial D_j}\times M_{M_j}\\
		=&\sum_{1\le j\le m}M_{\partial D_j}\times \prod_{1\le l\le n,l\not=j}M_{D_l}\\
		\ge& \pi^{n-1}\sum_{1\le j\le m}M_{\partial D_j}\times \prod_{1\le l\le n,l\not=j}M_{\partial D_l}\\
		=&n\pi^{n-1}M_S(Z_0,J,\lambda).
	\end{split}
\end{equation}

\

\emph{Step 2.} Assume that equality $M_S(Z_0,J,\lambda)=\frac{ M_H(Z_0,J,\rho)}{n\pi^{n-1}}$ holds. As there exists $k\in\{1,\ldots,m_j\}$ such that $h_j(z_{j,k})\not=0$, we know that $M_{\partial D_j}>0$ and $M_{D_j}>0$. Following from inequality \eqref{eq:0314b} and Theorem \ref{main theorem1}, we get that 
$$M_{D_j}=\pi M_{\partial D_j}$$
for any $1\le j\le n$, and then the three statements in Theorem \ref{thm:3.1} hold.

\

\emph{Step 3.} Assume that the three statements in Theorem \ref{thm:3.1} hold. Theorem \ref{main theorem1} tells us that 
\begin{equation}
	\label{eq:0314c}M_{D_j}=\pi M_{\partial D_j}
\end{equation}
holds for any $1\le j\le n$. For any $1\le j\le n$, denote that 
$$F_j=\frac{P_j^*\left(f_{u_j}\left(\prod_{1\le k\le m_j}f_{z_{j,k}}\right)\left(\sum_{1\le k\le m_j}\frac{df_{z_{j,k}}}{f_{z_{j,k}}}\right)\right)}{c_jdz}\in \mathcal{O}(D_j).$$
 Following from Remark \ref{rem1.1}, there exists $f_j\in H^2(D_j)$ such that $f_j^*=F_j$, and 
 we have 
 \begin{equation}
 	\label{eq:0314d}M_{D_j}=\int_{D_j}|F_j|^2e^{-\varphi_j}
 	 \end{equation}
and 
\begin{equation}
	\label{eq:0314e}M_{\partial D_j}=\frac{1}{2\pi}\int_{\partial D_j}|f_j|^2\bigg(\sum_{1\le k\le m_j}2\frac{\partial G_{D_j}(w_j,z_{j,k})}{\partial v_{w_j}}\bigg)^{-1}e^{-\varphi_j}|dw_j|.
\end{equation}
Then there exists $\tilde F_0\in H^2_{\rho}(M,\partial M)$ such that $\tilde F_0=f_j\times\prod_{1\le l\le n,l\not=j}F_l$ on $\partial D_j\times M_j$ for any $1\le j\le n$, and $\tilde F_0^*=\prod_{1\le j\le n}F_j$. By Lemma \ref{l:pro2}, Lemma \ref{l:pro3}, equality \eqref{eq:0314d} and \eqref{eq:0314e}, we know that 
\begin{equation}
	\label{eq:0314f}M_{H,j}(Z_0,J,\rho)=\|\tilde F_0\|^2_{\partial D_j\times M_j,\rho}.
\end{equation}
Note that $F_j(z_{j,k})=h_j(z_{j,k})$ for any $1\le j\le n$ and $1\le k\le m_j$, hence $\tilde F_0^*(z_{\beta})=h_0(z_{\beta})$ for any $\beta\in I_1$. Inequality \eqref{eq:0314f} implies that 
$$\sum_{1\le j\le n}M_{H,j}(Z_0,J,\rho)=\sum_{1\le j\le n}\|\tilde F_0\|^2_{\partial D_j\times M_j,\rho}\ge M_H(Z_0,J,\rho).$$
Combining inequality \eqref{eq:0314a}, we have 
\begin{equation}
	\label{eq:0314g}\sum_{1\le j\le n}M_{H,j}(Z_0,J,\rho)= M_H(Z_0,J,\rho).
\end{equation}
Using inequality \eqref{eq:0314b}, equality \eqref{eq:0314c} and \eqref{eq:0314g}, we get that 
$$M_H(Z_0,J,\rho)=n\pi^{n-1}M_S(Z_0,J,\lambda).$$

Thus, Theorem \ref{thm:3.1} holds.

\

\vspace{.1in} {\em Acknowledgements}. The authors would like to thank  Dr. Shijie Bao and Dr. Zhitong Mi for checking the manuscript and  pointing out some typos. The first named author was supported by National Key R\&D Program of China 2021YFA1003100, NSFC-11825101, NSFC-11522101 and NSFC-11431013.

\bibliographystyle{references}

\end{document}